%% file: article2.tex
\documentclass[12pt]{amsart}

\usepackage{amsmath}
\usepackage{amsfonts}
\usepackage{latexsym}
\usepackage{graphicx}
\usepackage{amssymb}
\usepackage{amsthm}
\usepackage[margin=2cm]{geometry}
\usepackage[hyphens]{url}
\usepackage{color}
\usepackage{enumerate}
\usepackage[shortlabels]{enumitem} 
\usepackage[small]{caption}
\usepackage{cite}       
\usepackage{pgfplots}
\usepackage{longtable}
\usepackage{stmaryrd}

\usepackage{array}           

\usepackage{algorithm}     
\usepackage[noend]{algpseudocode}
\algrenewcommand\algorithmicrequire{\textbf{Precondition:}}
\algrenewcommand\algorithmicensure{\textbf{Postcondition:}}

\usepackage{hyperref}

\usetikzlibrary{patterns}

\usepackage[T1]{fontenc}     
\usepackage{lmodern}         
\usepackage[utf8]{inputenc}  

\newtheorem{definition}{Definition}
\newtheorem{lemma}[definition]{Lemma}
\newtheorem{proposition}[definition]{Proposition}

\newtheorem{corollary}[definition]{Corollary}
\newtheorem{conjecture}[definition]{Conjecture}
\newtheorem{theorem}[definition]{Theorem}
\newtheorem{remark}[definition]{Remark}

\newcommand{\N}{\mathbb{N}}
\newcommand{\Z}{\mathbb{Z}}

\newcommand{\R}{\mathbb{R}}

\newcommand{\A}{\mathcal{A}}
\newcommand{\B}{\mathcal{B}}

\newcommand{\F}{\mathcal{F}}

\renewcommand{\L}{\mathcal{L}}
\newcommand{\Scal}{\mathcal{S}}
\newcommand{\T}{\mathcal{T}}
\newcommand{\U}{\mathcal{U}}

\newcommand{\X}{\mathcal{X}}
\newcommand{\Xcal}{\mathcal{X}}

\newcommand{\ba}{\mathbf{a}}
\newcommand{\be}{\mathbf{e}}
\newcommand{\bh}{\mathbf{h}}
\newcommand{\bk}{\mathbf{k}}
\newcommand{\bn}{\mathbf{n}}
\newcommand{\bm}{\mathbf{m}}
\newcommand{\bp}{\mathbf{p}}
\newcommand{\bq}{\mathbf{q}}

\newcommand{\bx}{\mathbf{x}}
\newcommand{\shape}{\textsc{shape}}
\newcommand{\GL}{\textrm{GL}}
\newcommand{\SFT}{\textsc{SFT}}
\newcommand{\zero}{\mathbf{0}}
\newcommand{\sctop}{\textsc{top}}
\newcommand{\scbottom}{\textsc{bottom}}
\newcommand{\scleft}{\textsc{left}}
\newcommand{\scright}{\textsc{right}}

\newcommand\tile[4]{
    \raisebox{-3mm}{
\begin{tikzpicture}[scale=0.9]
\draw (0, 0) -- (1, 0);
\draw (0, 0) -- (0, 1);
\draw (1, 1) -- (1, 0);
\draw (1, 1) -- (0, 1);
\node[rotate=0,font=\tiny] at (0.8, 0.5) {#1};
\node[rotate=0,font=\tiny] at (0.5, 0.8) {#2};
\node[rotate=0,font=\tiny] at (0.2, 0.5) {#3};
\node[rotate=0,font=\tiny] at (0.5, 0.2) {#4};
\end{tikzpicture}}}

\newcommand\tilerot[4]{
    \raisebox{-3mm}{
\begin{tikzpicture}[scale=0.9]
\draw (0, 0) -- (1, 0);
\draw (0, 0) -- (0, 1);
\draw (1, 1) -- (1, 0);
\draw (1, 1) -- (0, 1);
\node[rotate=90,font=\tiny] at (0.8, 0.5) {#1};
\node[rotate=0, font=\tiny] at (0.5, 0.8) {#2};
\node[rotate=90,font=\tiny] at (0.2, 0.5) {#3};
\node[rotate=0, font=\tiny] at (0.5, 0.2) {#4};
\end{tikzpicture}}}

\newcommand\dominoV[3]{
\begin{tikzpicture}[scale=.2]
\draw[fill=#1] (0,0) rectangle (1,1);
\draw[fill=#2] (0,1) rectangle (1,2);
\foreach \x/\y in {#3}
{\draw[thick] (\x,\y) -- (\x+1,\y+1);
 \draw[thick] (\x,\y+1) -- (\x+1,\y);}
\end{tikzpicture}}
\newcommand\dominoH[3]{
\begin{tikzpicture}[scale=.2]
\draw[fill=#1] (0,0) rectangle (1,1);
\draw[fill=#2] (1,0) rectangle (2,1);
\foreach \x/\y in {#3}
{\draw[thick] (\x,\y) -- (\x+1,\y+1);
 \draw[thick] (\x,\y+1) -- (\x+1,\y);}
\end{tikzpicture}}


\begin{document}

\title{Substitutive structure of Jeandel-Rao aperiodic tilings}
\author[S.~Labb\'e]{S\'ebastien Labb\'e}
\address[S.~Labb\'e]{Univ. Bordeaux, CNRS,  Bordeaux INP, LaBRI, UMR 5800, F-33400, Talence, France}
\email{sebastien.labbe@labri.fr}

\date{\today}
\begin{abstract}
Jeandel and Rao proved that 11 is the size of the smallest set of Wang tiles, i.e., unit squares with colored edges, that admit valid tilings (contiguous edges of adjacent tiles have the same color) of the plane, none of them being invariant under a nontrivial translation. We study herein the Wang shift $\Omega_0$ made of all valid tilings using the set $\mathcal{T}_0$ of 11 aperiodic Wang tiles discovered by Jeandel and Rao.  We show that there exists a minimal subshift $X_0$ of $\Omega_0$ such that every tiling in $X_0$ can be decomposed uniquely into 19 distinct patches of sizes ranging from 45 to 112 that are equivalent to a set of 19 self-similar and aperiodic Wang tiles.  We suggest that this provides an almost complete description of the substitutive structure of Jeandel-Rao tilings, as we believe that $\Omega_0\setminus X_0$ is a null set for any shift-invariant probability measure on $\Omega_0$.  The proof is based on 12 elementary steps, 10 of which involve the same procedure allowing one to desubstitute Wang tilings from the existence of a subset of marker tiles. The 2 other steps involve the addition of decorations to deal with fault lines and changing the base of the $\mathbb{Z}^2$-action through a shear conjugacy.  Algorithms are provided to find markers, recognizable substitutions, and shear conjugacy from a set of Wang tiles.
\end{abstract}

\maketitle


\setcounter{tocdepth}{1}
\tableofcontents

\section{Introduction}

Aperiodic tilings are much studied for their beauty but also for their links
with various aspects of science including dynamical systems \cite{MR1452190}, 
topology \cite{MR2446623}, number theory \cite{MR2742574},
theoretical computer science \cite{2017_jeandel_undecidability}
and crystallography.
Chapters 10 and 11 of \cite{MR857454} and the more recent
book on aperiodic order \cite{MR3136260} give an excellent overview of what is
known on aperiodic tilings.
The first examples of aperiodic tilings were
tilings of $\Z^2$ by 
\emph{Wang tiles}, that is, unit square tiles with colored edges
\cite{MR2939561,MR0286317,MR0297572,MR1417578,MR1417576,MR2507046}.
A tiling by Wang tiles is \emph{valid} if every contiguous edges have the same color.
The set of all valid tilings $\Z^2\to\T$ using translated copies of a finite set $\T$ of Wang tiles
is called the \emph{Wang shift} of $\T$ and denoted $\Omega_\T$ (rotated or reflected copies of tiles are not allowed).
It is a $2$-dimensional \emph{subshift} as it
is invariant under translations called \emph{shifts} and closed under taking limits.
A nonempty Wang shift $\Omega_\T$ is said to be \emph{aperiodic} if none of its
tilings have a nontrivial period.

\input{article2_macro_jeandel_rao.tex}
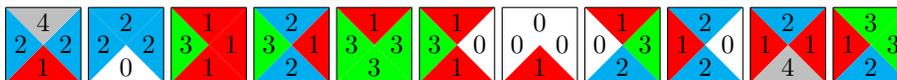
\begin{figure}[h]
\begin{center}
    \input{article2_T0_tiles_no_id.tex}
\end{center}
    \caption{The Jeandel-Rao's set $\T_0$ of 11 Wang tiles.}
    \label{fig:jeandel_rao_tile_set}
\end{figure}

Jeandel and Rao~\cite{jeandel_aperiodic_2015}
proved that the set $\T_0$ of Wang tiles shown in Figure~\ref{fig:jeandel_rao_tile_set}
is the smallest possible set of Wang tiles that is aperiodic.
Indeed, based on computer explorations, they proved
that every set of Wang tiles of cardinality $\leq 10$ either admits a
periodic tiling of the plane or does not tile the plane at all. Thus there is
no aperiodic set of Wang tiles of cardinality less than or equal to 10. In the same
work, they found this interesting candidate of cardinality 11 and they proved
that the associated Wang shift $\Omega_0=\Omega_{\T_0}$ is aperiodic.
Their proof is based on the representation of a set of Wang tiles as a
transducer and on the
description of a sequence of
transducers describing larger and larger infinite horizontal strips of Wang
tiles by iteratively taking products of transducers.

Their example is also minimal for the number of colors.
Indeed it is known that three colors are not enough to allow an aperiodic
set of tiles \cite{chen_decidability_2012} and
Jeandel and Rao mentioned in their
preprint that 
while they used 5 colors for the edges $\{0,1,2,3,4\}$,
they claim colors 0 and 4 can be merged without losing the aperiodic property.
In this work, we keep colors 0 and 4 as distinct and we associate the number 4
to the color gray in figures.

In this contribution, we show the existence of a minimal
aperiodic subshift $X_0$ of the Jeandel-Rao tilings $\Omega_0$.
The substitutive structure of $X_0$ can be described in terms
of a minimal, aperiodic and self-similar Wang shift $\Omega_\U$
based on a set of 19 Wang tiles $\U$ introduced in
\cite{MR3978536}.
Our main result can be summarized as follows.
\begin{theorem}\label{thm:main-result-in-simple-terms}
    There exists an aperiodic and minimal subshift $X_0$  of the Jeandel-Rao
    tilings $\Omega_0$ such that any tiling in $X_0$ can be decomposed uniquely
    into 19 distinct patches (two of size 45, four of size 70, six of size 72
    and seven of size 112) that are equivalent to the set of 19 Wang
    tiles $\U$.
\end{theorem}

We show that $\Omega_0\setminus X_0\neq\varnothing$ in
Proposition~\ref{prop:nonempty} due to the presence of some horizontal fault
lines in~$\Omega_0$, but we believe that $X_0$ gives an almost complete
description of the Jeandel-Rao tilings. 
More precisely and as opposed to the minimal subshift of 
the Kari-Culik tilings \cite{MR3668002},
we think the following holds.

\begin{conjecture}\label{conj:complement-measure-zero-Omega0}
    $\Omega_0\setminus X_0$ is of measure zero
    for any shift-invariant probability measure on $\Omega_0$.
\end{conjecture}

\begin{figure}[h]
\begin{center}
\includegraphics[width=.75\linewidth]{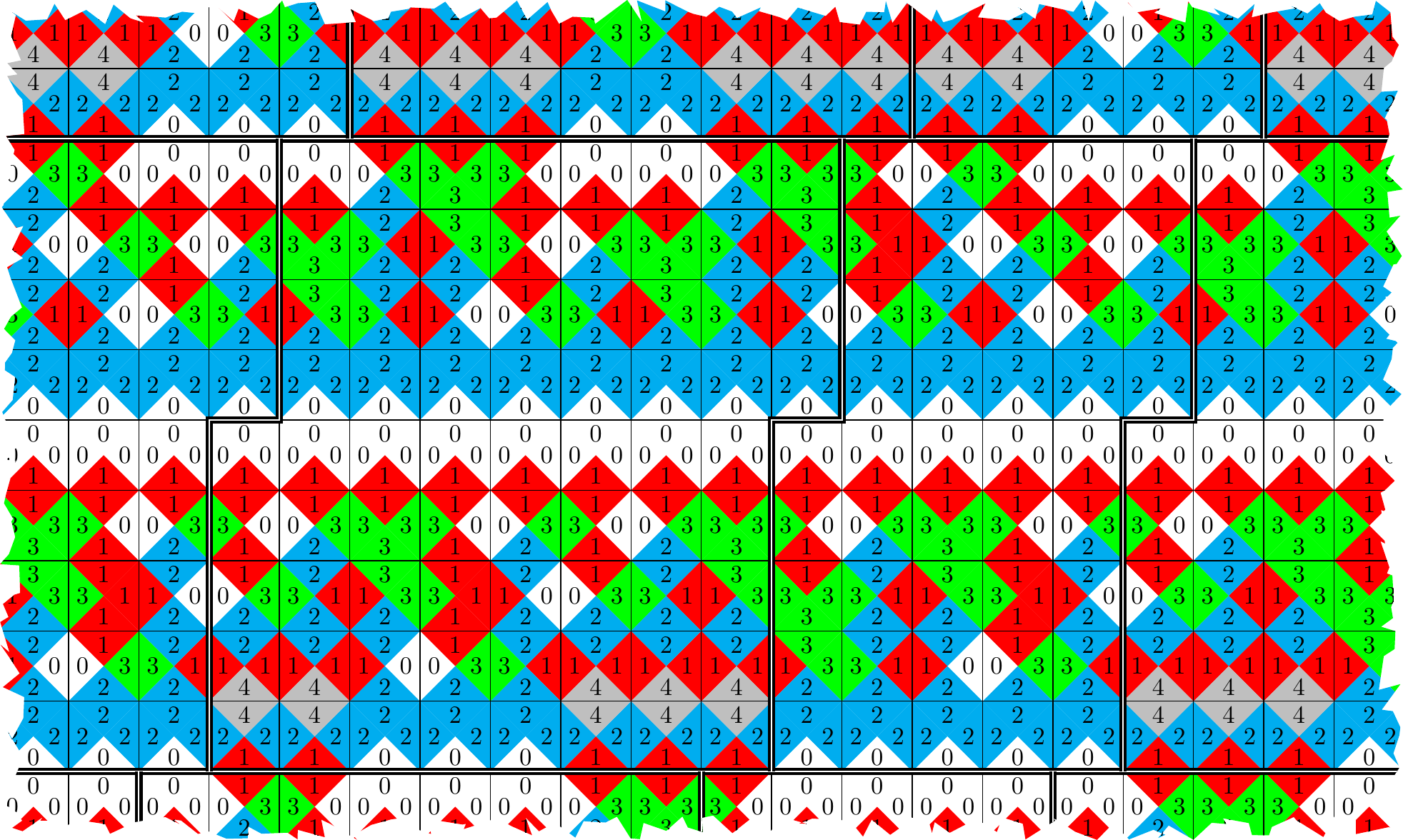}
\end{center}
\caption{A finite part of a Jeandel-Rao aperiodic tiling in $\Omega_0$. Any
    tiling in the minimal subshift $X_0$ of $\Omega_0$ that we describe can be
    decomposed uniquely into 19 supertiles (two of size 45, 
    four of size 70,
    six of size 72
    and seven of size 112). The thick black lines show the
    contour of the supertiles. The figure illustrates two complete supertiles
    of size 72 and 45, respectively. While Jeandel-Rao set of tiles $\T_0$ is not self-similar,
    the 19 supertiles are self-similar.  
    }
\label{fig:jeandel-rao-some-tiling}
\end{figure}

The decomposition of a tiling in $X_0\subset\Omega_0$ into 19 distinct patches
is illustrated in Figure~\ref{fig:jeandel-rao-some-tiling}.
All of the 19 patches and their link with the Wang tiles in $\U$ are shown in
Figure~\ref{fig:omega-12-to-0} in the appendix.
The 19 patches are complicated objects, but we provide a systematic construction
of them by means of a product of substitutions where each substitution 
(or $2$-dimensional morphism)
represents an elementary step in the procedure. The substitutions
are computed automatically using the same algorithms at each step.
There are a dozens of such elementary steps between Jeandel-Rao tilings
$X_0\subset\Omega_0$ and $\Omega_\U$ that are illustrated in
Figure~\ref{fig:subs-structure}.

\begin{figure}[h]
\begin{center}
\begin{tikzpicture}[auto,scale=1.2,every node/.style={scale=0.9}]
    \node (0) at (0,2) {$\Omega_0$};
    \node (1) at (1,2) {$\Omega_1$};
    \node (2) at (2,2) {$\Omega_2$};
    \node (3) at (3,2) {$\Omega_3$};
    \node (4) at (4,2) {$\Omega_4$};
    \node (X0) at (0,1) {$X_0$};
    \node (X1) at (1,1) {$X_1$};
    \node (X2) at (2,1) {$X_2$};
    \node (X3) at (3,1) {$X_3$};
    \node (X4) at (4,1) {$X_4$};
    \node (5) at (5,1) {$\Omega_5$};
    \node (6) at (6,1) {$\Omega_6$};
    \node (7) at (7,1) {$\Omega_7$};
    \node (8) at (8,1) {$\Omega_8$};
    \node (9) at (9,1) {$\Omega_9$};
    \node (10) at (10,1) {$\Omega_{10}$};
    \node (11) at (11,1) {$\Omega_{11}$};
    \node (12) at (12,1) {$\Omega_{12}$};
    \node (13) at (13,1) {$\Omega_\U$};

    \draw (X0) edge[draw=none] node[auto=false,sloped] {$\subset$} (0);
    \draw (X1) edge[draw=none] node[auto=false,sloped] {$\supset$} (1);
    \draw (X2) edge[draw=none] node[auto=false,sloped] {$\supset$} (2);
    \draw (X3) edge[draw=none] node[auto=false,sloped] {$\supset$} (3);
    \draw (X4) edge[draw=none] node[auto=false,sloped] {$\supset$} (4);
    \draw[to-,very thick] (X0) to node {$\omega_0$} (X1);
    \draw[to-,very thick] (X1) to node {$\omega_1$} (X2);
    \draw[to-,very thick] (X2) to node {$\omega_2$} (X3);
    \draw[to-,very thick] (X3) to node {$\omega_3$} (X4);
    \draw[to-,very thick] (0) to node {$\omega_0$} (1);
    \draw[to-,very thick] (1) to node {$\omega_1$} (2);
    \draw[to-,very thick] (2) to node {$\omega_2$} (3);
    \draw[to-,very thick] (3) to node {$\omega_3$} (4);
    \draw[to-,very thick] (X4) to node {$\jmath$} (5);
    \draw[to-,very thick] (5) to node {$\eta$} (6);
    \draw[to-,very thick] (6) to node {$\omega_6$} (7);
    \draw[to-,very thick] (7) to node {$\omega_7$} (8);
    \draw[to-,very thick] (8) to node {$\omega_8$} (9);
    \draw[to-,very thick] (9) to node {$\omega_9$} (10);
    \draw[to-,very thick] (10) to node {$\omega_{10}$} (11);
    \draw[to-,very thick] (11) to node {$\omega_{11}$} (12);
    \draw[to-,very thick] (12) to node {$\rho$} (13);
    \draw[to-,very thick] (13) to [loop above] node {$\omega_\U$} (13);
\end{tikzpicture}
\end{center}
    \caption{Substitutive structure of Jeandel-Rao aperiodic Wang shift
    $\Omega_0$ and its minimal subshift $X_0$ leading to the self-similar
    aperiodic and minimal Wang shift $\Omega_\U$ introduced in
    \cite{MR3978536}.}
    \label{fig:subs-structure}
\end{figure}

The substitutive structure of Jeandel-Rao tilings
is given by 
a sequence of sets of Wang tiles $\{\T_{i}\}_{1\leq i\leq 12}$
together with their associated Wang shifts $\{\Omega_{i}\}_{1\leq i\leq 12}$.
The cardinality of the involved sets of tiles is in the table below.
\[
\begin{array}{l|cccccccccccccc}
\text{Set of tiles} & \T_0 &\T_1 &\T_2 &\T_3 &\T_4 &\T_5 &\T_6 &\T_7 
    &\T_8 &\T_9 &\T_{10} &\T_{11} &\T_{12} & \U\\
\hline
\text{Cardinality} & 11 & 13 & 20 & 24 & 28 & 29 & 29 & 20 & 20 & 22 & 18 & 21
    & 19 & 19
\end{array}
\]
The morphism $\omega_0\,\omega_1\,\omega_2\,\omega_3 :\Omega_4\to\Omega_0$ is
shown in Figure~\ref{fig:omega-4-to-0}.
The map $\eta\,\omega_6\,\omega_7\,\omega_8\,
\omega_9\,\omega_{10}\,\omega_{11}\,\rho:\Omega_\U\to\Omega_5$ is shown in
Figure~\ref{fig:omega-12-to-5}.
The map $\omega_0\,\omega_1\,\omega_2\,\omega_3\,\jmath\,
\eta\,\omega_6\,\omega_7\,\omega_8\,
\omega_9\,\omega_{10}\,\omega_{11}\,\rho:\Omega_\U\to\Omega_0$
is the map which is shown in Figure~\ref{fig:omega-12-to-0}.
The map $\jmath:\Omega_5\to\Omega_4$ is not onto but its image 
$X_4=\jmath(\Omega_5)$ has an important role in the comprehension of
Jeandel-Rao aperiodic tilings as it leads to the minimal subshift $X_0$.
The fact that each tiling in $X_0$ can be decomposed uniquely
into this procedure follows from the fact that each morphism $\omega_i$
with $0\leq i\leq 3$
or $6\leq i\leq 11$
is recognizable and onto up to a shift,
and both $\jmath:\Omega_5\to X_4$ and
$\eta:\Omega_6\to\Omega_5$ are one-to-one and onto.
The term \emph{recognizable morphism} essentially means that the morphism is
one-to-one up to a shift and its precise definition can be found in 
Section~\ref{sec:recognizability}.

The construction of the morphisms $\omega_i$ is inspired from a well-known method
to study self-similar aperiodic tilings.
One way to prove that a Wang shift is aperiodic is to use the \emph{unique
composition property} \cite{MR1637896} also known as
\emph{composition-decomposition method} \cite{MR3136260}
(see Proposition~\ref{prop:expansive-recognizable-aperiodic}).
That method was used in \cite{MR3978536}
to show that $\Omega_\U$ is self-similar, minimal and aperiodic
based on the notion of marker tiles and recognizability.
The reader may refer to Section~\ref{sec:self-similar}
for the definition of self-similarity.

The same method can be used in the context of Wang shifts that are not
self-similar. The idea is to prove that a Wang shift is similar to another one
which is known to be aperiodic (see
Lemma~\ref{lem:aperiodic-implies-aperiodic}).
This reminds of \cite{MR4015135}
where the authors study the recognizability for sequences of morphisms
in the theory of $S$-adic systems on $\Z$ \cite{MR3330561}.
Note that applying a sequence of recognizable substitutions in the context of
hierarchical tilings of $\R^d$ was also considered in
\cite{MR3226791,frank_introduction_2018}.

Due to the fact that we have many steps to perform to understand Jeandel-Rao tilings, we
improve the method used in \cite{MR3978536} 
which,
based on the existence of marker tiles $M\subset\T$
among a set of Wang tiles $\T$, proved the existence of 
a set of Wang tiles $\Scal$ and a
recognizable $2$-dimensional morphism $\Omega_\Scal\to\Omega_\T$ that is onto up to
a shift. 
The improvement is that we provide two algorithms.
Algorithm~\ref{alg:find-markers} 
finds a set of marker tiles $M\subset\T$ (if it exists)
from a set of Wang
tiles $\T$ and a given surrounding radius to consider (the surrounding radius
input is necessary since otherwise it is undecidable).
Algorithm~\ref{alg:find-recognizable-sub-from-markers}
computes the set of Wang tiles $\Scal$ and the recognizable
$2$-dimensional morphism $\Omega_\Scal\to\Omega_\T$ from a set of markers $M\subset\T$.
The morphism is of the form
$\square\mapsto\square,\square\mapsto\dominoV{none}{none}{}$
or of the form $\square\mapsto\square,\square\mapsto\dominoH{none}{none}{}$
mapping each tile from $\Scal$ on a tile in $\T$ or on a domino of two tiles in
$\T$.

We use these two algorithms to describe the substitutive structure of the Wang
tilings $\Omega_0$ made with the tiles from Jeandel-Rao set of tiles $\T_0$.
It turns out that Jeandel-Rao tiles has a subset of markers which allow
to desubstitute uniquely any Jeandel-Rao tiling.  
In other words, the two algorithms prove the existence of a set of Wang tiles
$\T_1$, its Wang shift $\Omega_1=\Omega_{\T_1}$ and
a recognizable $2$-dimensional
morphism $\omega_0:\Omega_1\to\Omega_0$ that is onto up to a shift.  
We reapply the same method
again and again and again on the resulting tiles 
to get sets of Wang tiles $\T_2$, $\T_3$ and $\T_4$.

The two algorithms stop to work for the set of Wang tiles $\T_4$ since it has
no markers.
The first reason is that $\Omega_4$ has fault lines as it is the case for Robinson tilings
\cite{MR0297572}. These fault lines are also present in $\Omega_0$ as there
exists tilings in $\Omega_0$ with a biinfinite horizontal sequence of tiles all
of them with color 0 on the bottom edge.
Such a horizontal fault line of 0's can be seen in 
    Figure~\ref{fig:jeandel-rao-some-tiling}. Translating to the left or to the
    right the upper-half tiling above the fault line preserves the validity of 
    the Wang tiling but breaks the 19 patterns obtained from the substitutive
    structure.
    This implies that not all tilings in $\Omega_0$ and in $\Omega_4$ are the
    image of tilings in $\Omega_\U$ by a sequence of substitutions. 
We deal with this issue as done in \cite{GAHLER2012627}
by creating a new set of tiles $\T_5$ which adds decorations on tiles $\T_4$
and a map $\jmath:\T_5\to\T_4$ which removes the decorations.
We prove that $\jmath$ is one-to-one on tilings $\Omega_5$.
The advantage is that the Wang shift $\Omega_5=\Omega_{\T_5}$ has no fault
lines.

The set of Wang tiles $\T_5$ has no markers neither and the reason is that tilings in $\Omega_5$ admits diagonal markers instead of marker tiles appearing on rows or columns.
So we introduce a new set of tiles $\T_6$ such that tilings in $\Omega_6$
are tilings in $\Omega_5$ sheared by the action of the matrix
$\left(\begin{smallmatrix}
        1 & 1 \\
        0 & 1
\end{smallmatrix}\right)^{-1}$.

The set of Wang tiles $\T_6$ has markers, thus we can reapply the two algorithms
again. We get a new set of Wang tiles $\T_7$ and a recognizable morphism
$\Omega_7\to\Omega_6$ that is onto up to a shift.
We apply this method five more times to get the set of Wang tiles $\T_{12}$ and its
Wang shift $\Omega_{12}$.
We prove that the Wang shift $\Omega_{12}$ is equivalent to
$\Omega_\U$ where $\U$ is a set of 19 Wang tiles which was
introduced in \cite{MR3978536}.
It was proved that the Wang shift $\Omega_\U$
is self-similar, aperiodic and minimal
as there exists an expansive and primitive morphism
$\omega_\U:\Omega_\U\to\Omega_\U$ that is
recognizable and onto up to a shift.

The statement of the main result of this contribution on the substitutive
structure of aperiodic Jeandel-Rao Wang tilings is below
in which we use the following notation
\begin{equation*}
    {\overline{X}}^{\sigma} 
    = \bigcup_{\bk\in\Z^2}\sigma^\bk X
    = \bigcup_{\bk\in\Z^2}\{\sigma^\bk(x) \mid x\in X\}
\end{equation*}
for the closure of a set $X\subset\A^{\Z^2}$ under the shift $\sigma$.

\begin{theorem}\label{thm:introduction}
    Let $\Omega_0$ be the Jeandel-Rao Wang shift.
There exist sets of Wang tiles $\{\T_{i}\}_{1\leq i\leq 12}$
together with their associated Wang shifts $\{\Omega_{i}\}_{1\leq i\leq 12}$
that provide the substitutive structure of Jeandel-Rao tilings.
More precisely,
\begin{enumerate}[\rm (i)]
    \item
    There exists a sequence of recognizable $2$-dimensional morphisms:
    \begin{equation*}
        \Omega_0 \xleftarrow{\omega_0}
        \Omega_1 \xleftarrow{\omega_1}
        \Omega_2 \xleftarrow{\omega_2}
        \Omega_3 \xleftarrow{\omega_3}
        \Omega_4
    \end{equation*}
    that are onto up to a shift, i.e.,
    $\overline{\omega_i(\Omega_{i+1})}^\sigma=\Omega_{i}$
    for each $i\in\{0,1,2,3\}$.
\item There exists an embedding $\jmath:\Omega_5\to\Omega_4$ which is a
    topological conjugacy onto its image.
\item 
There exists a shear conjugacy
$\eta:\Omega_6\to\Omega_5$
which shears Wang tilings by the action of the matrix
$\left(\begin{smallmatrix}
        1 & 1 \\
        0 & 1
\end{smallmatrix}\right)$.

    \item
    There exists a sequence of recognizable $2$-dimensional morphisms:
    \begin{equation*}
        \Omega_6 \xleftarrow{\omega_6}
        \Omega_7 \xleftarrow{\omega_7}
        \Omega_8 \xleftarrow{\omega_8}
        \Omega_9 \xleftarrow{\omega_9}
        \Omega_{10} \xleftarrow{\omega_{10}}
        \Omega_{11} \xleftarrow{\omega_{11}}
        \Omega_{12} 
    \end{equation*}
    that are onto up to a shift, i.e.,
    $\overline{\omega_i(\Omega_{i+1})}^\sigma=\Omega_{i}$
    for each $i\in\{6,7,8,9,10,11\}$.

\item The Wang shift $\Omega_{12}$ is equivalent to $\Omega_\U$, thus is
    self-similar, aperiodic and minimal.
\end{enumerate}
\end{theorem}

This result has the following consequences.

\begin{corollary}\label{cor:aperiodic-minimal-5-12}
    $\Omega_i$ is aperiodic and minimal
    for every $i$ with $5\leq i\leq 12$.
\end{corollary}

Let $X_{4}=\jmath(\Omega_{5})$ be the image of the embedding $\jmath$,
as well as
$X_3=\overline{\omega_{3}(X_{4})}^\sigma$,
$X_2=\overline{\omega_{2}(X_{3})}^\sigma$,
$X_1=\overline{\omega_{1}(X_{2})}^\sigma$ and
$X_0=\overline{\omega_{0}(X_{1})}^\sigma$.

\begin{corollary}\label{cor:aperiodic-minimal-0-4}
    $X_i\subseteq\Omega_i$ is an aperiodic and minimal subshift of $\Omega_i$
    for every $i$ with $0\leq i\leq 4$.
\end{corollary}

\subsection*{Structure of the article}
After the introduction and the preliminaries,
the article is structured into four parts.
In Part~\ref{part:1},
we recall the definition of markers, the
desubstitution of Wang shifts and we propose algorithms to compute them
and
we desubstitute Jeandel-Rao tilings from $\Omega_0$ to $\Omega_4$.
In Part~\ref{part:2},
we construct the set of tiles $\T_5$ by adding decorations to tiles $\T_4$ to avoid
fault lines in tilings of $\Omega_4$, we prove that
$\jmath:\Omega_5\to\Omega_4$ is an embedding and
we construct the shear conjugacy $\eta:\Omega_6\to\Omega_5$.
In Part~\ref{part:3}
we desubstitute tilings from $\Omega_6$ to $\Omega_{12}$
and
we prove the main results.
In Part~\ref{part:4},
we examine the set $\Omega_4\setminus X_4$
and show that Jeandel-Rao tilings are not minimal.


\subsection*{Code}

All computations based on SageMath \cite{sagemathv8.9} with \texttt{slabbe}
optional package \cite{labbe_slabbe_0_6_2019} 
associated to this contribution are gathered in a SageMath Jupyter
notebook \cite{labbe_ipynb_arxiv_1808_07768_v4}.

\subsection*{Acknowledgments}

I want to thank Vincent Delecroix for many helpful discussions at LaBRI in
Bordeaux during the preparation of this article including
some hints on how to prove Proposition~\ref{prop:has-measure-0}.
I am thankful to Michaël Rao for providing me the first proof of
Proposition~\ref{prop:rao-impossible-52}.
This contribution would not have been possible without
the Gurobi linear program solver \cite{gurobi} which was very helpful in
solving many instances of tiling problems in seconds (instead of minutes or
hours) but it turns out that Knuth's dancing links algorithm
\cite{knuth_dancing_2000} performs well to find markers.

I acknowledge financial support from the Laboratoire International
Franco-Québécois de Recherche en Combinatoire (LIRCO), the 
Agence Nationale de la Recherche Agence Nationale de la Recherche through the
project CODYS (ANR-18-CE40-0007) and the Horizon  2020  European  Research
Infrastructure  project  OpenDreamKit (676541).

I am very thankful to the anonymous referees for their in-depth reading and
valuable comments from which I learned and leaded to a great improvement of the
presentation.
The revision of the article have benefited from the support of
Erwin Schrödinger International Institute
for Mathematics and Physics
during my stay at the conference Numeration 2019 (Vienna).

%

\section{Preliminaries on Wang tilings and $d$-dimensional words}
\label{sec:preliminaries}

In this section, we introduce 
subshifts, shifts of finite type,
Wang tiles,
fusion of Wang tiles,
the transducer representation of Wang tiles,
$d$-dimensional words, morphisms, and languages.
It is mostly the same as the preliminary section of
\cite{MR3978536}.

We denote by $\Z=\{\dots,-1,0,1,2,\dots\}$ the integers and
by $\N=\{0,1,2,\dots\}$ the nonnegative integers.

\subsection{Subshifts and shifts of finite type}


We follow the notations of \cite{MR1861953}.
Let $\A$ be a finite set, $d\geq 1$, and let $\A^{\Z^d}$ be the set of all maps
$x:\Z^d\to\A$, furnished with the 
product topology of the discrete topology on $\A$.

We write a typical point $x\in\A^{\Z^d}$ as $x=(x_\bm)=(x_\bm:\bm\in\Z^d)$,
where $x_\bm\in\A$ denotes the value of $x$ at $\bm$. 
The topology $\A^{\Z^d}$ is compatible with the metric $\delta$ defined for all
$x,x'\in\A^{\Z^d}$ by $\delta(x,x')=2^{-\min\left\{|\bn|\,:\,
x_\bn\neq x'_\bn\right\}}$.
The \emph{shift action} $\sigma:\bn\mapsto
\sigma^\bn$ of $\Z^d$ on $\A^{\Z^d}$ is defined by
\begin{equation}\label{eq:shift-action}
    (\sigma^\bn(x))_\bm = x_{\bm+\bn}
\end{equation}
for every $x=(x_\bm)\in\A^{\Z^d}$ and $\bn\in\Z^d$. A subset $X\subset
\A^{\Z^d}$ is \emph{shift-invariant} if $\sigma^\bn(X)=X$ for every
$\bn\in\Z^d$, and a closed, shift-invariant subset $X\subset\A^{\Z^d}$ is a
\emph{subshift}. 
If $X\subset\A^{\Z^d}$ is a subshift we write
$\sigma$ for the restriction of the shift-action
\eqref{eq:shift-action} to $X$. 
If $X\subset\A^{\Z^d}$ is a subshift it will sometimes be helpful to specify the
shift-action of $\Z^d$ explicitly and to write $(X,\sigma)$ instead of $X$.
A subshift $(X,\sigma)$ is called \emph{minimal} if $X$ does not
contain any nonempty, proper, closed shift-invariant subset.

For any subset $S\subset\Z^d$ we denote by $\pi_S:\A^{\Z^d}\to\A^S$ the
projection map which restricts every $x\in\A^{\Z^d}$ to $S$. 
A \emph{pattern} is a function $p:S\to\A$ for some finite subset
$S\subset\Z^d$.
A subshift $X\subset\A^{\Z^d}$ is a 
\emph{shift of finite type} (SFT) if there exists a finite set $\F$
of \emph{forbidden} patterns such that
\begin{equation}\label{eq:SFT}
    X = \{x\in\A^{\Z^d} \mid \pi_S\circ\sigma^\bn(x)\notin\F
    \text{ for every } \bn\in\Z^d \text{ and } S\subset\Z^d\}.
\end{equation}
In this case, we write $X=\SFT(\F)$.

We adapt the definition of conjugacy of dynamical systems from
\cite[p. 185]{MR1369092} to subshifts.
Let $X\subset\A^{\Z^d}$ and $Y\subset\B^{\Z^d}$ be two subshifts.
A \emph{homomorphism} $\theta:(X,\sigma)\to(Y,\sigma)$ is a continuous function
$\theta:X\to Y$ satisfying the commuting property
that $\sigma^\bk\circ\theta=\theta\circ\sigma^\bk$ for every $\bk\in\Z^d$.
A homomorphism is called an \emph{embedding}
if it is one-to-one, a \emph{factor map} if it is onto, and a \emph{topological
conjugacy} if it is both one-to-one and onto and its inverse map is continuous.
Two subshifts are \emph{topologically conjugate} if there is a topological
conjugacy between them.

If $X\subset\A^{\Z^d}$ and $Y\subset\B^{\Z^d}$ are two subshifts,
we say that a \emph{$\GL_d(\Z)$-homomorphism} $\theta:(X,\sigma)\to(Y,\sigma)$ is a continuous function
$\theta:X\to Y$ satisfying the commuting property
that 
$\sigma^{M\bk}\circ\theta=\theta\circ\sigma^\bk$ 
for every $\bk\in\Z^d$ for some matrix $M\in \GL_d(\Z)$.
A $\GL_d(\Z)$-homomorphism is called an \emph{$\GL_d(\Z)$-conjugacy} if it is
both one-to-one and onto and its inverse map is continuous.
A $\GL_d(\Z)$-conjugacy is called an \emph{shear conjugacy}
if the matrix $M\in\GL_d(\Z)$ is a shear matrix.
The notion of $\GL_d(\Z)$-conjugacy corresponds to \emph{flip-conjugacy} when
$d=1$ 
\cite{MR1613140}
and to \emph{extended symmetry} when $X=Y$
\cite{MR3721878}.

\subsection{Wang tiles}

A \emph{Wang tile} $\tau=\tile{$a$}{$b$}{$c$}{$d$}$ is a unit square with colored edges
formally represented as a tuple of four colors $(a,b,c,d)\in I\times J\times
I\times J$
where $I$, $J$ are
two finite sets (the vertical and horizontal colors respectively). 
For each Wang tile $\tau=(a,b,c,d)$, we denote by
$\scright(\tau)=a$,
$\sctop(\tau)=b$,
$\scleft(\tau)=c$,
$\scbottom(\tau)=d$
the colors of the right, top, left, and bottom edges of $\tau$
\cite{wang_proving_1961,MR0297572}.

Let $\T$ be a set of Wang tiles. A \emph{valid tiling of $\Z^2$ by $\T$} is an assignment $x$
of tiles to each position of $\Z^2$ so that contiguous edges have the same
color, that is, it is a function $x:\Z^2\to\T$ satisfying
\begin{align}
    \scright\circ x(\bn)&=\scleft\circ x(\bn+\be_1)\label{eq:validwangtiling1}\\
    \sctop\circ x(\bn)&=\scbottom\circ x(\bn+\be_2)\label{eq:validwangtiling2}
\end{align}
for every $\bn\in\Z^2$
where $\be_1=(1,0)$ and $\be_2=(0,1)$.
We denote by $\Omega_\T\subset\T^{\Z^2}$ the set of all valid Wang tilings of $\Z^2$
by $\T$ and we call it the \emph{Wang shift} of $\T$. It is a SFT of
the form \eqref{eq:SFT} since there exists a finite set $\F$ of forbidden patterns
made of all horizontal and vertical dominoes of two tiles that do not share an
edge of the same color.

A set of Wang tiles $\T$ of cardinality $n\in\N$ will be written as a set
$\{\tau_0,\tau_1,\dots\tau_{n-1}\}$. Using the obvious bijection between the
sets $\{0,1,\dots, n-1\}$ and $\T$, any tiling $x\in\Omega_\T$ can be seen as a
map $\Z^2\to\{0,1,\dots, n-1\}$ and vice versa.

A set of Wang tiles $\T$ \emph{tiles} the plane if $\Omega_\T\neq\varnothing$
and \emph{does not tile} the plane if $\Omega_\T=\varnothing$.
A tiling $x\in\Omega_\T$ is \emph{periodic} if there is a nonzero period
$\bn\in\Z^2\setminus\{(0,0)\}$ such that $x=\sigma^\bn(x)$
and otherwise it is said \emph{nonperiodic}.
A set of Wang tiles $\T$ is \emph{periodic} if there is a tiling
$x\in\Omega_\T$ which is periodic. A set of Wang tiles $\T$ is \emph{aperiodic} if
$\Omega_\T\neq\varnothing$ and every tiling $x\in\Omega_\T$ is nonperiodic.
As explained in the first page of \cite{MR0297572} 
(see also \cite[Prop. 5.9]{MR3136260}),
if $\T$ is periodic, then there is a tiling $x$ by $\T$ with two linearly
independent translation vectors (in particular a tiling $x$ with vertical
and horizontal translation vectors).

We say that two sets of Wang tiles $\T$ and $\Scal$ are
\emph{equivalent} if there exist two bijections $i:I\to I'$ $j:J\to J'$ such
that
\begin{equation*}
    \Scal = \{(i(a),j(b),i(c),j(d)) \mid (a,b,c,d) \in \T\}.
\end{equation*}


\subsection{Fusion of Wang tiles}

Now, we introduce a fusion operation on Wang tiles
that share an edge in a tiling according to 
Equations~\ref{eq:validwangtiling1} and~\ref{eq:validwangtiling2}.
Suppose that colors $\{A,B,C,D,X,Y,W,Z\}$ belongs to a 
magma $(\mathcal{C},\cdot)$.
We define two binary operations on Wang tiles
$u=\tile{Y}{B}{X}{A}$ and
$v=\tile{Z}{D}{W}{C}$
as
\begin{equation*}
u\boxbar v=\tile{Z}{B$\cdot$D}{X}{A$\cdot$C}
    \quad
    \text{ if }
    Y = W
    \qquad
    \text{ and }
    \qquad
u\boxminus v=
    \raisebox{-3mm}{
\begin{tikzpicture}[scale=0.9]
\draw (0, 0) -- (1, 0);
\draw (0, 0) -- (0, 1);
\draw (1, 1) -- (1, 0);
\draw (1, 1) -- (0, 1);
\node[rotate=90,font=\tiny] at (0.8, 0.5) {Y$\cdot$Z};
\node[rotate=0,font=\tiny] at (0.5, 0.8) {D};
\node[rotate=90,font=\tiny] at (0.2, 0.5) {X$\cdot$W};
\node[rotate=0,font=\tiny] at (0.5, 0.2) {A};
\end{tikzpicture}}
    \quad \text{ if } B = C.
\end{equation*}
If $Y\neq W$, we say that $u\boxbar v$ is \emph{not defined}.
Similarly, if $B\neq C$, we say that $u\boxminus v$ is \emph{not defined}.
In this contribution, we may assume that the operation $\cdot$ is
associative so we always denote it implicitly by concatenation of colors.

In what follows, we propose algorithms and results that works for both
operations $\boxbar$ and $\boxminus$. 
It is thus desirable to have a common notation to denote both, so we define
\[
    u \boxslash^1 v = u \boxbar v
    \qquad
    \text{ and }
    \qquad
    u \boxslash^2 v = u \boxminus v.
\]
If $u\boxslash^i v$ is defined for some $i\in\{1,2\}$, it means that tiles $u$
and $v$ can appear at position $\bn$ and $\bn+\be_i$ in a tiling for some
$\bn\in\Z^d$. 
For each $i\in\{1,2\}$, one can define a new set of tiles from 
two set of Wang tiles $\T$ and $\Scal$ as
\begin{equation*}
    \T\boxslash^i\Scal = \{u\boxslash^i v \text{ defined }\mid
                        u\in\T,v\in\Scal\}.
\end{equation*}

\subsection{Transducer representation of Wang tiles}

A transducer $M$ is a labeled directed graph whose nodes are called
\emph{states} and edges are called \emph{transitions}. The transitions are
labeled by pairs $a|b$ of letters. The first letter $a$ is the input symbol and
the second letter $b$ is the output symbol. There is no initial nor final state.
A transducer $M$ computes a relation $\rho(M)$ between bi-infinite sequences of
letters. 

As observed in \cite{MR1417578} and extensively used in
\cite{jeandel_aperiodic_2015}, any finite set of Wang tiles may be interpreted
as a transducer. To a given set of tiles $\T$, the states of the corresponding
transducer $M_\T$ are the colors of the vertical edges. The colors of
horizontal edges are the input and output symbols. There is a transition from
state $s$ to state $t$ with label $a|b$ if and only if there is a tile
$(t,b,s,a)\in\T$ whose left, right, bottom and top edges are colored by $s$,
$t$, $a$ and $b$, respectively:
\begin{equation*}
    M_\T = \left\{ s\xrightarrow{a|b} t : 
           \tile{$t$}{$b$}{$s$}{$a$}=(t,b,s,a) \in \T \right\}.
\end{equation*}
The Jeandel-Rao set of Wang tiles defined in
Figure~\ref{fig:jeandel_rao_tile_set} can be seen as a transducer with 4
states and 11 transitions consisting of two connected components $T_0$ and
$T_1$ (see Figure~\ref{fig:transducer}).
\begin{figure}[h]
\begin{center}
\begin{tikzpicture}[scale=2,>=latex,auto]
\node at (-1,0)  {$T_0:$};
\node at (4,0)  {$T_1:$};
\node[rectangle,draw,thick] (node_2) at (5,0)  {$2$};
\node[rectangle,draw,thick] (node_0) at (0,1)  {$0$};
\node[rectangle,draw,thick] (node_1) at (2,1)  {$1$};
\node[rectangle,draw,thick] (node_3) at (1,0)  {$3$};
\draw[->] (node_0) edge[loop left] node {$1|0$} (node_0);
\draw[->] (node_1) edge[loop right] node {$4|2$} (node_1);
\draw[->] (node_2) edge[loop above] node {$0|2,1|4$} (node_2);
\draw[->] (node_3) edge[loop below] node {$3|1$} (node_3);
\draw [->,bend left=10]  (node_0) to node {$2|1$} (node_3);
\draw [->,bend right=10]  (node_1) to node[swap] {$2|3$} (node_3);
\draw [->,bend right=10] (node_1) to node[swap] {$2|2$} (node_0);
\draw [->,bend left=10]  (node_3) to node {$1|1$} (node_0);
\draw [->,bend right=10]  (node_3) to node[swap] {$1|1,2|2$} (node_1);
\end{tikzpicture}
\end{center}
\caption{The Jeandel-Rao set of tiles $\T_0$ seen as a transducer $M_{\T_0}$.
    Each tile of $\T_0$ corresponds to a transition.
    For example, the first tile in Figure~\ref{fig:jeandel_rao_tile_set}
    is associated to the transition $2\xrightarrow{1|4}2$, etc.
    The transducer $M_{\T_0}$
    consists of two connected components identified as $T_0$ and $T_1$ 
    in \cite{jeandel_aperiodic_2015}.}
\label{fig:transducer}
\end{figure}
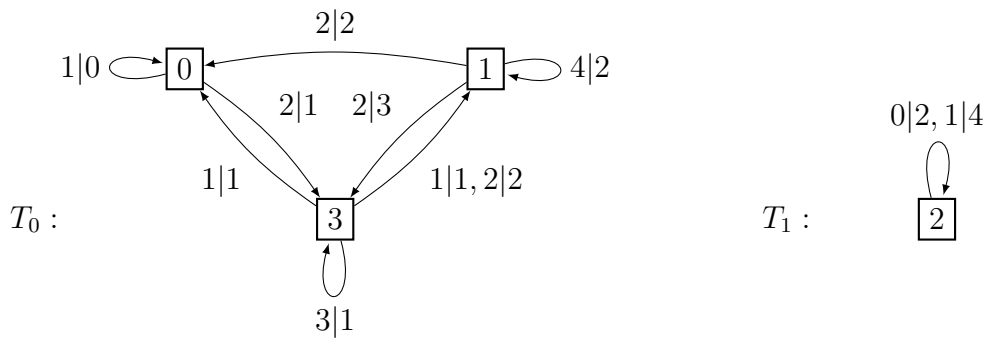
Sequences $x$ and $y$ are in the relation $\rho(M_\T)$ if and only
if there exists a row of tiles, with matching vertical edges, whose bottom edges
form the sequence $x$ and whose top edges form the sequence $y$. 
For example, consider the $6\times1$ rectangle tiled with Jeandel-Rao tiles:
\begin{center}
    \input{article2_T0_6x1_tiling.tex}
\end{center}
The sequence of bottom colors 
is $232212$.
Starting in state $0$ and reading that word as input in the transducer
$M_{\T_0}$ we follow the transitions:
\[
0\xrightarrow{2|1}
3\xrightarrow{3|1}
3\xrightarrow{2|2}
1\xrightarrow{2|3}
3\xrightarrow{1|1}
1\xrightarrow{2|2}
0
\]
We finish in state $0$ and we get $112312$ as output which corresponds
to the sequence of top colors of the same row.
In general, there is a
one-to-one correspondence between valid tilings of the plane, and
the execution of the transducer on biinfinite sequences.

Notice that the result of the usual composition of transducers
corresponds to the transducer of $\T\boxslash^2\Scal$:
\begin{equation}\label{eq:boxminus_with_transducer}
    M_\T\circ M_{\Scal} = M_{\T\boxslash^2\Scal}.
\end{equation}
As done in \cite{jeandel_aperiodic_2015},
the result of the composition can be filtered by 
recursively removing any source or sink state from the transducer
reducing the size of the set of tiles.
This is very helpful for doing computations.

\subsection{$d$-dimensional word}

In this section, we recall the definition of $d$-dimensional word that appeared
in \cite{MR2579856} and we keep the notation $u\odot^i v$ they proposed
for the concatenation. 

We denote by
$\{\be_k|1\leq k\leq d\}$ the canonical
basis of $\Z^d$ where $d\geq1$ is an integer.
If $i\leq j$ are integers, then $\llbracket i, j\rrbracket$ denotes the
interval of integers $\{i, i+1, \dots, j\}$.
Let $\bn=(n_1,\dots,n_d)\in\N^d$ and $\A$ be an alphabet.
We denote by $\A^{\bn}$ the set of functions
\begin{equation*}
    u:
\llbracket 0,n_1-1\rrbracket
\times
\cdots
\times
\llbracket 0,n_d-1\rrbracket
\to\A.
\end{equation*}
An element $u\in\A^\bn$ is called a
\emph{$d$-dimensional word $u$ of shape $\bn=(n_1,\dots,n_d)\in\N^d$}
on the alphabet~$\A$.
We use the notation $\shape(u)=\bn$ when necessary.
The set of all finite $d$-dimensional words is 
$\A^{*^d}=\{\A^\bn\mid\bn\in\N^d\}$.
A $d$-dimensional word of shape $\be_k+\sum_{i=1}^d\be_i$ is called a
\emph{domino in the direction $\be_k$}.
When the context is clear, we write $\A$ instead of $\A^{(1,\dots,1)}$.
When $d=2$, we represent a $d$-dimensional word $u$ of shape $(n_1,n_2)$ as a
matrix with Cartesian coordinates:
\begin{equation*}
    u=
    \left(\begin{array}{ccc}
        u_{0,n_2-1} &\dots   & u_{n_1-1,n_2-1} \\
        \dots   &\dots   & \dots \\
        u_{0,0} &\dots   & u_{n_1-1,0}
    \end{array}\right).
\end{equation*}
Let $\bn,\bm\in\N^d$ and $u\in\A^\bn$ and $v\in\A^\bm$.
If there exists an index $i$ such that the
shapes $\bn$ and $\bm$ are equal except maybe at index $i$,
then the \emph{concatenation of $u$ and $v$ in the direction $\be_i$} 
is \emph{defined}: it is
the 
$d$-dimensional word $u\odot^i v$ of shape $(n_1,\dots,n_{i-1},n_i+m_i,n_{i+1},\dots,n_d)\in\N^d$
given as
\begin{equation*}
    (u\odot^i v) (\ba) = 
\begin{cases}
    u(\ba)          & \text{if}\quad 0 \leq a_i < n_i,\\
    v(\ba-n_i\be_i) & \text{if}\quad n_i \leq a_i < n_i+m_i.
\end{cases}
\end{equation*}
If the shapes $\bn$ and $\bm$ are not equal except at
index $i$, we say that the concatenation of $u\in\A^\bn$ and $v\in\A^\bm$ in
the direction $\be_i$ is \emph{not defined}.
The following equation illustrates the concatenation of words in the direction $\be_1$ and $\be_2$ when $d=2$:
\[
    \arraycolsep=2.5pt
\left(\begin{array}{ccccc}
2 & 8 & 7  \\
7 & 3 & 9 \\
1 & 1 & 0 \\
6 & 6 & 7 \\
7 & 4 & 3 
\end{array}\right)
\odot^1
\left[
\left(\begin{array}{ccccc}
4 & 5 \\
10 & 5
\end{array}\right)
\odot^2
\left(\begin{array}{ccccc}
3 & 10 \\
9 & 9 \\
0 & 0 \\
\end{array}\right)
\right]
    =
\left(\begin{array}{ccccc}
2 & 8 & 7  \\
7 & 3 & 9 \\
1 & 1 & 0 \\
6 & 6 & 7 \\
7 & 4 & 3 
\end{array}\right)
\odot^1
\left(\begin{array}{ccccc}
3 & 10 \\
9 & 9 \\
0 & 0 \\
4 & 5 \\
10 & 5
\end{array}\right)
    =
\left(\begin{array}{ccccc}
2 & 8 & 7 & 3 & 10 \\
7 & 3 & 9 & 9 & 9 \\
1 & 1 & 0 & 0 & 0 \\
6 & 6 & 7 & 4 & 5 \\
7 & 4 & 3 & 10 & 5
\end{array}\right).
\]

Let $\bn,\bm\in\N^d$ and $u\in\A^\bn$ and $v\in\A^\bm$.
We say that $u$ \emph{occurs in $v$ at position $\bp\in\N^d$} if
$v$ is large enough, i.e., $\bm-\bp-\bn\in\N^d$ and
\[
    v(\ba+\bp) = u(\ba)
\]
for all $\ba=(a_1,\dots,a_d)\in\N^d$ such that 
$0\leq a_i<n_i$ with $1\leq i\leq d$.
If $u$ occurs in $v$ at some position, then we say that $u$ is a
$d$-dimensional \emph{subword} or \emph{factor} of $v$.

The distinction between tilings and 2-dimensional words is illustrated in the
Figure~\ref{fig:tiling-vs-word}. Our point of view is to consider Wang tilings
as 2-dimensional words. We 
always use the alphabet $\llbracket 0,n-1\rrbracket$ to represent a set of Wang tiles of cardinality $n$ with some chosen bijection between the two sets.
For Jeandel-Rao tiles, we use the alphabet
             $\llbracket 0,10\rrbracket$
             following the order from left to right shown in
             Figure~\ref{fig:jeandel_rao_tile_set}.
Depending on the context, we illustrate 2-dimensional words and tilings in this contribution
in one of the three ways depicted in Figure~\ref{fig:tiling-vs-word}
acknowledging that there is a bijection between the representations.

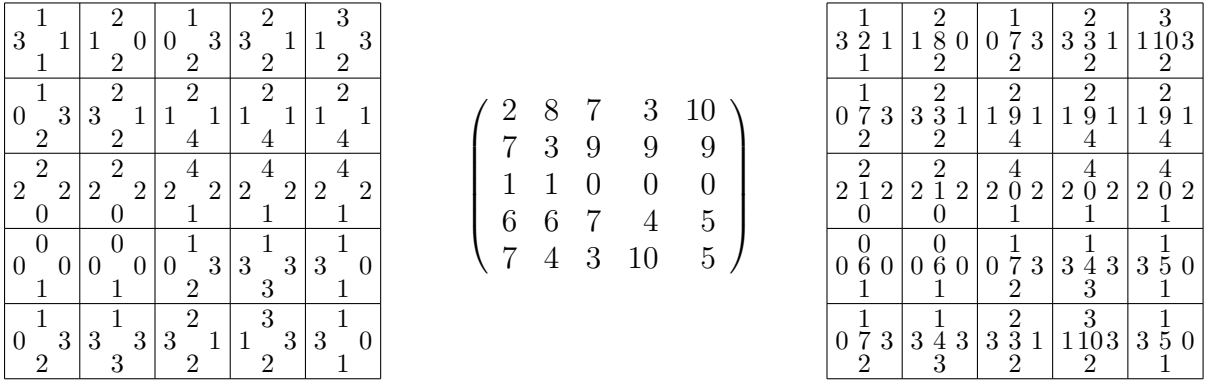
\begin{figure}[h]
    \raisebox{-25mm}{\input{article2_T0_5x5_tiling_noid_nocolor.tex}}
    \qquad
    $\input{article2_T0_5x5_matrix.tex}$
    \qquad
    \raisebox{-25mm}{\input{article2_T0_5x5_tiling_nocolor.tex}}
    \caption{Left: a Wang tiling with Jeandel-Rao tiles. 
             Middle: the 2-dimensional word over the 
             alphabet $\llbracket 0,10\rrbracket$
             giving the indices of Jeandel-Rao tiles 
             appearing in the tiling on the left.
             Right: letters of the 2-dimensional word 
             embedded in the tiling at the center of each tile.}
    \label{fig:tiling-vs-word}
\end{figure}

\subsection{$d$-dimensional language}

A subset $L\subseteq\A^{*^d}$ is called a $d$-dimensional \emph{language}. The
\emph{factorial closure} of a language $L$ is
\begin{equation*}
    \overline{L}^{Fact}
    = \{u\in\A^{*^d} \mid u\text{ is a $d$-dimensional subword of some } 
                           v\in L\}.
\end{equation*}
A language $L$ is \emph{factorial} if $\overline{L}^{Fact}=L$.
All languages considered in this contribution are factorial.
Given a tiling $x\in\A^{\Z^d}$, the \emph{language} $\L(x)$ defined by $x$ is
\begin{equation*}
    \L(x) = \{u\in\A^{*^d} \mid u\text{ is a $d$-dimensional subword of } x\}.
\end{equation*}
The \emph{language} of a subshift $X\subseteq\A^{\Z^d}$ is
    $\L_X = \cup_{x\in X} \L(x)$.
Conversely, given a factorial language $L\subseteq\A^{*^d}$ we define the subshift
\begin{equation*}
    \X_L = \{x\in\A^{\Z^d}\mid \L(x)\subseteq L\}.
\end{equation*}
A $d$-dimensional subword $u\in\A^{*^d}$ is \emph{allowed} in a 
subshift $X\subset\A^{\Z^d}$ if $u\in\L_X$
and it is \emph{forbidden} in $X$ if $u\notin\L_X$.
A language $L\subseteq\A^{*^d}$ is \emph{forbidden} in a 
subshift
$X\subset\A^{\Z^d}$ if $L\cap\L_X=\varnothing$.

\subsection{$d$-dimensional morphisms}

In this section, we generalize the definition of $d$-dimensional morphisms
\cite{MR2579856} to the case where the domain and codomain are different as for
$S$-adic systems \cite{MR3330561}.

Let $\A$ and $\B$ be two alphabets.
Let $L\subseteq\A^{*^d}$ be a factorial language.
A function $\omega:L\to\B^{*^d}$ is a \emph{$d$-dimensional
morphism} if for every
$i$ with $1\leq i\leq d$,
and every $u,v\in L$ such that $u\odot^i v\in L$ is defined
we have
that the concatenation $\omega(u)\odot^i \omega(v)$
in direction $\be_i$ is defined and
\begin{equation*}
    \omega(u\odot^i v) = \omega(u)\odot^i \omega(v).
\end{equation*}
Note that the left-hand side of the equation is defined since
$u\odot^i v$ belongs to the domain of $\omega$.
A $d$-dimensional morphism $L\to\B^{*^d}$ is thus completely defined from the
image of the letters in $\A$, so we sometimes denote
a $d$-dimensional morphism as a rule $\A\to\B^{*^d}$ 
when the language $L$ is unspecified.

The next lemma can be deduced from the definition.
It says that when $d\geq 2$ every $d$-dimensional morphism
defined on the whole set $L=\A^{*^d}$ is uniform.
We say that a $d$-dimensional morphism $\omega:L\to\B^{*^d}$ is \emph{uniform}
if there exists a shape $\bn\in\N^d$ such that $\omega(a)\in\B^\bn$ for
every letter $a\in\A$.
These are called block-substitutions in \cite{frank_introduction_2018}.

\begin{lemma}
If $d\geq 2$, every $d$-dimensional morphism $\omega:\A^{*^d}\to\B^{*^d}$ is uniform.
\end{lemma}

Therefore, to consider non-uniform $d$-dimensional morphisms when $d\geq 2$, we
need to restrict the domain to a strict subset $L\subsetneq\A^{*^d}$.
In \cite{MR2579856} and \cite[p.144]{MR1014984}, they consider the case $\A=\B$
and they restrict the domain of $d$-dimensional morphisms to the language they
generate.

Given a language $L\subseteq\A^{*^d}$ of $d$-dimensional words and 
a $d$-dimensional morphism $\omega:L\to\B^{*^d}$, we define the image of the
language $L$ under $\omega$ as the language
\begin{equation*}
\overline{\omega(L)}^{Fact}
    = \{u\in\B^{*^d} \mid u\text{ is a $d$-dimensional subword of }
                  \omega(v) \text{ with } v\in L\}
    \subseteq \B^{*^d}.
\end{equation*}

Let $L\subseteq\A^{*^d}$ be a factorial language
and $\Xcal_L\subseteq\A^{\Z^d}$ be the subshift generated by $L$.
A $d$-dimensional morphism 
$\omega:L \to\B^{*^d}$ 
can be extended to a $d$-dimensional morphism
$\omega:\Xcal_L\to\B^{\Z^d}$
in such a way that the origin of $\omega(x)$ is at zero position
in the word $\omega(x_\zero)$
for all $x\in\Xcal_L$.
In general, the closure under the shift of the image of a subshift $X\subseteq\A^{\Z^d}$ under $\omega$
is the subshift
\begin{equation*}
\overline{\omega(X)}^{\sigma}
    = \{\sigma^\bk\omega(x)\in\B^{\Z^d} \mid \bk\in\Z^d, x\in X\}
    \subseteq \B^{\Z^d}.
\end{equation*}
Now we show that $d$-dimensional morphisms preserve minimality of subshifts.

\begin{lemma}\label{lem:minimal-implies-minimal}
    Let $\omega:X\to\B^{\Z^d}$ be a $d$-dimensional morphism for some 
    $X\subseteq\A^{\Z^d}$.
    If $X$ is a minimal subshift, then $\overline{\omega(X)}^{\sigma}$ is
    a minimal subshift.
\end{lemma}

\begin{proof}
    Let $\varnothing\neq Z\subseteq\overline{\omega(X)}^{\sigma}$ be a closed
    shift-invariant subset.
    We want to show that $\overline{\omega(X)}^{\sigma}\subseteq Z$.
    Let $u\in\overline{\omega(X)}^{\sigma}$.
    Thus $u=\sigma^\bk\omega(x)$ for some $\bk\in\Z^d$ and $x\in X$.
    Since $Z\neq\varnothing$, there exists $z\in Z$.
    Thus $z=\sigma^{\bk'}\omega(x')$ for some $\bk'\in\Z^d$ and $x'\in X$.
    Since $X$ is minimal, there exists a sequence $(\bk_n)_{n\in\N}$,
    $\bk_n\in\Z^d$, such that $x=\lim_{n\to\infty}\sigma^{\bk_n} x'$.
    For some other sequence $(\bh_n)_{n\in\N}$, $\bh_n\in\Z^d$, we have
    \begin{equation*}
        u 
        = \sigma^\bk\omega(x) 
        = \sigma^\bk\omega\left(\lim_{n\to\infty}\sigma^{\bk_n} x'\right) 
        = \sigma^\bk\lim_{n\to\infty}\sigma^{\bh_n} \omega\left(x'\right)
        = \lim_{n\to\infty}\sigma^{\bk+\bh_n-\bk'} z.
    \end{equation*}
    Since $Z$ is closed and shift-invariant, it follows that $u\in Z$.
\end{proof}

Suppose now that $\A=\B$.
We say that a $d$-dimensional morphism $\omega:L\to\A^{*^d}$, with $L\subset\A^{*^d}$, is
\emph{expansive}
if for every $a\in\A$ and $K\in\N$,
there exists $m\in\N$ such that 
$\min(\shape(\omega^m(a)))>K$.
We say that $\omega$ is \emph{primitive}
if there exists $m\in\N$ such that
for every $a,b\in\A$ the letter $b$ occurs in $\omega^m(a)$.

The definition of prolongable substitutions \cite[Def.
1.2.18--19]{MR2742574} can be adapted in the case of $d$-dimensional morphisms.
Let $\omega:L \to\A^{*^d}$ with $L\subseteq\A^{*^d}$ be a $d$-dimensional morphism.
Let $s\in\{+1,-1\}^d$.
We say that $\omega$ is \emph{prolongable on a letter $a\in\A$} in the
hyperoctant of sign $s$ if the letter $a$ appears in the appropriate corner of
its own image $\omega(a)$ more precisely at position
$p=(p_1,\dots,p_d)\in\N^d$ where
\[
p_i = 
\begin{cases}
0       & \text{if}\quad s_i = +1,\\
n_i-1   & \text{if}\quad s_i = -1.
\end{cases}
\]
where $\bn=(n_1,\dots,n_d)\in\N^d$ is the shape of $\omega(a)$.
If $\omega$ is prolongable on letter $a\in\A$ in the hyperoctant of
sign $s$ 
and if
$\lim_{m\to\infty}\min(\shape(\omega^m(a)))=\infty$, then
$\lim_{m\to\infty}\omega^m(a)$ is a $d$-dimensional infinite word
$s_0\N\times\dots\times s_{d-1}\N\to\A$.

\subsection{Self-similar subshifts}\label{sec:self-similar}

In this section we consider languages and subshifts defined from substitutions
leading to self-similar structures.
A subshift $X\subseteq\A^{\Z^d}$ (resp. a language $L\subseteq\A^{*^d}$)
is \emph{self-similar}
if there exists an expansive
$d$-dimensional morphism $\omega:\A\to\A^{*^d}$ such that
$X=\overline{\omega(X)}^\sigma$
(resp.  $L=\overline{\omega(L)}^{Fact}$).

Self-similar languages and subshifts can be constructed by iterative
application of the morphism $\omega$ starting with the letters.
The \emph{language} $\L_\omega$ defined by an expansive $d$-dimensional
morphism $\omega:\A\to\A^{*^d}$ is
\begin{equation*}
    \L_\omega = \{u\in\A^{*^d} \mid u\text{ is a $d$-dimensional subword of }
    \omega^n(a) \text{ for some } a\in\A\text{ and } n\in\N \}.
\end{equation*}
It satisfies
$\L_\omega=\overline{\omega(\L_\omega)}^{Fact}$
and thus is self-similar.
The \emph{substitutive shift} $\X_\omega=\X_{\L_\omega}$
defined from the language of $\omega$ is a self-similar subshift.
If $\omega$ is primitive then $\X_\omega$ is minimal using standard arguments
\cite[\S 5.2]{MR2590264}.
Thus $\X_\omega$ is the smallest nonempty subshift
$X\subseteq\A^{\Z^d}$ satisfying $X=\overline{\omega(X)}^{\sigma}$.

\subsection{$d$-dimensional recognizability and aperiodicity}\label{sec:recognizability}

The definition of recognizability dates back to the work of Host, Quéffelec and
Mossé \cite{MR1168468}. See also \cite{akiyama_mosse_2017} who proposed a
completion of the statement and proof for B. Mossé's unilateral recognizability
theorem.
The definition introduced below is based on work of Berthé et al.
\cite{MR4015135} on the recognizability in the case
of $S$-adic systems where more than one substitution is involved.

Let $X\subseteq\A^{\Z^d}$ and
$\omega:X\to\B^{\Z^d}$ be a $d$-dimensional morphism.
If $y\in\overline{\omega(X)}^{\sigma}$, i.e.,
$y=\sigma^\bk\omega(x)$ for some $x\in X$ and $\bk\in\Z^d$, where $\sigma$ is
the $d$-dimensional shift map, we say that $(\bk,x)$ is a
\emph{$\omega$-representation of $y$}. We say that it is \emph{centered} if
$y_\zero$ lies inside of the image of $x_\zero$, i.e., if
$\zero\leq\bk<\shape(\omega(x_\zero))$ coordinate-wise.
We say that $\omega$ is \emph{recognizable in $X\subseteq\A^{\Z^d}$}
if each $y\in\B^{\Z^d}$ has at most one centered $\omega$-representation 
$(\bk,x)$ with $x\in X$.


The next proposition is well-known, see \cite{MR1637896,MR1168468}, who showed that
recognizability and aperiodicity are equivalent for primitive substitutive
sequences. We state only one direction (the easy one) of the equivalence.
Its proof with the same notations can be found in
\cite{MR3978536}.

\begin{proposition}\label{prop:expansive-recognizable-aperiodic}
    Let $X\subseteq\A^{\Z^d}$ and
    $\omega:X\to\A^{\Z^d}$ be an expansive $d$-dimensional morphism.
    If $X$ is a self-similar subshift such that
    $\overline{\omega(X)}^\sigma=X$ and
    $\omega$ is recognizable in $X$, then $X$ is aperiodic.
\end{proposition}

The next lemma is very important for the current contribution.

\begin{lemma}\label{lem:aperiodic-implies-aperiodic}
    Let $\omega:X\to \B^{\Z^d}$ be some $d$-dimensional morphism where
    $X\subseteq\A^{\Z^d}$ is a subshift.
    If $X$ is aperiodic and $\omega$ is recognizable in $X$, then
    $\overline{\omega(X)}^\sigma$ is aperiodic.
\end{lemma}

\begin{proof}
    Let $y\in\overline{\omega(X)}^\sigma$.
    Then, there exist $\bk\in\Z^d$ and $x\in X$ such that 
    $(\bk, x)$ is a centered $\omega$-representation of $y$, i.e.,
    $y=\sigma^\bk\omega(x)$.
    Suppose by contradiction that $y$ has a nontrivial period
    $\bp\in\Z^d\setminus\zero$.
    Since $y =\sigma^\bp y =\sigma^{\bp+\bk}\omega(x)$,
    we have that
    $(\bp+\bk, x)$ is a $\omega$-representation of $y$.
    Since $\omega$ is recognizable, this representation is not centered.
    Therefore there exists $\bq\in\Z^d\setminus\zero$ such that
    $y_\zero$ lies in the image of $x_\bq=(\sigma^\bq x)_\zero$.
    Therefore there exists $\bk'\in\Z^d$ such that
    $(\bk',\sigma^\bq x)$ is a centered $\omega$-representation of $y$.
    Since $\omega$ is recognizable, we conclude that
    $\bk=\bk'$ and $x=\sigma^\bq x$. Then $x\in X$ is periodic which is a
    contradiction.
\end{proof}

%


Let $X$, $Y$ be two subshifts
and $\omega:X\to Y$ be a $d$-dimensional morphism.
In general, $\omega(X)$ is not closed under the shift which implies that
$\omega$ is not onto. This motivates the following definition.
If $Y=\overline{\omega(X)}^\sigma$, then
we say that $\omega$
is \emph{onto up to a shift}.

\begin{remark}
We believe the existence of a $d$-dimensional morphism $\omega:X\to \B^{\Z^d}$
which is recognizable in $X$ implies the existence of a
\emph{homeomorphism} between $X$ and $\overline{\omega(X)}^\sigma$ where 
elements of $X$ and $\overline{\omega(X)}^\sigma$ are tilings of $\R^d$ instead of $\Z^d$
and the alphabet is replaced by a set of tiles in $\R^d$
(as done in \cite{MR1452190} or see also stone inflations in \cite{MR3136260}).
We do not attempt a proof of this here
as we use in this contribution the point of view of symbolic dynamics and tilings
of $\Z^d$ which makes some notions easier (concatenation, morphisms).
On the other hand, notions like injectivity, surjectivity and
homeomorphisms are less natural as we must use recognizability and expressions
like ``onto up to a shift'' instead.
\end{remark}


%

\part{Desubstituting Wang tilings with markers}\label{part:1}

\section{Markers, desubstitution of Wang shifts and algorithms}
\label{sec:markers}

In this section, we recall from \cite{MR3978536} the notion of
markers and the result on the existence of a $2$-dimensional morphism 
between two Wang shifts that is recognizable and onto up to a shift.
We also propose new algorithms to find markers and desubstitutions of Wang
shifts.

\subsection{Markers}

In what follows we consider a set of Wang tiles as an alphabet and a Wang tiling as
$2$-dimensional word. This allows to use the concepts of languages,
$2$-dimensional morphisms, recognizability introduced in the preliminaries in
the context of Wang tilings.

\begin{definition}\label{def:markers}
    Let $\T$ be a set of Wang tiles 
    and let $\Omega_\T$ be its Wang shift.
    A nonempty subset $M\subset\T$ is called \emph{markers in the
    direction $\be_1$} 
    if positions of tiles from $M$ are nonadjacent columns, that is,
    for all tilings $w\in\Omega_\T$
    there exists $A\subset\Z$ such that
    \begin{equation*}
        w^{-1}(M) = A \times \Z
        \quad
        \text{ with }
        \quad
        1\notin A-A.
    \end{equation*}
    It is called \emph{markers in the direction $\be_2$} if
    positions of tiles from $M$ are nonadjacent rows, that is,
    for all tilings $w\in\Omega_\T$
    there exists $B\subset\Z$ such that
    \begin{equation*}
        w^{-1}(M) = \Z \times B
        \quad
        \text{ with }
        \quad
        1\notin B-B.
    \end{equation*}
\end{definition}

Note that it follows from the definition that a subset of markers is a proper
subset of $\T$ as the case $M=\T$ is impossible.

\begin{figure}[h]
\begin{center}
    \includegraphics[width=.85\linewidth]{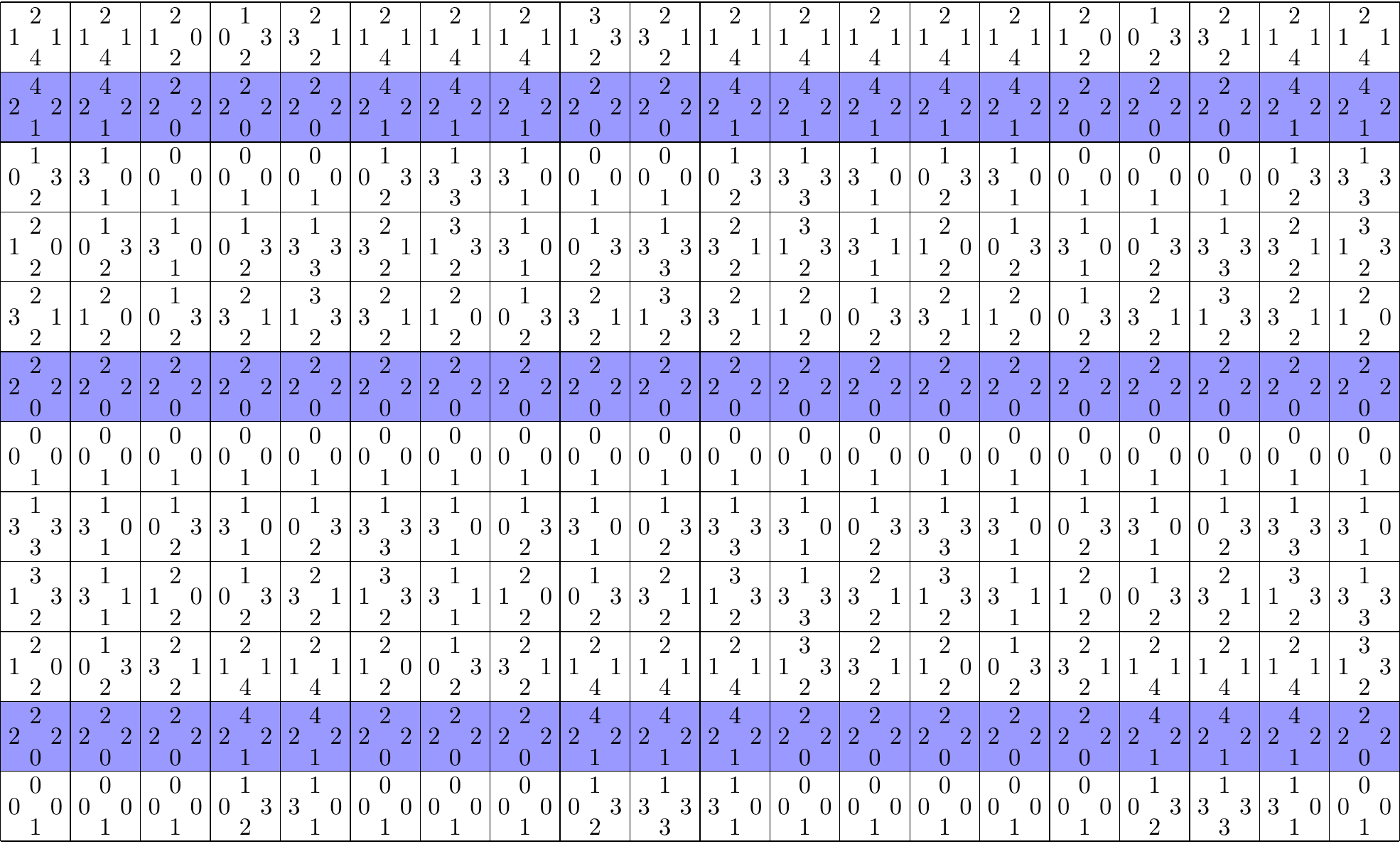}
\end{center}
\caption{ 
    This is a copy without colors of the Jeandel-Rao tiling shown in
    Figure~\ref{fig:jeandel-rao-some-tiling}.
    The tiles $M=\{(2,4,2,1), (2,2,2,0)\}\subset\T_0$ are markers in the direction $\be_2$
    and are shown with a blue background. They form complete nonadjacent rows
    of tiles in any tilings in $\Omega_0$.}
\label{fig:jeandel-rao-some-tiling-with-markers}
\end{figure}

Let us illustrate the definition on the set $\T_0$ of Jeandel-Rao tiles.
The two tiles in the subset
$M=\left\{
\raisebox{-3mm}{
\begin{tikzpicture}
[scale=.8]
\tikzstyle{every node}=[font=\footnotesize]
\fill[cyan] (1.0, 0.0) -- (0.5, 0.5) -- (1.0, 1.0);
\fill[lightgray] (0.0, 1.0) -- (0.5, 0.5) -- (1.0, 1.0);
\fill[cyan] (0.0, 0.0) -- (0.5, 0.5) -- (0.0, 1.0);
\fill[red] (0.0, 0.0) -- (0.5, 0.5) -- (1.0, 0.0);
\draw (1.0, 0.0) -- ++ (0,1);
\draw (0.0, 1.0) -- ++ (1,0);
\draw (0.0, 0.0) -- ++ (0,1);
\draw (0.0, 0.0) -- ++ (1,0);
\node[rotate=0,black] at (0.8, 0.5) {2};
\node[rotate=0,black] at (0.5, 0.8) {4};
\node[rotate=0,black] at (0.2, 0.5) {2};
\node[rotate=0,black] at (0.5, 0.2) {1};
\end{tikzpicture}
}
,
\raisebox{-3mm}{
\begin{tikzpicture}
[scale=.8]
\tikzstyle{every node}=[font=\footnotesize]
\fill[cyan] (2.1, 0.0) -- (1.6, 0.5) -- (2.1, 1.0);
\fill[cyan] (1.1, 1.0) -- (1.6, 0.5) -- (2.1, 1.0);
\fill[cyan] (1.1, 0.0) -- (1.6, 0.5) -- (1.1, 1.0);
\fill[white] (1.1, 0.0) -- (1.6, 0.5) -- (2.1, 0.0);
\draw (2.1, 0.0) -- ++ (0,1);
\draw (1.1, 1.0) -- ++ (1,0);
\draw (1.1, 0.0) -- ++ (0,1);
\draw (1.1, 0.0) -- ++ (1,0);
\node[rotate=0,black] at (1.9000000000000001, 0.5) {2};
\node[rotate=0,black] at (1.6, 0.8) {2};
\node[rotate=0,black] at (1.3, 0.5) {2};
\node[rotate=0,black] at (1.6, 0.2) {0};
\end{tikzpicture}
}
\right\}\subset\T_0$
both have left and right color $2$ (blue)
but none of the other nine tiles from $\T_0\setminus M$ have this color on their
left and right sides (see Figure~\ref{fig:jeandel_rao_tile_set}).
Therefore any valid Wang tiling $w:\Z^2\to\T_0$ is made of horizontal lines using either tiles from $M$ or from $\T_0\setminus M$
(see Figure~\ref{fig:jeandel-rao-some-tiling-with-markers}).
Moreover, the horizontal lines using tiles from $M$ are nonadjacent
simply since top colors of tiles from $M$ (2 and 4) do not match their bottom
colors (0 and 1). 
Those two conditions implies that
the positions of the tiles in $M$ in any tiling are non adjacent complete
rows in $\Z^2$, 
that is, $M$ is a subset of
markers in the direction $\be_2$.
Proving that a subset of tiles is a subset of
markers uses very local observations. It leads to the following criteria.

\begin{lemma}\label{lem:criteria-markers}
    Let $\T$ be a set of Wang tiles 
    and let $\Omega_\T$ be its Wang shift.
    A nonempty subset $M\subset\T$ is a subset of markers in the
    direction $\be_1$ if and only if
    \begin{equation*}
        M\odot^1 M, \qquad
        M\odot^2 (\T\setminus M), \qquad
        (\T\setminus M) \odot^2 M
    \end{equation*}
    are forbidden in $\Omega_\T$.
    A nonempty subset $M\subset\T$ is a subset of markers in the
    direction $\be_2$ if and only if
    \begin{equation*}
        M\odot^2 M, \qquad
        M\odot^1 (\T\setminus M), \qquad
        (\T\setminus M) \odot^1 M
    \end{equation*}
    are forbidden in $\Omega_\T$.
\end{lemma}

\begin{proof}
    Suppose that $M\subset\T$ is a subset of markers in the
    direction $\be_2$ (the proof for direction $\be_1$ is the similar).
    For any tiling $w\in\Omega_\T$,
    there exists $B\subset\Z$ such that
        $w^{-1}(M) = \Z \times B$ with
        $1\notin B-B$.
    In any tiling $w\in\Omega_\T$,
    this means that the tile to the left and to the right of any tile in $M$
    also belongs to $M$. 
    Therefore, $M\odot^1 (\T\setminus M)$ and $(\T\setminus M) \odot^1 M$
    are forbidden in $\Omega_\T$.
    Moreover, the fact that
        $1\notin B-B$ implies that
        $M\odot^2 M$ is forbidden in $\Omega_\T$.

    Conversely, suppose that
        $M\odot^2 M$,
        $M\odot^1 (\T\setminus M)$ and
        $(\T\setminus M) \odot^1 M$
    are forbidden in $\Omega_\T$.
    The last two conditions implies that 
    in any tiling $w\in\Omega_\T$,
    the tile to the left and to the right
    of any tile in $M$ also belongs to $M$.
    Therefore tiles in $M$ appears as complete rows in $w$, that is,
        $w^{-1}(M) = \Z \times B$ for some $B\subset\Z$.
    Since $M\odot^2 M$ is forbidden in $\Omega_\T$, it means that
    the rows are nonadjacent, or equivalently, $1\notin B-B$.
    We conclude that $M$ is a subset of markers in the
    direction $\be_2$.
\end{proof}

Lemma~\ref{lem:criteria-markers} provides a way to prove that a subset is a
subset of markers and searching for them. For that, we need the following definition.

\begin{definition}[\bf surrounding of radius $r$]
    Let $X=\SFT(\F)\subset\A^{\Z^2}$ be a shift of finite type for some
    finite set $\F$ of forbidden patterns. A $2$-dimensional word
$u\in\A^\bn$,
with $\bn=(n_1,n_2)\in\N^2$,
admits a \emph{surrounding of radius $r\in\N$}
if there exists $w\in\A^{\bn+2(r,r)}$ 
such that 
    $u$ occurs in $w$ at position $(r,r)$
    and
$w$ contains no occurrences of forbidden patterns from $\F$.
\end{definition}

If a word admits a surrounding of radius $r\in\N$, it does not mean
it is in the language of the SFT. But if it admits no surrounding of radius $r$
for some $r\in\N$, then for sure it is not in the language of the SFT.
We state the following lemma in the context of Wang tiles.

\begin{lemma}\label{lem:surrounding}
    Let $\T$ be a set of Wang tiles and $u\in\T^\bn$ be a rectangular pattern
    seen as a $2$-dimensional word with $\bn=(n_1,n_2)\in\N^2$. If
    $u$ is allowed in $\Omega_\T$, then for every $r\in\N$ the word $u$ has a
    surrounding of radius $r$.
\end{lemma}

Equivalently the lemma says that
if there exists $r\in\N$ such that $u$ has no surrounding of radius~$r$,
then~$u$ is forbidden in $\Omega_\T$ and this is how we use
Lemma~\ref{lem:surrounding} to find markers.
We propose Algorithm~\ref{alg:find-markers} to compute markers from a Wang
tile set and a chosen surrounding radius to bound the computations. 
If the algorithm finds nothing, then maybe there is no markers or maybe 
one should try again after increasing the surrounding radius.
We prove in the next lemma that if the output is nonempty, it contains a subset of markers.

\begin{algorithm}[h]
    \caption{Find markers.
        If no markers are found, one should try increasing the radius $r$.}
    \label{alg:find-markers}
  \begin{algorithmic}[1]
    \Require $\T$ is a set of Wang tiles;
             $i\in\{1,2\}$ is a direction $\be_i$;
             $r\in\N$ is some radius.
      \Function{FindMarkers}{$\T$, $i$, $r$}
        \State $j\gets 3-i$
        \State $D_j \gets \left\{(u,v)\in\T^2\mid \text{ domino } u\odot^jv \text{
             admits a surrounding of radius $r$ with tiles in $\T$}\right\}$
        \State $U \gets \{\{u\}\mid u\in\T\}$
            \Comment Suggestion: use a union-find data structure 
        \ForAll{$(u,v) \in D_j$}
            \State Merge the sets containing $u$ and $v$ in the partition $U$.
        \EndFor
        \State $D_i \gets \left\{(u,v)\in\T^2\mid \text{ domino } u\odot^iv \text{
             admits a surrounding of radius $r$ with tiles in $\T$}\right\}$
         \State\Return $\{\text{set } M \text{ in the partition } U \mid
                               \left(M\times M\right) \cap D_i=\varnothing\}$
      \EndFunction
      \Ensure The output contains zero, one or more subsets of markers 
              in the direction $\be_i$.
  \end{algorithmic}
\end{algorithm}

\begin{lemma}
    If there exists $r\in\N$ and $i\in\{1,2\}$ such that the output of
      $\Call{FindMarkers}{\T, i, r}$
      contains a set $M$, then $M\subset\T$ is a subset of markers in the
      direction $\be_i$.
\end{lemma}

\begin{proof}
    Suppose that $i=2$, the case $i=1$ being similar.
    The output set $M$ is nonempty since is was created from the union of
    nonempty sets (see lines 4-6 in Algorithm~\ref{alg:find-markers}).
    Using Lemma~\ref{lem:surrounding},
    lines 3 to 6 implies that
        $M\odot^1 (\T\setminus M)$ and
        $(\T\setminus M) \odot^1 M$
    are forbidden in $\Omega_\T$.
    The lines 7 and 8 implies that
        $M\odot^2 M$ is forbidden in $\Omega_\T$.
    Then $M\subset\T$ is a subset of markers in the direction~$\be_i$.
\end{proof}

We believe that if a set of Wang tiles $\T$ has a subset of markers in the
direction $\be_i$ then there exists a surrounding radius $r\in\N$ such that
$\Call{FindMarkers}{\T, i, r}$ outputs this set of markers.
The fact that there is no upper bound for the surrounding radius is related to
the undecidability of the domino problem.
In practice, in the study of Jeandel-Rao tilings, a surrounding radius of 1, 2
or 3 is enough.

\subsection{Desubstitution of Wang shifts}

In this section we prove that
if a set of Wang tiles $\T$ has a subset of marker tiles, then
there exists another set $\Scal$ of Wang tiles and a 
nontrivial recognizable $2$-dimensional morphism
$\Omega_\Scal\to\Omega_\T$ that is onto up to a shift.
Thus, every Wang
tiling by $\T$ is up to a shift the image under a nontrivial $2$-dimensional
morphism $\omega$ of a unique Wang tiling in $\Omega_\Scal$.
We will see that the $2$-dimensional morphism is essentially $1$-dimensional.
The result proved here extends the one proved in \cite{MR3978536}. 

Before stating the result, let us see how the two markers in Jeandel-Rao tiles
allows to desubstitute tilings. We know from the previous section that markers
appears as nonadjacent rows in Jeandel-Rao tilings. Therefore the row above
(and below) some row of markers is made of nonmarker tiles. Let us consider the row above.
The idea is to collapse that row onto the row of markers just below and we perform this using
fusion of tiles.
More precisely, the first marker tile has color 4 on top and only one tile has color 4
on the bottom. Therefore, that marker tile is always matched on top with that tile.
Similarly, the second marker tile has color 2 on top and 
there are only four tiles
that have color 2 on the bottom. Therefore, the second marker
is always matched on top with one of those 4 tiles.
This means that 
one can glue, e.g. with tape, each
marker with one of the possible tiles that can go above of it. Doing so, one
gets five dominoes that may be represented as new Wang tiles with the fusion
operation $\boxminus$ introduced in the preliminaries as schematized in the following equation:
\[
    \left\{
    \raisebox{-9mm}{
    \input{article2_T0_0_9_domino.tex}
    \input{article2_T0_1_3_domino.tex}
    \input{article2_T0_1_7_domino.tex}
    \input{article2_T0_1_8_domino.tex}
    \input{article2_T0_1_10_domino.tex}}\right\}
    +
    \raisebox{-8mm}{
    \includegraphics[width=3cm]{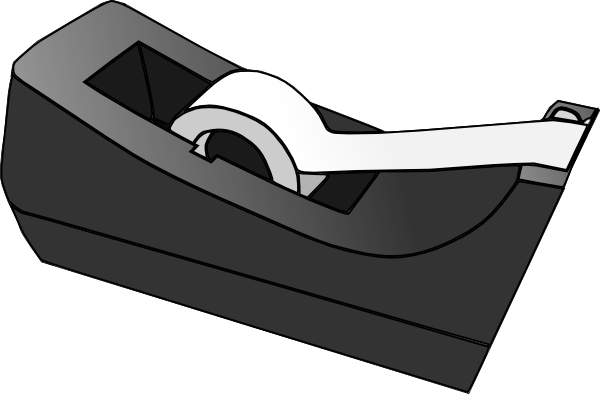}}
    \approx
    \left\{
    \raisebox{-4mm}{
\begin{tikzpicture}
[scale=1]
\tikzstyle{every node}=[font=\small]
\draw (9.8, 0.0) -- ++ (0,1);
\draw (8.8, 1.0) -- ++ (1,0);
\draw (8.8, 0.0) -- ++ (0,1);
\draw (8.8, 0.0) -- ++ (1,0);
\node[rotate=90,black] at (9.600000000000001, 0.5) {21};
\node[rotate=0,black] at (9.3, 0.8) {2};
\node[rotate=90,black] at (9.0, 0.5) {21};
\node[rotate=0,black] at (9.3, 0.2) {1};
\draw (10.9, 0.0) -- ++ (0,1);
\draw (9.9, 1.0) -- ++ (1,0);
\draw (9.9, 0.0) -- ++ (0,1);
\draw (9.9, 0.0) -- ++ (1,0);
\node[rotate=90,black] at (10.700000000000001, 0.5) {21};
\node[rotate=0,black] at (10.4, 0.8) {2};
\node[rotate=90,black] at (10.1, 0.5) {23};
\node[rotate=0,black] at (10.4, 0.2) {0};
\draw (12.0, 0.0) -- ++ (0,1);
\draw (11.0, 1.0) -- ++ (1,0);
\draw (11.0, 0.0) -- ++ (0,1);
\draw (11.0, 0.0) -- ++ (1,0);
\node[rotate=90,black] at (11.8, 0.5) {23};
\node[rotate=0,black] at (11.5, 0.8) {1};
\node[rotate=90,black] at (11.2, 0.5) {20};
\node[rotate=0,black] at (11.5, 0.2) {0};
\draw (13.100000000000001, 0.0) -- ++ (0,1);
\draw (12.100000000000001, 1.0) -- ++ (1,0);
\draw (12.100000000000001, 0.0) -- ++ (0,1);
\draw (12.100000000000001, 0.0) -- ++ (1,0);
\node[rotate=90,black] at (12.900000000000002, 0.5) {20};
\node[rotate=0,black] at (12.600000000000001, 0.8) {2};
\node[rotate=90,black] at (12.3, 0.5) {21};
\node[rotate=0,black] at (12.600000000000001, 0.2) {0};
\draw (14.200000000000001, 0.0) -- ++ (0,1);
\draw (13.200000000000001, 1.0) -- ++ (1,0);
\draw (13.200000000000001, 0.0) -- ++ (0,1);
\draw (13.200000000000001, 0.0) -- ++ (1,0);
\node[rotate=90,black] at (14.000000000000002, 0.5) {23};
\node[rotate=0,black] at (13.700000000000001, 0.8) {3};
\node[rotate=90,black] at (13.4, 0.5) {21};
\node[rotate=0,black] at (13.700000000000001, 0.2) {0};
\end{tikzpicture}}
\right\}.
\]

Therefore to build some Jeandel-Rao tiling, it is sufficient to build a tiling
with another set $\T_1$ of Wang tiles obtained from the set $\T_0$ after removing the
markers and adding the five tiles obtained from the fusion operation. 
One may also remove the tile with color
4 on the bottom since it always appear above of a marker tile.
One may recover some Jeandel-Rao tiling by applying a $2$-dimensional morphism
which replaces the five merged tiles by their associated equivalent dominoes,
see Figure~\ref{fig:image-of-substitution}. 
It turns out
that this decomposition is unique as we prove below.
The creation of the set $\T_1$ from $\T_0$ gives the intuition on the
construction of Algorithm~\ref{alg:find-recognizable-sub-from-markers} which 
follows the same recipe and takes any set of Wang tiles with markers as input.

\begin{figure}[h]
    \begin{center}
    \begin{tikzpicture}[auto,scale=1.4]
        \node (A) at (0,0) {\includegraphics[width=.39\linewidth]{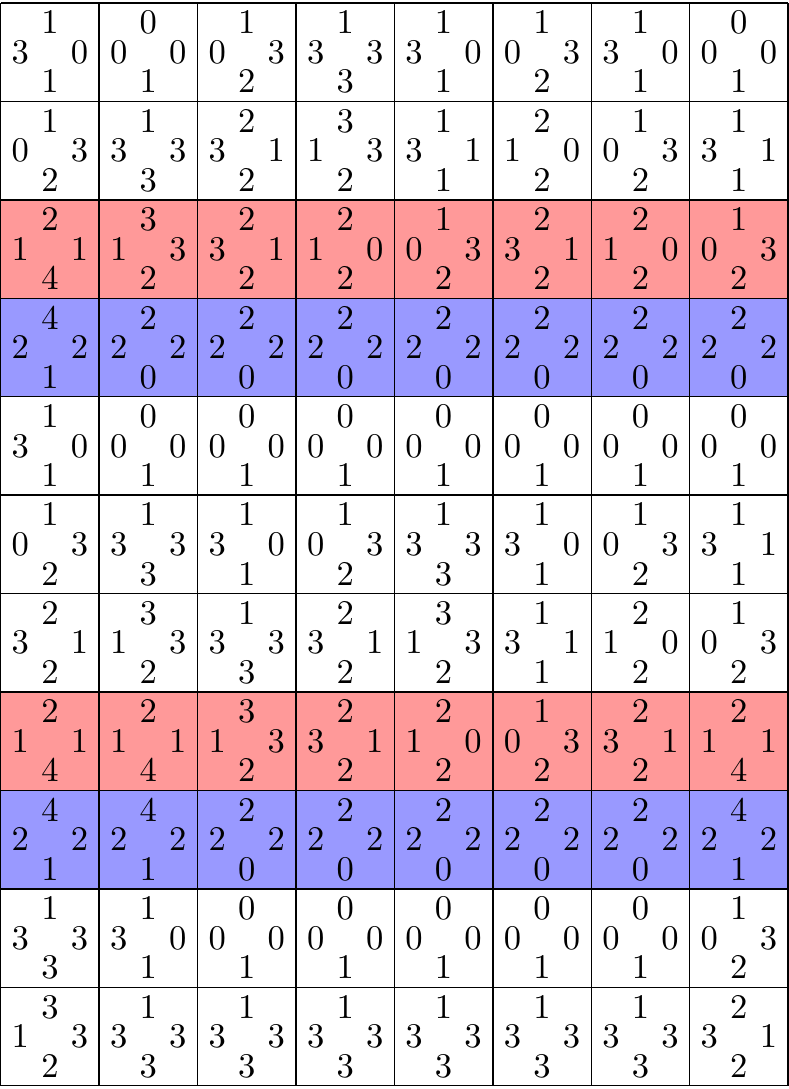}};
        \node (B) at (6,0) {\includegraphics[width=.39\linewidth]{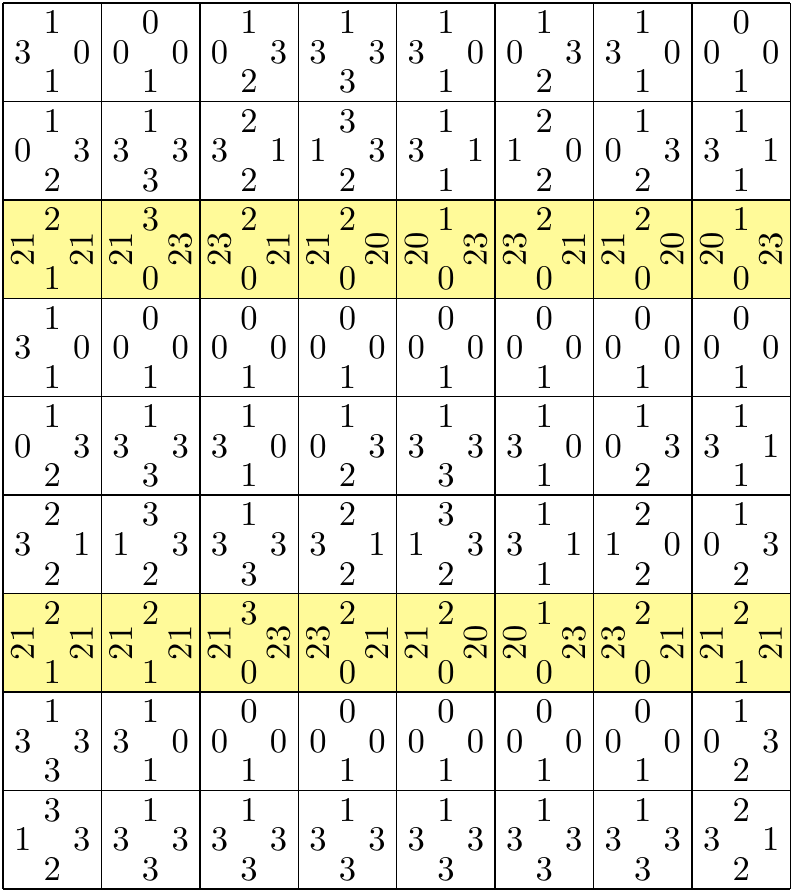}};
        \draw[<-|] (A) to node {$\omega_0$} (B);
    \end{tikzpicture}
    \end{center}
    \caption{The effect of the $2$-dimensional morphism
    $\omega_0:\Omega_1\to\Omega_0$ on tilings.}
    \label{fig:image-of-substitution}
\end{figure}

We can now state the small generalization of Theorem 10 from
\cite{MR3978536}. Indeed, while the presence of markers allows to
desubstitute tilings, there is a choice to be made.  We may construct the
substitution in such a way that the markers are on the left or
on the right in the image of letters that are dominoes (bottom or top for
marker in the direction $\be_2$). We make this distinction in the statement of
the following result.

\begin{theorem}\label{thm:exist-homeo}
    Let $\T$ be a set of Wang tiles 
    and let $\Omega_\T$ be its Wang shift.
    If there exists a subset
    $M\subset\T$ 
    of markers in the direction 
    $\be_i\in\{\be_1,\be_2\}$,
    then 
\begin{enumerate}[\rm (i)]
\item there exists
    a set of Wang tiles $\Scal_R$
    and a $2$-dimensional morphism
    $\omega_R:\Omega_{\Scal_R}\to\Omega_\T$
    such that 
    \begin{equation*}
        \omega_R(\Scal_R)\subseteq (\T\setminus M)\cup 
        \left((\T\setminus M)\odot^i M\right)
    \end{equation*}
    which is recognizable and onto up to a shift and
\item
    there exists a set of Wang tiles $\Scal_L$ and a $2$-dimensional morphism
    $\omega_L:\Omega_{\Scal_L}\to\Omega_\T$
    such that 
    \begin{equation*}
        \omega_L(\Scal_L)\subseteq (\T\setminus M)\cup 
        \left(M\odot^i (\T\setminus M)\right)
    \end{equation*}
    which recognizable and onto up to a shift.
\end{enumerate}
    There exists a surrounding radius $r\in\N$ such that $\omega_R$ and
    $\omega_L$ are computed using
    Algorithm~\ref{alg:find-recognizable-sub-from-markers}.
\end{theorem}

\begin{proof}
    The existence of $\omega_R$ was proved in \cite[Theorem
    10]{MR3978536}. Its proof follows the line of
    Algorithm~\ref{alg:find-recognizable-sub-from-markers}.

    The existence of $\omega_L$ was not proved in \cite[Theorem
    10]{MR3978536}, but it can be deduced by considering the
    image of the set of Wang tiles $\T$ under a reflection by the vertical or
    horizontal axis. For example, if $\be_i=\be_1$, then we consider the image
    under a reflection by the vertical axis:
    \begin{equation*}
        \T' = \left\{ \tile{$c$}{$b$}{$a$}{$d$}\,\,\,\middle|\,
        \tile{$a$}{$b$}{$c$}{$d$}=(a,b,c,d) \in \T \right\}
    \end{equation*}
    and the image $M'\subset\T'$ of markers accordingly.
    Then \cite[Theorem 10]{MR3978536} provide the existence of
    a set of Wang tiles $\Scal'$ and
    a $2$-dimensional morphism $\omega':\Scal'\to\T'$ which is 
        recognizable, onto up to a shift and
    such that
    \begin{equation*}
        \omega'(\Scal')\subseteq (\T'\setminus M')\cup 
        \left((\T'\setminus M')\odot^i M'\right).
    \end{equation*}
    By taking the image of $\omega'$ under a reflection by the vertical axis,
    we get the desired $\omega_L$ and $\Scal_L$.
\end{proof}

\begin{algorithm}
    \caption{Find a recognizable desubstitution of $\Omega_\T$ from markers}
    \label{alg:find-recognizable-sub-from-markers}
  \begin{algorithmic}[1]
          \Require $\T$ is a set of Wang tiles;
                   $M\subset\T$ is a subset of markers;
                   $i\in\{1,2\}$ is a direction $e_i$;
                   $r\in\N$ is a surrounding radius;
                   $s\in\{\scleft,\scright\}$ determines
                   whether the image of merged tiles is 
                   in $M\odot^i (\T\setminus M)$ (markers on the left)
                   or in $(\T\setminus M)\odot^i M$ (markers on the right).
      \Function{FindSubstitution}{$\T$, $M$, $i$, $r$, $s$}
        \State $D \gets \left\{(u,v)\in\T^2\mid \text{ domino } u\odot^iv \text{
             admits a surrounding of radius $r$ with tiles in $\T$}\right\}$
        \If{$s=\scleft$}
            \State $ P \gets \left\{(u,v)\in D
                \mid u\in M \text{ and } v\in\T\setminus M \right\}$
            \State $K \gets \left\{v\in\T\setminus M\mid \text{ there exists }
                            u \in \T\setminus M \text{ such that }
                            (u,v) \in D \right\}$
        \ElsIf{$s=\scright$}
            \State $ P \gets \left\{(u,v)\in D
                \mid u\in\T\setminus M \text{ and } v\in M \right\}$
            \State $K \gets \left\{u\in\T\setminus M\mid \text{ there exists }
                            v \in \T\setminus M \text{ such that }
                            (u,v) \in D \right\}$
        \EndIf
        \State $K\gets\Call{Sort}{K}$, $P\gets\Call{Sort}{P}$
        \Comment lexicographically on the indices of tiles
        \State $\Scal\gets K\cup \{u\boxslash^iv\mid(u,v)\in P\}$
        \Comment defines uniquely indices of tiles in $\Scal$ from $0$ to
        $|\Scal|-1$.
        \State \Return $\Scal$, $\omega:\Omega_\Scal\to\Omega_\T:
        \begin{cases}
        u\boxslash^i v
        \quad \mapsto \quad
        u\odot^i v
        & \text{ if }
        (u, v) \in P\\
        u \quad \mapsto \quad
        u & \text{ if } u \in K.
        \end{cases}$
      \EndFunction
      \Ensure $\Scal$ is a set of Wang tiles;
              $\omega:\Omega_\Scal\to\Omega_\T$ is 
        recognizable and onto up to a shift.
  \end{algorithmic}
\end{algorithm}

In the definition of $\omega$
in Algorithm~\ref{alg:find-recognizable-sub-from-markers},
given two Wang tiles $u$ and $v$ such that
$u\boxslash^i v$ is defined for $i\in\{1,2\}$, the map
\begin{equation*}
u\boxslash^i v
\quad
\mapsto
\quad
u\odot^i v
\end{equation*}
can be seen as a decomposition of Wang tiles:
\begin{center}
\begin{tikzpicture}
[scale=0.900000000000000]
\draw (0, 0) -- (1, 0);
\draw (0, 0) -- (0, 1);
\draw (1, 1) -- (1, 0);
\draw (1, 1) -- (0, 1);
\node[rotate=0,font=\tiny] at (0.8, 0.5) {Z};
\node[rotate=0,font=\tiny] at (0.5, 0.8) {BD};
\node[rotate=0,font=\tiny] at (0.2, 0.5) {X};
\node[rotate=0,font=\tiny] at (0.5, 0.2) {AC};
\node at (1.5,.5) {$\mapsto$};
\node at (3.0,0.5) {\begin{tikzpicture}[scale=0.900000000000000]
\draw (0, 0) -- (1, 0);
\draw (0, 0) -- (0, 1);
\draw (1, 1) -- (1, 0);
\draw (1, 1) -- (0, 1);
\node[rotate=0,font=\tiny] at (0.8, 0.5) {Y};
\node[rotate=0,font=\tiny] at (0.5, 0.8) {B};
\node[rotate=0,font=\tiny] at (0.2, 0.5) {X};
\node[rotate=0,font=\tiny] at (0.5, 0.2) {A};
\draw (1, 0) -- (2, 0);
\draw (1, 0) -- (1, 1);
\draw (2, 1) -- (2, 0);
\draw (2, 1) -- (1, 1);
\node[rotate=0,font=\tiny] at (1.8, 0.5) {Z};
\node[rotate=0,font=\tiny] at (1.5, 0.8) {D};
\node[rotate=0,font=\tiny] at (1.2, 0.5) {Y};
\node[rotate=0,font=\tiny] at (1.5, 0.2) {C};
\end{tikzpicture}};
\end{tikzpicture}
\qquad
    \raisebox{5mm}{\text{ or }}
\qquad
\begin{tikzpicture}
[scale=0.900000000000000]
\draw (0, 0) -- (1, 0);
\draw (0, 0) -- (0, 1);
\draw (1, 1) -- (1, 0);
\draw (1, 1) -- (0, 1);
\node[rotate=90,font=\tiny] at (0.8, 0.5) {YZ};
\node[rotate=0,font=\tiny] at (0.5, 0.8) {D};
\node[rotate=90,font=\tiny] at (0.2, 0.5) {XW};
\node[rotate=0,font=\tiny] at (0.5, 0.2) {A};
\node at (1.5,.5) {$\mapsto$};
\node at (2.5,0.5) {\begin{tikzpicture}[scale=0.900000000000000]
\draw (0, 0) -- (1, 0);
\draw (0, 0) -- (0, 1);
\draw (1, 1) -- (1, 0);
\draw (1, 1) -- (0, 1);
\node[rotate=0,font=\tiny] at (0.8, 0.5) {Y};
\node[rotate=0,font=\tiny] at (0.5, 0.8) {B};
\node[rotate=0,font=\tiny] at (0.2, 0.5) {X};
\node[rotate=0,font=\tiny] at (0.5, 0.2) {A};
\draw (0, 1) -- (1, 1);
\draw (0, 1) -- (0, 2);
\draw (1, 2) -- (1, 1);
\draw (1, 2) -- (0, 2);
\node[rotate=0,font=\tiny] at (0.8, 1.5) {Z};
\node[rotate=0,font=\tiny] at (0.5, 1.8) {D};
\node[rotate=0,font=\tiny] at (0.2, 1.5) {W};
\node[rotate=0,font=\tiny] at (0.5, 1.2) {B};
\end{tikzpicture}};
\end{tikzpicture}
\end{center}
whether $i=1$ or $i=2$.
The reader may wonder how the substitution decides the color $Y$ (color $B$ if
$i=2$) from its input tiles.  The answer is that
Algorithm~\ref{alg:find-recognizable-sub-from-markers} is performing
a desubstitution. Therefore the two tiles sharing the vertical side with
letter $Y$ are known from the start and the algorithm just creates a new tile
$(Z,BD,X,AC)$ and claims that it will always get replaced by the two tiles with
shared edge with color $Y$.

\begin{remark}
In Algorithm~\ref{alg:find-recognizable-sub-from-markers}, the set of tiles
$\Scal$ can be computed with the help of transducers since
the fusion operation $\boxslash^i$ on Wang tiles can be seen as the composition
of transitions (see Equation~\eqref{eq:boxminus_with_transducer}).
\end{remark}

\section{Desubstitution from $\Omega_0$ to $\Omega_4$:
$\Omega_0\xleftarrow{\omega_0}
\Omega_1\xleftarrow{\omega_1}
\Omega_2\xleftarrow{\omega_2}
\Omega_3\xleftarrow{\omega_3'}
\Omega_4$}
\label{sec:desubstitution-0-to-4}

When Algorithm~\ref{alg:find-markers} returns a nonempty list of sets of
markers, then we may use Algorithm~\ref{alg:find-recognizable-sub-from-markers}
to find the substitution from the markers found. This allows to desubstitute
uniquely tilings in $\Omega_0$ into tilings in $\Omega_4$.

\begin{proposition}\label{prop:from-jeandel-rao-0-to-4}
    Let $\Omega_0$ be the Jeandel-Rao Wang shift.
There exist sets of Wang tiles 
$\T_1$, $\T_2$, $\T_3$ and $\T'_4$
together with their associated Wang shifts 
$\Omega_1$, $\Omega_2$, $\Omega_3$ and $\Omega_4$
and there exists a sequence of recognizable $2$-dimensional morphisms:
    \begin{equation*}
        \Omega_0 \xleftarrow{\omega_0}
        \Omega_1 \xleftarrow{\omega_1}
        \Omega_2 \xleftarrow{\omega_2}
        \Omega_3 \xleftarrow{\omega_3'}
        \Omega_4
    \end{equation*}
    that are onto up to a shift, i.e.,
    $\overline{\omega_i(\Omega_{i+1})}^\sigma=\Omega_{i}$
    for each $i\in\{0,1,2\}$ and
    $\overline{\omega_3'(\Omega_{4})}^\sigma=\Omega_{3}$.
\end{proposition}

\input{article2_omega_X.tex}
\input{article2_tiles_table.tex}
\input{article2_markers.tex}

\begin{proof}
    The proof is done by executing the function
    $\Call{FindMarkers}$ 
    followed by 
    $\Call{FindSubstitution}$
    and repeating this process four times in total.
    Each time $\Call{FindMarkers}$ 
    finds at least one subset of markers using a surrounding radius 
    of size at most 3.
    Thus using Theorem~\ref{thm:exist-homeo},
    we may find a desubstitution of tilings.

    The proof is done in SageMath \cite{sagemathv8.9}
    using \texttt{slabbe} optional package
    \cite{labbe_slabbe_0_6_2019}.
    First we define the Jeandel-Rao set of Wang tiles $\T_0$:
    {\footnotesize
\begin{verbatim}
sage: from slabbe import WangTileSet
sage: tiles = [(2,4,2,1), (2,2,2,0), (1,1,3,1), (1,2,3,2), (3,1,3,3), (0,1,3,1), 
....:          (0,0,0,1), (3,1,0,2), (0,2,1,2), (1,2,1,4), (3,3,1,2)]
sage: tiles = [[str(a) for a in tile] for tile in tiles]
sage: T0 = WangTileSet(tiles)
\end{verbatim}
    }
\[
\T_0 =\left\{\raisebox{-5mm}{\includegraphics{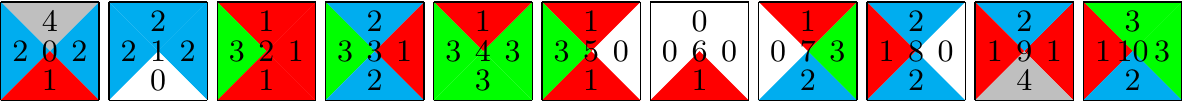}}\right\}
\]

We use Knuth's dancing links algorithm \cite{knuth_dancing_2000} because it is faster at this task than
MILP or SAT solvers.
We desubstitute $\T_0$:
    {\footnotesize
\begin{verbatim}
sage: T0.find_markers(i=2, radius=1, solver="dancing_links")
[[0, 1]]
sage: M0 = [0,1]
sage: T1,omega0 = T0.find_substitution(M=M0, i=2, side="left", 
....:                                  radius=1, solver="dancing_links")
\end{verbatim}
    }
and we obtain:
    \[
\omega_0:\Omega_1\to\Omega_0:\left\{ \omegaO\right.
    \]
    \[
\T_1 =\left\{\raisebox{-5mm}{\includegraphics{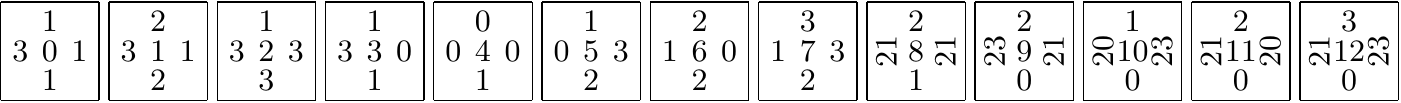}}\right\}
    \]
The above computation of $\omega_0$ and $\T_1$ confirms the observations we made
    earlier in Section~\ref{sec:markers}.

We desubstitute $\T_1$:
    {\footnotesize
\begin{verbatim}
sage: T1.find_markers(i=2, radius=1, solver="dancing_links")
[[8, 9, 10, 11, 12]]
sage: M1 = [8, 9, 10, 11, 12]
sage: T2,omega1 = T1.find_substitution(M=M1, i=2, side="left", 
....:                                  radius=1, solver="dancing_links")
\end{verbatim}
    }
and we obtain:
    \[
\omega_1:\Omega_2\to\Omega_1:\left\{ \omegaI\right.
    \]
    \[
\T_2 =\left\{\raisebox{-10mm}{\includegraphics{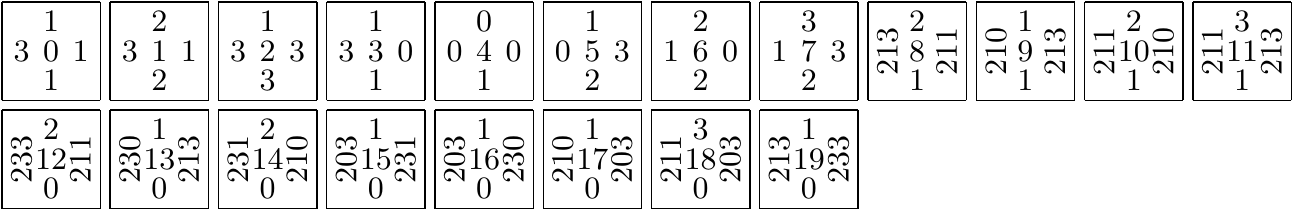}}\right\}
    \]

We desubstitute $\T_2$:
    {\footnotesize
\begin{verbatim}
sage: T2.find_markers(i=2, radius=2, solver="dancing_links")
[[8, 9, 10, 11, 12, 13, 14, 15, 16, 17, 18, 19]]
sage: M2 = [8, 9, 10, 11, 12, 13, 14, 15, 16, 17, 18, 19]
sage: T3,omega2 = T2.find_substitution(M=M2, i=2, side="left", 
....:                                  radius=2, solver="dancing_links")
\end{verbatim}
    }
and we obtain:
    \[
\omega_2:\Omega_3\to\Omega_2:\left\{ \omegaII\right.
    \]
    \[
\T_3 =\left\{\raisebox{-10mm}{\includegraphics{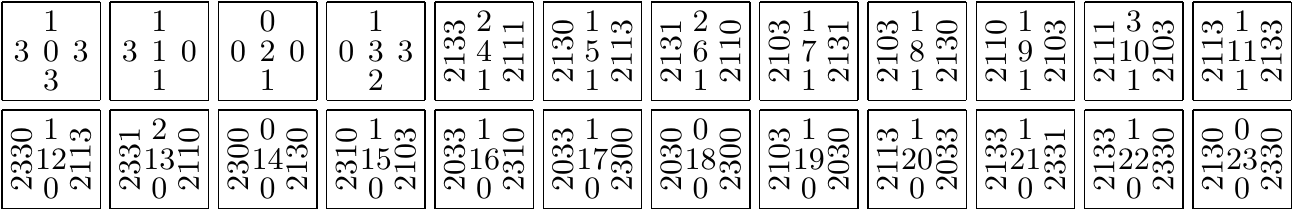}}\right\}
    \]

We desubstitute $\T_3$:
    {\footnotesize
\begin{verbatim}
sage: T3.find_markers(i=2, radius=3, solver="dancing_links")
[[0, 1, 2, 3]]
sage: M3 = [0, 1, 2, 3]
sage: T4p,omega3p = T3.find_substitution(M=M3, i=2, side="right", 
....:                                    radius=3, solver="dancing_links")
\end{verbatim}
}
and we obtain:
    \[
\omega_3':\Omega_4\to\Omega_3:\left\{ \omegaIIIp\right.
    \]
    \[
\T_4' =\left\{\raisebox{-15mm}{\includegraphics{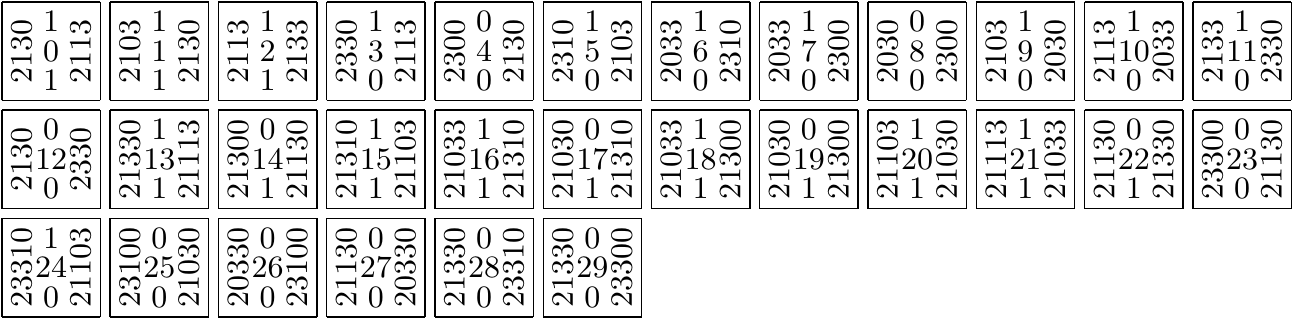}}\right\}
    \]

\end{proof}


The set of tiles $\T_4'$ contains 30 tiles. We show in
the next section that it contains 2 useless tiles, 
i.e., tiles that never appear in any tiling of $\Omega_4$,
so that we can remove those two tiles from $\T_4'$.

\section{Removing two useless tiles from $\T_4'$ to get $\T_4$}
\label{sec:removing-2-tiles-from-T4}

In this section, we state some observations on what cannot be done
with the Jeandel-Rao set of tiles. This allows to gain some initial local
intuitions on what is allowed and what is forbidden in Jeandel-Rao tilings but
also illustrates that not everything can be decided by using only a small
neighborhood.

\subsection{Some forbidden patterns in Jeandel-Rao tilings}

The next lemma shows easily by hand that some patterns are forbidden in
$\Omega_0$. The first two patterns were already proven impossible in
\cite[Proposition 1]{jeandel_aperiodic_2015}.

\begin{lemma}
    \label{lem:impossible-by-hand}
    No tiling of the plane with Jeandel-Rao set of tiles contains
    any of the patterns shown in Figure~\ref{fig:impossible}.
\end{lemma}

\begin{figure}[h]
\includegraphics[scale=.7]{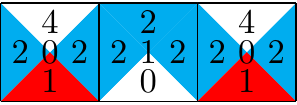}
\qquad
\includegraphics[scale=.7]{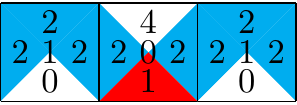}
\qquad
\includegraphics[scale=.7]{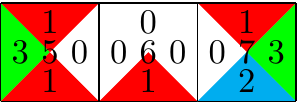}
\qquad
\includegraphics[scale=.7]{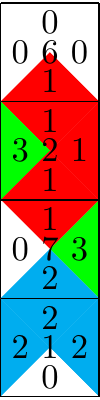}
\qquad
\includegraphics[scale=.7]{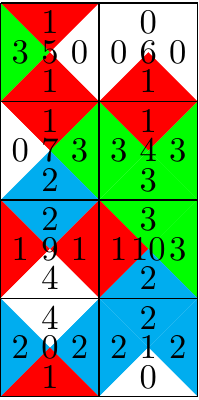}
    \caption{Those patterns are forbidden in any tiling of the plane using
    Jeandel-Rao set of tiles. The first two were proved impossible in Proposition~1
    of in \cite{jeandel_aperiodic_2015}.}
    \label{fig:impossible}
\end{figure}

\begin{proof}
    There is no way to find three tiles to put above (resp. below) the first
    (resp. second) pattern. Indeed, the two on the side are forced, but no
    tile can fit in the middle:
\begin{center}
\includegraphics[scale=.7]{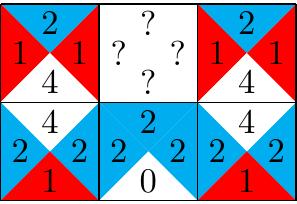},
\qquad
\includegraphics[scale=.7]{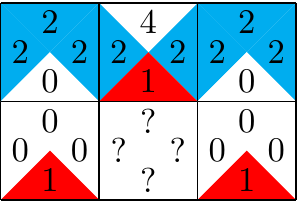}.
\end{center}
The only way to put three tiles above the third pattern is to use the first
pattern which we just have shown is forbidden:
\begin{center}
\includegraphics[scale=.7]{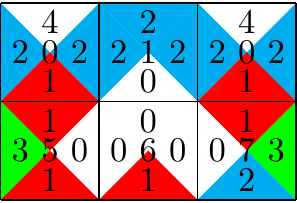}.
\end{center}
The only way to put four tiles to the left and to the right of the
fourth pattern is shown below and we remark the third forbidden pattern appears
on the top row:
\begin{center}
\includegraphics[scale=.7]{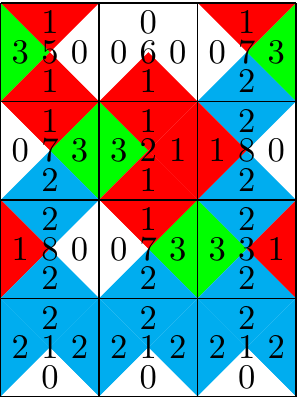}.
\end{center}
There is a unique way to tile the $1\times 4$ rectangle to the left of the
fifth considered pattern:
\begin{center}
    \includegraphics[scale=.7]{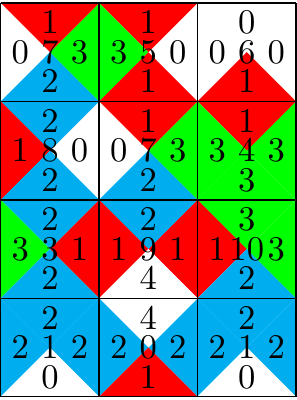}
\end{center}
    We observe that the row on the bottom is forbidden.
\end{proof}

\subsection{Jeandel-Rao transducer representation of $\T_4$}

   In \cite{jeandel_aperiodic_2015}, the set $\T_0$ of 11 Wang
   tiles is represented as a transducer made of two connected components $T_0$
   and $T_1$, see Figure~\ref{fig:transducer}.
   They observed that the strongly connected
   components of the product of transducers $T_1T_1$, $T_1T_0T_1$,
   $T_1T_0T_0T_1$ and $T_0T_0T_0T_0T_0$ are empty.  
   The emptiness of the products imply that the only possible tilings in
   $\Omega_0$ are made of horizontal strips of height 4 or 5 generated by the
   products $T_a=T_1T_0T_0T_0T_0$ and $T_b=T_1T_0T_0T_0$
   shown in Figure~\ref{fig:jeandel_rao_figure_7}.

   Such tilings indeed exist as proved by \cite{jeandel_aperiodic_2015}.
   Jeandel and Rao proved aperiodicity of $\T_0$ by proving that
   the sequence of horizontal strips of height $5$ and of height
   $4$ coded respectively by the transducers $T_a$ and $T_b$
   is exactly the set of factors of the Fibonacci word, i.e., the fixed point of
   the morphism $a\mapsto ab, b\mapsto a$ (see also
   Proposition~\ref{prop:Fibonacci}).

   As a consequence, every tiling in $\Omega_0$ can be desubstituted uniquely by the 31
   patterns of rectangular shape $1\times 4$ or $1\times 5$ associated to each
   of the transitions of the two transducers shown at
   Figure~\ref{fig:jeandel_rao_figure_7}.
   The reader may compare these 31 transitions with the set $\T_4'$ of 30 Wang
   tiles computed in the previous section and see that the transition
   $2030\xrightarrow{0|0} 2310$ does not correspond to a tile in~$\T_4'$.

\begin{figure}[h]
\begin{tikzpicture}[auto,yscale=1.3,xscale=2,>=latex]
    \node at (-1,1) {$T_b:$};
    \node[rectangle,draw] (2030) at (4, 0) {2030};
    \node[rectangle,draw] (2033) at (2, 0) {2033};
    \node[rectangle,draw] (2103) at (4, 2) {2103};
    \node[rectangle,draw] (2113) at (1, 0) {2113};
    \node[rectangle,draw] (2130) at (2, 2) {2130};
    \node[rectangle,draw] (2133) at (0, 0) {2133};
    \node[rectangle,draw] (2300) at (3, 1) {2300};
    \node[rectangle,draw] (2310) at (3, 0) {2310};
    \node[rectangle,draw] (2330) at (0, 2) {2330};
    \draw[->] (2030) -- node[swap] {0|0} (2300);
    \draw[->] (2033) -- node {0|1} (2300);
    \draw[->] (2033) -- node[swap] {0|1} (2310);
    \draw[->] (2103) -- node {0|1} (2030);
    \draw[->] (2103) -- node[swap] {1|1} (2130);
    \draw[->] (2113) -- node[swap] {0|1} (2033);
    \draw[->] (2113) -- node {1|1} (2133);
    \draw[->] (2130) -- node[swap] {1|1} (2113);
    \draw[->] (2130) -- node[swap] {0|0} (2330);
    \draw[->] (2133) -- node {0|1} (2330);
    \draw[->] (2300) -- node {0|0} (2130);
    \draw[->] (2310) -- node[pos=.7] {0|1} (2103);
    \draw[->] (2330) -- node {0|1} (2113);
    \draw[->,dashed] (2030) -- node {0|0} (2310); 
\end{tikzpicture}
\begin{tikzpicture}[auto,scale=2,>=latex]
    \node at (-1,1) {$T_a:$};
    \node[draw,rectangle] (20330) at (2, 2) {20330};
    \node[draw,rectangle] (21030) at (5, 2) {21030};
    \node[draw,rectangle] (21033) at (3, 1) {21033};
    \node[draw,rectangle] (21103) at (5, 0) {21103};
    \node[draw,rectangle] (21113) at (2, 1) {21113};
    \node[draw,rectangle] (21130) at (0, 2) {21130};
    \node[draw,rectangle] (21300) at (4, 1.4) {21300};
    \node[draw,rectangle] (21310) at (4, 0.6) {21310};
    \node[draw,rectangle] (21330) at (1, 1) {21330};
    \node[draw,rectangle] (23100) at (3, 2) {23100};
    \node[draw,rectangle] (23300) at (0, 0) {23300};
    \node[draw,dashed,rectangle] (23310) at (1, 0) {23310}; 
    \draw[->] (20330) -- node {0|0} (23100);
    \draw[->] (21030) -- node[swap] {1|0} (21300);
    \draw[->] (21030) -- node {1|0} (21310);
    \draw[->] (21033) -- node {1|1} (21300);
    \draw[->] (21033) -- node {1|1} (21310);
    \draw[->] (21103) -- node[swap] {1|1} (21030);
    \draw[->] (21113) -- node {1|1} (21033);
    \draw[->] (21130) -- node {0|0} (20330);
    \draw[->] (21130) -- node {1|0} (21330);
    \draw[->] (21300) -- node {1|0} (21130);
    \draw[->] (21310) -- node {1|1} (21103);
    \draw[->] (21330) -- node {1|1} (21113);
    \draw[->] (21330) -- node[swap] {0|0} (23300);
    \draw[->] (23100) -- node {0|0} (21030);
    \draw[->] (23300) -- node {0|0} (21130);
    \draw[->,dashed] (21330) -- node {0|0} (23310); 
    \draw[->,dashed] (23310) -- node {0|1} (21103); 
\end{tikzpicture}
    \caption{
        The strongly connected components of $T_a=T_1T_0T_0T_0T_0$ (below) and
        $T_b=T_1T_0T_0T_0$ (above) after deletion of sink and source
        transitions. This is a reproduction of Figure 7 (b) page 18 of
        Jeandel-Rao preprint.
        The edges 
        $2030\xrightarrow{0|0} 2310$
        and
        $21330\xrightarrow{0|0} 23310\xrightarrow{0|1} 21103$
        are dashed because we show that the associated patterns are forbidden in
        any tiling of the plane (Lemma~\ref{lem:impossible-by-hand} and
        Proposition~\ref{prop:rao-impossible-52}). 
        Each solid transition corresponds to a Wang tile in the set $\T_4$.}
    \label{fig:jeandel_rao_figure_7}
\end{figure}
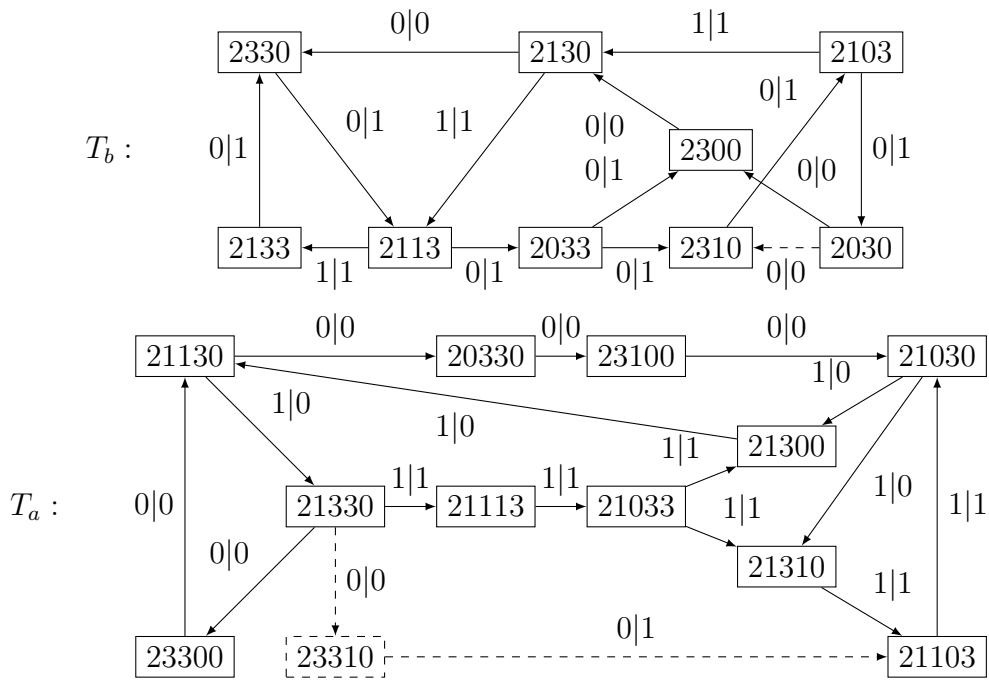

\begin{figure}[h]
\begin{center}
    \includegraphics[scale=.7]{article2_impossible_2030_2310.pdf}
\qquad
\begin{tikzpicture}[scale=1.2]
\tikzstyle{every node}=[font=\footnotesize]
    \node at (0.5,0.5) {$\varnothing$};
\draw (0, 0) -- (1, 0);
\draw (0, 0) -- (0, 1);
\draw (1, 1) -- (1, 0);
\draw (1, 1) -- (0, 1);
\node[rotate=90] at (0.8, 0.5) {2310};
\node[rotate=0] at (0.5, 0.8) {0};
\node[rotate=90] at (0.2, 0.5) {2030};
\node[rotate=0] at (0.5, 0.2) {0};
\end{tikzpicture}
\qquad
    \includegraphics[scale=.7]{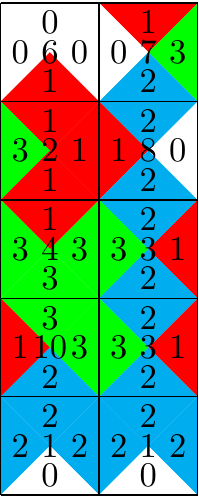}
\qquad
\begin{tikzpicture}[scale=1.2]
\tikzstyle{every node}=[font=\footnotesize]
    \node at (0.5,0.5) {28};
\draw (0, 0) -- (1, 0);
\draw (0, 0) -- (0, 1);
\draw (1, 1) -- (1, 0);
\draw (1, 1) -- (0, 1);
\node[rotate=90] at (0.8, 0.5) {23310};
\node[rotate=0] at (0.5, 0.8) {0};
\node[rotate=90] at (0.2, 0.5) {21330};
\node[rotate=0] at (0.5, 0.2) {0};
\begin{scope}[xshift=1cm]
    \node at (0.5,0.5) {24};
\draw (0, 0) -- (1, 0);
\draw (0, 0) -- (0, 1);
\draw (1, 1) -- (1, 0);
\draw (1, 1) -- (0, 1);
\node[rotate=90] at (0.8, 0.5) {21103};
\node[rotate=0] at (0.5, 0.8) {1};
\node[rotate=90] at (0.2, 0.5) {23310};
\node[rotate=0] at (0.5, 0.2) {0};
\end{scope}
\end{tikzpicture}
\end{center}
    \caption{
    From left to right:
    the pattern in $\T_0$;
    the tile 
    associated to the transition $2030\xrightarrow{0|0} 2310$
    which does not appear in $\T_4'$;
    the pattern in $\T_0$
    and the two tiles in $\T_4'$
    associated to the consecutive transitions 
        $21330\xrightarrow{0|0}23310\xrightarrow{0|1}21103$.
    All these patterns are forbidden in $\Omega_0$ and in $\Omega_4$. 
    The proof for the former transition is easy and is done in
    Lemma~\ref{lem:impossible-by-hand}.
    The proof for the latter is impossible without a computer.}
\label{fig:21330-0|0-23310-0|1-21103}
\end{figure}

In this section we show that only 28 of the 31 transitions shown in
Figure~\ref{fig:jeandel_rao_figure_7} correspond to patterns that
appear in $\Omega_0$.
The patterns in
$\Omega_0$ and $\Omega_4$ associated to the three transitions that can be
removed are shown in Figure~\ref{fig:21330-0|0-23310-0|1-21103}.
The proof that the pattern associated to the transition $2030\xrightarrow{0|0}
2310$ does not appear in $\Omega_0$ is easy and is done in
Lemma~\ref{lem:impossible-by-hand}.
Another independent proof of this is
Proposition~\ref{prop:from-jeandel-rao-0-to-4} since the tile $(2310,0,2030,0)$
associated to the transition $2030\xrightarrow{0|0} 2310$ 
does not belong to $\T_4'$.

The proof that the other two transitions are useless is much more difficult
as some large valid surroundings of the pattern 
associated to $21330\xrightarrow{0|0}23310\xrightarrow{0|1}21103$
can be found (see Figure~\ref{fig:surrounding}), so we have to go far until we
reach a forbidden pattern.
Moreover Proposition~\ref{prop:from-jeandel-rao-0-to-4} 
fails to prove it since
the two tiles associated to the transitions
$21330\xrightarrow{0|0}23310\xrightarrow{0|1}21103$
belong to $\T_4'$.

\begin{figure}
\begin{center}
\includegraphics[width=.8\linewidth]{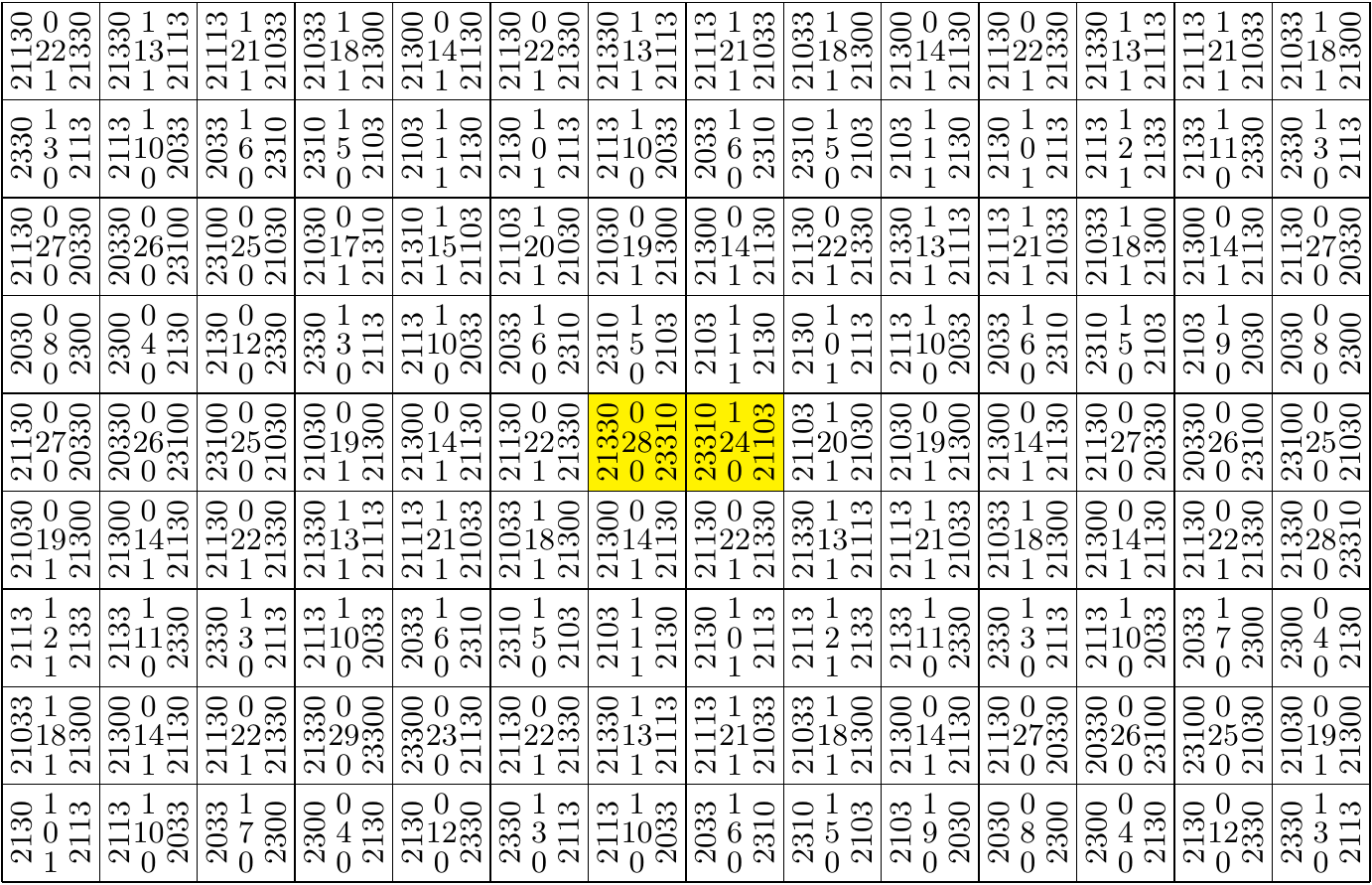}
\end{center}
\caption{A surrounding of height 9 and width 14 with tiles of $\T_4'$ around
the pattern $21330\xrightarrow{0|0}23310\xrightarrow{0|1}21103$.}
\label{fig:surrounding}
\end{figure}

\subsection{Removing two tiles from $\T_4'$}

The first proof of this result was provided by Michaël Rao using the transducer
approach during a visit of him at LaBRI, Bordeaux France (personal
communication, November 2017). We provide here an independent proof 
also done by computer exploration with SageMath \cite{sagemathv8.9}.


\begin{proposition}\label{prop:rao-impossible-52}
    (Michaël Rao, 2017)
    No tiling of the plane with Jeandel-Rao set of tiles contains the 
    pattern $21330\xrightarrow{0|0}23310\xrightarrow{0|1}21103$.
\end{proposition}

\begin{proof}
    To show that the pattern $23310\xrightarrow{0|1}21103$ is forbidden is
    enough as it must occur if and only if the other pattern
    $21330\xrightarrow{0|0}23310$ occurs.  We show that there exists no
    rectangle of width 71 and height 9 tiled with the 30 tiles of $\T_4'$
    with the tile $23310\xrightarrow{0|1}21103$ in the center of that rectangle.

    We do the proof 
    in SageMath \cite{sagemathv8.9} with the open source optional package
    \texttt{slabbe} \cite{labbe_slabbe_0_6_2019} developed by the author. 
    In that package,
    a Wang tiling problem is reduced to
    mixed integer linear program (MILP), 
    boolean satisfiability problem (SAT)
    or exact cover problem (solved with Knuth's dancing links algorithm).
    
    Below is the code to construct the problem. We 
    construct the solver searching for
    a tiling of the rectangle of size $71\times 9$ 
    tiled with the tiles from $\T_4'$
    with the tile $(21103, 1, 23310, 0)$
    in the center of that rectangle,
    that is, at position $(35,4)$:
    {\footnotesize
\begin{verbatim}
sage: T4p[24]
('21103', '1', '23310', '0')
sage: S = T4p.solver(width=71, height=9, preassigned_tiles={(35,4):24})
\end{verbatim}}
    Using a reduction of that problem to SAT
    using 19170 variables and 1078659 clauses,
    the Glucose SAT solver \cite{glucose} says there is no solution in less than 4 seconds:
    {\footnotesize
\begin{verbatim}
sage: %time S.has_solution(solver="glucose")
CPU times: user 2.02 s, sys: 104 ms, total: 2.12 s
Wall time: 3.93 s
False
\end{verbatim}}

    Using a reduction of the same problem to a MILP instance,
    using 19170 variables and 1838 constraints,
    the linear program is shown to have no solution using
    \texttt{Gurobi} solver \cite{gurobi} 
    in about 45 seconds on a normal 2018 desktop computer with 8 available
    cpus:
    {\footnotesize
\begin{verbatim}
sage: %time S.has_solution(solver="Gurobi")
CPU times: user 2min 42s, sys: 2.7 s, total: 2min 45s
Wall time: 44.7 s
False
\end{verbatim}}
\end{proof}


As a consequence, we have the following corollary:

\begin{corollary}\label{cor:can-remove-2-tile}
    The Wang shift $\Omega_4$ generated by $\T_4'$ is generated by a subset
    \begin{equation*}
    \T_4=\T_4'\setminus \left\{
\raisebox{-4mm}{
    \rm
\begin{tikzpicture}[scale=1.2]
\tikzstyle{every node}=[font=\footnotesize]
\draw (0, 0) -- (1, 0);
\draw (0, 0) -- (0, 1);
\draw (1, 1) -- (1, 0);
\draw (1, 1) -- (0, 1);
\node[rotate=90] at (0.8, 0.5) {23310};
\node[rotate=0] at (0.5, 0.8) {0};
\node[rotate=90] at (0.2, 0.5) {21330};
\node[rotate=0] at (0.5, 0.2) {0};
\end{tikzpicture},
\begin{tikzpicture}[scale=1.2]
\tikzstyle{every node}=[font=\footnotesize]
\draw (0, 0) -- (1, 0);
\draw (0, 0) -- (0, 1);
\draw (1, 1) -- (1, 0);
\draw (1, 1) -- (0, 1);
\node[rotate=90] at (0.8, 0.5) {21103};
\node[rotate=0] at (0.5, 0.8) {1};
\node[rotate=90] at (0.2, 0.5) {23310};
\node[rotate=0] at (0.5, 0.2) {0};
\end{tikzpicture}}
        \right\}.
    \end{equation*}
    There exists a one-to-one map $\iota:\T_4\to\T_4'$
    such that its extension to the Wang shift $\Omega_4$
    is the identity.
    The morphism $\omega_3:\Omega_4\to\Omega_3$ defined as
    $\omega_3=\omega_3'\circ\iota$
    is recognizable and onto up to a shift.
\end{corollary}

Both $\T_4$ and $\iota:\Omega_4\to\Omega_4$ are
shown below.

\[
    \iota:\T_4\to\T_4' \left\{\iotaIII\right.
\]
\[
    \T_4 =\left\{\raisebox{-15mm}{\includegraphics{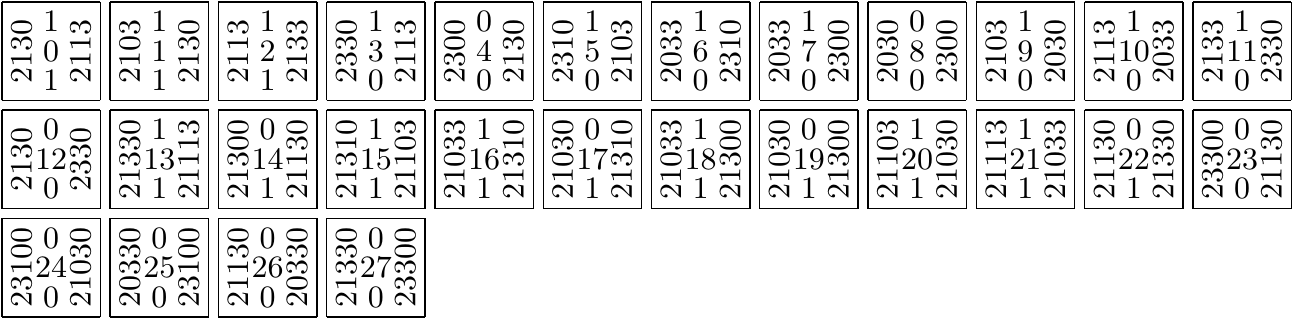}}\right\}
\]
The product morphism 
$\omega_0\,\omega_1\,\omega_2\,\omega_3:\Omega_4\to\Omega_0$
is below and is illustrated with tiles in Figure~\ref{fig:omega-4-to-0}.
\[
    \left\{\omegaOtoIV\right.
\]

The next natural step would be to find markers $M\subset\T_4$ but 
we face a problem: $\T_4$ has no markers.
Even worse: there are two problems. Each of these two distinct reasons is dealt
in the next two sections.

\part{Dealing with the troubles}\label{part:2}

\section{Adding decorations : $\Omega_4\xleftarrow{\jmath}\Omega_5$}
\label{sec:adding-decorations-jmath}

As mentioned, the set of tiles $\T_4$ has no markers. 
It has no markers in the direction $\be_2$ for a reason that is dealt in the next section.
It has no markers in the direction $\be_1$ because
there are tilings in $\Omega_4$ that have a horizontal biinfinite line
consisting of color 0 only (there are also some other tilings that have a
horizontal biinfinite line of 1's only).
We call such a line of a constant color a \emph{fault line}. 
The upper half-plane above a horizontal fault line can be translated
to the left or to the right keeping the validity of the tiling. The validity of this
sliding operation prevents the existence of markers in the direction $\be_1$ in $\T_4$ as
it naturally breaks the column-structure imposed by such marker tiles.

Horizontal fault lines also happens for Jeandel-Rao tilings themselves.
See Figure~\ref{fig:jeandel-rao-some-tiling} which illustrates a tiling in $\Omega_0$ containing a fault line of 0's.  
This implies that the set $\T_0$ of Jeandel-Rao tiles also has no markers in
the direction $\be_1$ although it has markers in the direction $\be_2$.
The reader may observe in Figure~\ref{fig:jeandel-rao-some-tiling} that sliding
the upper half-plane above the horizontal fault line of 0's breaks the structure of supertiles.
Such tilings are still valid Wang tilings, they belong to the Jeandel-Rao Wang shift
but they can not be desubstituted further since they do not belong to $X_0$ (see Theorem~\ref{thm:main-result-in-simple-terms}). Moreover, we believe that $\Omega_0\setminus X_0$
contains nothing else than those obtained by sliding along the fault line
(see Conjecture~\ref{conj:fault line}).
This is also why the Jeandel-Rao Wang
shift is not minimal and this is developed later in Part~\ref{part:4} of this contribution.

Fault lines happen also with Robinson tilings and this is the
reason why the Robinson tilings are not minimal. In \cite{GAHLER2012627}, the
authors add decorations on Robinson tiles to avoid fault lines and they described
the unique minimal subshift of Robinson tilings.

In this section, to avoid fault lines in $\Omega_4$, we add decorations on the
bottom and top edges of tiles of $\T_4$. We replace some 0's by 6's and some
1's by 5's.  More precisely, we define a new set $\T_5$ of 29 Wang tiles as
\begin{equation*}
    \T_5 = \left\{\raisebox{-15mm}{\includegraphics{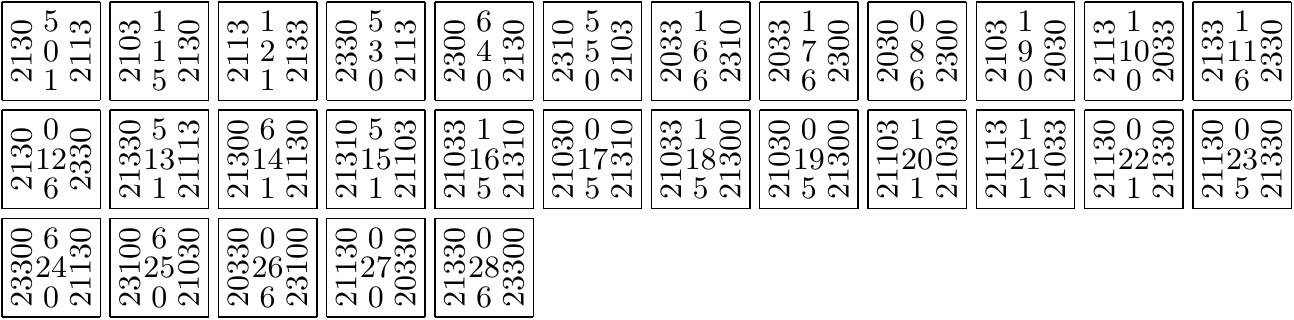}}\right\}
\end{equation*}
together with its Wang shift $\Omega_5=\Omega_{\T_5}$.
We define the following projection of colors:
\begin{equation*}
\begin{array}{rcl}
    \jmath:\{0,1,5,6\} & \to & \{0,1\}\\
x & \mapsto &
\begin{cases}
0  & \text{ if } x \in\{0,6\},\\
1  & \text{ if } x \in\{1,5\}.
\end{cases}
\end{array}
\end{equation*}
When $\jmath$ is applied to the bottom and top colors of tiles of $\T_5$, 
we obtain another map (that we also denote by $\jmath$)
\begin{equation*}
\begin{array}{rccc}
    \jmath:&\T_5 & \to & \T_4\\
    &\tile{Y}{A}{X}{C} & \mapsto &
\tile{Y}{$\jmath$(A)}{X}{$\jmath$(C)}.
\end{array}
\end{equation*}
which maps the tiles $\T_5$ onto the tiles $\T_4$.
The map $\jmath$ is not one-to-one on the tiles of $\T_5$ as
two distinct tiles are mapped to the same:
    \begin{equation}\label{eq:pi_not_injective}
        \jmath\left(
\raisebox{-5mm}{
    \rm
\begin{tikzpicture}[scale=1.2]
\tikzstyle{every node}=[font=\footnotesize]
\draw (0, 0) -- (1, 0);
\draw (0, 0) -- (0, 1);
\draw (1, 1) -- (1, 0);
\draw (1, 1) -- (0, 1);
\node at (.5,.5) {23};
\node[rotate=90] at (0.8, 0.5) {21330};
\node[rotate=0] at (0.5, 0.8) {0};
\node[rotate=90] at (0.2, 0.5) {21130};
\node[rotate=0] at (0.5, 0.2) {5};
\end{tikzpicture}}\;\right)
=
\raisebox{-5mm}{
\begin{tikzpicture}[scale=1.2]
\tikzstyle{every node}=[font=\footnotesize]
\draw (0, 0) -- (1, 0);
\draw (0, 0) -- (0, 1);
\draw (1, 1) -- (1, 0);
\draw (1, 1) -- (0, 1);
\node at (.5,.5) {22};
\node[rotate=90] at (0.8, 0.5) {21330};
\node[rotate=0] at (0.5, 0.8) {0};
\node[rotate=90] at (0.2, 0.5) {21130};
\node[rotate=0] at (0.5, 0.2) {1};
\end{tikzpicture}}
=
        \jmath\left(
\raisebox{-5mm}{
\begin{tikzpicture}[scale=1.2]
\tikzstyle{every node}=[font=\footnotesize]
\draw (0, 0) -- (1, 0);
\draw (0, 0) -- (0, 1);
\draw (1, 1) -- (1, 0);
\draw (1, 1) -- (0, 1);
\node at (.5,.5) {22};
\node[rotate=90] at (0.8, 0.5) {21330};
\node[rotate=0] at (0.5, 0.8) {0};
\node[rotate=90] at (0.2, 0.5) {21130};
\node[rotate=0] at (0.5, 0.2) {1};
\end{tikzpicture}} \;\right).
\end{equation}
But $\jmath$ defines a one-to-one $2$-dimensional morphism on tilings of
$\Omega_5$ given by
\[
    \jmath:\Omega_5\to\Omega_4:\left\{\omegaIV\right.
\]

\begin{proposition}\label{prop:jmath-5-to-4-is-one-to-one-continuous}
    The $2$-dimensional morphism $\jmath:\Omega_5\to\Omega_4$ is an embedding
    and is a topological conjugacy onto its image $X_4=\jmath(\Omega_5)$.
\end{proposition}

We recall that an embedding commutes with the $\Z^2$-actions, i.e., the shift
maps $\sigma$.

\begin{proof}
    First we prove that the map $\jmath:\Omega_5\to\Omega_4$ is well-defined.
    Indeed, we check that for each tile $u\in\T_5$, we have $\jmath(u)\in\T_4$.
    Moreover, if tiles $u,v\in\T_5$ can be adjacent in a tiling in $\Omega_5$,
    then $\jmath(u)$, $\jmath(v)$ can be adjacent in a tiling of $\Omega_4$.
    Therefore, if $x\in\Omega_5$, then $\jmath(x)\in\Omega_4$ which proves that
    the map $\jmath:\Omega_5\to\Omega_4$ is well-defined.

    Now we prove that the map $\jmath:\Omega_5\to\Omega_4$ is one-to-one.
    Suppose that $x,y\in\Omega_5$ are such that $\jmath(x)=\jmath(y)$.
    Let $z\in\Omega_4$ be the tiling such that $z=\jmath(x)=\jmath(y)$.
    Let $P\subset\Z^2$ be the positions of tiles in $z$ that have more than one
    preimage under $\jmath$.
    Notice that the only two tiles of $\T_5$ that are mapped to the same tile
    in $\T_4$ are the ones given at Equation~\eqref{eq:pi_not_injective}.
    So we have:
    \begin{equation*}
        P 
        = z^{-1}\left(\{t\in\T_4\mid 
        \#\jmath^{-1}(t)>1
        \}\right)
        = z^{-1}\left(
                \raisebox{-5mm}{
                \begin{tikzpicture}[scale=1.2]
                \tikzstyle{every node}=[font=\footnotesize]
                \draw (0, 0) -- (1, 0);
                \draw (0, 0) -- (0, 1);
                \draw (1, 1) -- (1, 0);
                \draw (1, 1) -- (0, 1);
                \node[rotate=90] at (0.8, 0.5) {21330};
                \node[rotate=0] at (0.5, 0.8) {0};
                \node[rotate=90] at (0.2, 0.5) {21130};
                \node[rotate=0] at (0.5, 0.2) {1};
                \end{tikzpicture}}
            \;\right).
    \end{equation*}
    By definition, we have $x_p = y_p$ for every $p\in\Z^2\setminus P$.
    Let $p\in P$ and suppose by contradiction that $x_p\neq y_p$.
    From Equation~\eqref{eq:pi_not_injective}, this means that the bottom
    colors of $x_p$ and $y_p$ are distinct:
    \[
        \scbottom(x_p)
        \neq
        \scbottom(y_p).
    \]
    This also means that the top colors of the tile just below are:
    \[
        \{\sctop(x_{p-\be_2}), \sctop(y_{p-\be_2})\}=
        \{\scbottom(x_p), \scbottom(y_p)\}=\{1,5\}.
    \]
    Therefore we have that $p-\be_2\notin P$ or otherwise
        $\{\sctop(x_{p-\be_2}), \sctop(y_{p-\be_2})\}=\{0\}$.
    This means that $x_{p-\be_2} = y_{p-\be_2}$.
    Therefore the top colors of $x_{p-\be_2} = y_{p-\be_2}$ are the same which
    implies that the bottom colors of $x_p$ and $y_p$ are the same
    \[
        \scbottom(x_p)
        =
        \sctop(x_{p-\be_2})
        =
        \sctop(y_{p-\be_2})
        =
        \scbottom(y_p)
    \]
    which is a contradiction. This proves that $\jmath:\Omega_5\to\Omega_4$ is
    one-to-one.

    We prove that $\jmath$ is continuous.
    If $x,x'\in\Omega_5$ then
    \begin{equation*}
        \delta(\jmath(x), \jmath(x')) = \delta(x,x').
    \end{equation*}
    Thus if $(x_n)_{n\in\N}$ is a sequence of tilings $x_n\in\Omega_5$ such
    that $x=\lim_{n\to\infty}x_n$ exists, then
    $\jmath(x)=\lim_{n\to\infty}\jmath(x_n)$ and $\jmath$ is continuous.
\end{proof}

The map $\jmath:\Omega_5\to\Omega_4$ is not onto but we believe it is $\mu$-almost
the case where $\mu$ is any shift invariant probability measure on $\Omega_4$.
In Section~\ref{sec:fault lines}
we provide an example of a tiling $x\in\Omega_4\setminus X_4$
and we explain why we believe that
$\Omega_4\setminus X_4$ has measure 0.
The code to construct the set of Wang tiles $\T_5$ introduced in this section is below:
    {\footnotesize
\begin{verbatim}
sage: tiles5 = [(2113, 5, 2130, 1), (2130, 1, 2103, 5), (2133, 1, 2113, 1),
....: (2113, 5, 2330, 0), (2130, 6, 2300, 0), (2103, 5, 2310, 0),
....: (2310, 1, 2033, 6), (2300, 1, 2033, 6), (2300, 0, 2030, 6),
....: (2030, 1, 2103, 0), (2033, 1, 2113, 0), (2330, 1, 2133, 6),
....: (2330, 0, 2130, 6), (21113, 5, 21330, 1), (21130, 6, 21300, 1),
....: (21103, 5, 21310, 1), (21310, 1, 21033, 5), (21310, 0, 21030, 5),
....: (21300, 1, 21033, 5), (21300, 0, 21030, 5), (21030, 1, 21103, 1),
....: (21033, 1, 21113, 1), (21330, 0, 21130, 1), (21330, 0, 21130, 5),
....: (21130, 6, 23300, 0), (21030, 6, 23100, 0), (23100, 0, 20330, 6),
....: (20330, 0, 21130, 0), (23300, 0, 21330, 6)]
sage: T5 = WangTileSet([[str(a) for a in tile] for tile in tiles5])
\end{verbatim}}

We finish this section by a proof that the image of $\jmath$ is subshift of
finite type. This result will be helpful in a future work on $X_0$ but its
proof is more easily done here.

\begin{lemma}
    $X_0$, $X_1$, $X_2$, $X_3$ and $X_4$ are subshifts of finite type.
\end{lemma}

\begin{proof}
    We want to show that $X_4=\SFT(G)$ where $G$ is the
    following set of forbidden dominoes
    \[
        G = \left(\left(\T_4\odot^1\T_4\right)\cup
            \left(\T_4\odot^2\T_4\right)\right)
            \setminus
            \jmath(D)
    \]
    and $D$ is the set of dominoes of tiles in $\T_5$ that allows some
    surrounding:
    \[
        D = \left\{u\odot^iv \in\T_5^{*^2}\mid 
           i\in\{1,2\} 
           \text{ and } u\odot^iv \text{
           admits a surrounding of radius $3$ with tiles in $\T_5$}\right\}.
    \]
    The set $D$ contains 37 horizontal dominoes and 75 vertical dominoes as
    shown in the computation below.
{\footnotesize
\begin{verbatim}
    sage: D_horizontal = T5.dominoes_with_surrounding(i=1, radius=3, solver="dancing_links")
    sage: len(D_horizontal)
    37
    sage: D_vertical = T5.dominoes_with_surrounding(i=2, radius=3, solver="dancing_links")
    sage: len(D_vertical)
    75
\end{verbatim}}

    We show that $\SFT(G)\supseteq X_4$.
    Since $X_4=\jmath(\Omega_5)$, let $x\in\Omega_5$.
    Let $a\odot^i b$ be a domino appearing in $\jmath(x)$.
    There exists some domino $u\odot^i v$ appearing in $x$
    such that $a\odot^i b=\jmath(u\odot^i v)$.
    Thus $u\odot^i v\in D$ and $a\odot^i b=\jmath(u\odot^i v)\in\jmath(D)$.
    Therefore $a\odot^i b\notin G$.
    Thus $\jmath(x)\in\SFT(G)$.

    We show that $\SFT(G)\subseteq X_4$.
    Let us identify the tiles in $\T_4$ with their indices
    in $\llbracket 0, 27\rrbracket$
    and the tiles in $\T_5$ with their indices
    in $\llbracket 0, 28\rrbracket$.
    Recall that the map $\jmath$ is one-to-one everywhere on $\T_5$ except
    tiles 22 and 23 that are mapped to the tile 22.
Note that among the vertical dominoes in $D$, the only one where the top tile
is 22 or 23 are
$7 \odot^2 22$,
$18\odot^2 22$,
$0\odot^2 23$,
$3 \odot^2 23$ and 
$13\odot^2 23$:
{\footnotesize
\begin{verbatim}
    sage: sorted((u,v) for (u,v) in D_vertical if v in [22,23])
    [(0, 23), (3, 23), (7, 22), (13, 23), (18, 22)]
\end{verbatim}}
\noindent
    Let $x\in\SFT(G)\subset\llbracket 0, 27\rrbracket^{\Z^2}$.
    We want to construct $y\in\Omega_5$ such that $x=\jmath(y)$.
    Let $y\in\T_5^{\Z^2}=\llbracket 0, 28\rrbracket^{\Z^2}$
    such that for every $p\in\Z^2$ we have
    \[
    y_p = 
    \begin{cases}
    \jmath^{-1}(x_p) & \text{if}\quad x_p\neq 22,\\
    22   & \text{if}\quad x_p=22 \text{ and } x_{p-\be_2}\in\{7,18\},\\
    23   & \text{if}\quad x_p=22 \text{ and } x_{p-\be_2}\in\{0,3,13\}.
    \end{cases}
    \]
    By definition, we have $x=\jmath(y)$.
    Suppose by contradiction that $y\notin\Omega_5$.
    This means that $y$ contains a domino $a\odot^i b$ 
    which is not valid, in particular $a\odot^i b\notin D$.
    Then $\jmath(a\odot^i b)$ is a domino which occurs in $x\in\SFT(G)$. 
    Thus $\jmath(a\odot^i b)\notin G$ which means 
    $\jmath(a\odot^i b)\in\jmath(D)$.
    Therefore there exists $c\odot^i d\in D$ such that
    $\jmath(a\odot^i b)=\jmath(c\odot^i d)$.
    The equality $a\odot^i b=c\odot^i d$ leads to a contradiction.
    Then $\{a,c\}=\{22,23\}$ or $\{b,d\}=\{22,23\}$.
    The fact that $c\odot^2 d\in D$ and $a\odot^2 b\notin D$
    and since only the bottom edge of tiles $22$ and $23$ in $\T_5$ differ
    implies that $i=2$, $\{b,d\}=\{22,23\}$.
    Tiles $\{22,23\}$ are not adjacent in $y$ since tile $22$ is not
    adjacent in $x$. It implies that $\{a,c\}\cap\{22,23\}=\varnothing$.
    Thus $\jmath(a)=\jmath(c)$ implies that $a=c$.
    Since $a\odot^2 b$ is not valid,
    we must have $b=23$ and $d=22$.
    We reach a contradiction because
    all occurrences of $a\odot^2 23$ in $y$
    are valid since $a\in\jmath^{-1}(\{0,3,13\})=\{0,3,13\}$.

    From the recognizability of $\omega_0\omega_1\omega_2\omega_3$,
    it follows that $X_0=\SFT(\omega_0\omega_1\omega_2\omega_3(G)\cup H)$
    where
    \[
        H = 
        \left(\left(\T_0\odot^1\T_0\right)\cup
            \left(\T_0\odot^2\T_0\right)\right)
            \setminus
        \left\{u\odot^iv \mid 
        u,v\in\T_0,
        u\odot^iv 
           \text{ is valid and }
           i\in\{1,2\} 
               \right\}.
    \]
    Thus $X_0$ is a subshift of finite type and similarly for intermediate
    $X_1$, $X_2$ and $X_3$.
\end{proof}

\section{A shear conjugacy: $\Omega_5\xleftarrow{\eta}\Omega_6$}
\label{sec:shear-eta}

In the previous section, we added decorations to avoid horizontal fault lines
but this is not enough.
There are still no subset of
markers in $\T_5$ neither in the direction $\be_1$ nor in the direction $\be_2$. 
But the tilings in $\Omega_5$ have a particular property
illustrated in Figure~\ref{fig:tiling_with_T5}. On each line of
slope 1 in a tiling in $\Omega_5$, we see tiles from a subset $M\subset\T_5$.
To deal with this case, we need to adapt the approach.

\begin{figure}[h]
\begin{center}
    \includegraphics{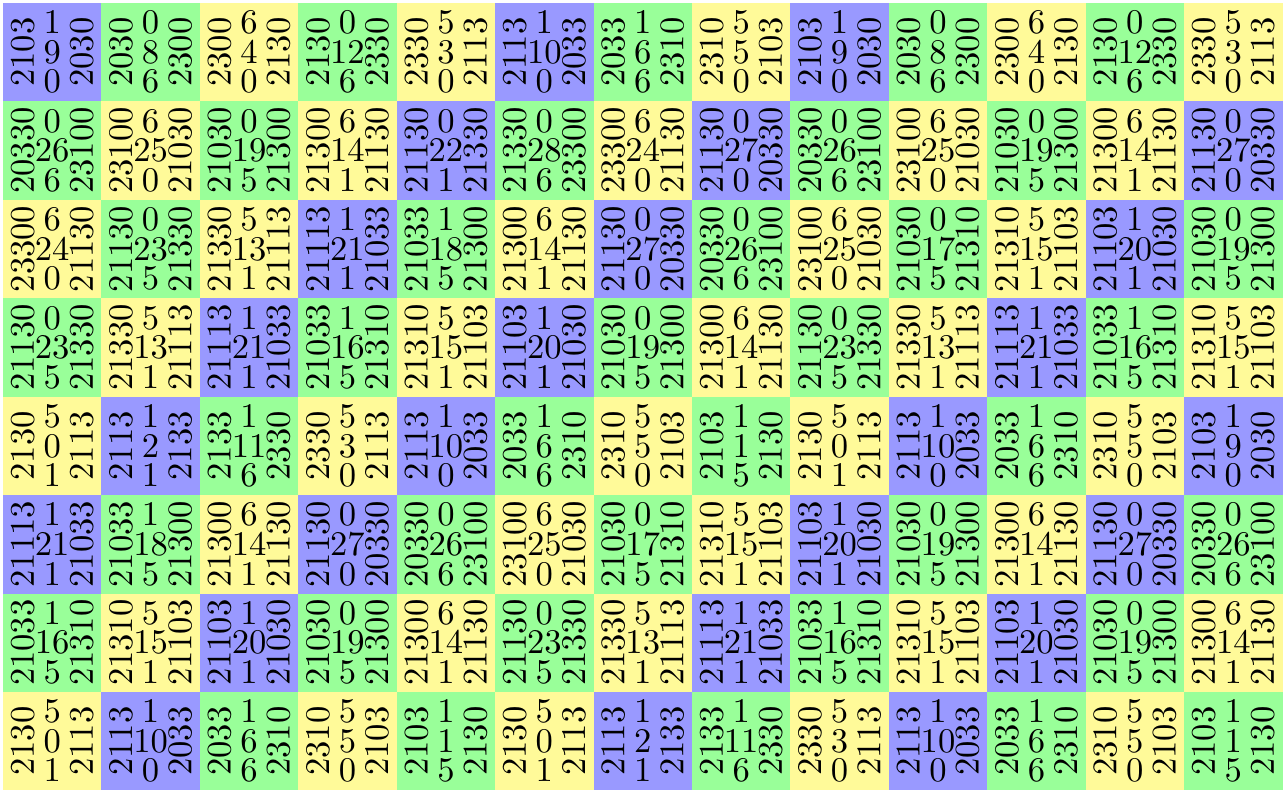}
\end{center}
    \caption{
        In a tiling in $\Omega_5$, all tiles appear in any chosen vertical
        column, so one can not find markers in direction $\be_1$ nor $\be_2$ to
        desubstitute those tilings. But the tiles appearing on a line of slope
        1 belong to either of the following three sets:
        $\{1, 6, 7, 8, 11, 12, 16, 17, 18, 19, 23, 26, 28\}$ (yellow),
        $\{0, 3, 4, 5, 13, 14, 15, 24, 25\}$ (green),
        $\{2, 9, 10, 20, 21, 22, 27\}$ (blue).
        These are \emph{kind of diagonal} markers.}
    \label{fig:tiling_with_T5}
\end{figure}

One possibility is to generalize the notion of markers to directions other than
$\be_1$ and $\be_2$ including markers appearing on lines of slopes $1$, $-1$
and possibly other rational slopes. 
A bad consequence of this option is that we need to
adapt all subsequent steps including Theorem~\ref{thm:exist-homeo} and
Algorithm~\ref{alg:find-recognizable-sub-from-markers}.
Their modification is possible but the resulting new algorithms will be more
complicated as they need to compute patches of size at least $2\times 2$ 
with their surroundings. If possible, we want to avoid to do that.
Instead, we change bijectively the set of tiles $\T_5$ so that the tilings
with the new tiles $\T_6$ are the image under the application of a shear matrix
on the tilings with the original tiles.

In this section, we introduce an operation on any set of Wang tiles which has
the effect of the application of a shear matrix 
$\left(\begin{smallmatrix}
        1 & 1 \\
        0 & 1
\end{smallmatrix}\right)$
on the associated tilings. Let
$\T$ be a set of Wang tiles and
\begin{equation*}
    \textsc{Dominoes}^i(\T) = \{(s,t)\in\T^2 \mid 
    \text{ the pattern } s\odot^i t 
    \text{ appears in some tilings of }\Omega_\T\}
\end{equation*}
be the set of allowed dominoes in the direction $i\in\{1,2\}$.
We define the map $\theta$ as the fusion of some hidden flat tile that appears
on top of two adjacent tiles (see
Figure~\ref{fig:stepbystepshear}):
\begin{equation*}
\begin{array}{rccl}
    \theta:& \textsc{Dominoes}^1(\T) & \to & \Scal\\
    &\left(\tile{Y}{A}{X}{C}, \tile{Z}{B}{Y}{D}\;\right)
    & \mapsto &
    \tilerot{YB}{B}{XA}{C}
    =
    \tile{Y}{A}{X}{C}
    \boxslash^2
    \tile{B}{B}{A}{A}
\end{array}
\end{equation*}
where $\Scal=\theta(\textsc{Dominoes}^1(\T))$ is the image of $\theta$.

\begin{figure}
\begin{tikzpicture}[auto,scale=0.9]
\tikzstyle{every node}=[font=\tiny]
\newcommand\image[1]{
\begin{tikzpicture}[scale=0.9]
\draw (0, 0) -- (1, 0);
\draw (0, 0) -- (0, 1);
\draw (1, 1) -- (1, 0);
\draw (1, 1) -- (0, 1);
\node at (0.8, 0.5) {Y};
\node at (0.5, 0.8) {A};
\node at (0.2, 0.5) {X};
\node at (0.5, 0.2) {C};
    \draw (1,1) -- node[swap] {B} ++ (#1:1) 
                -- node[swap] {B} ++ (-1,0)
                -- node[swap] {A} ++ (#1+180:1);
\end{tikzpicture}}
\node (A) at (0,0) {\image{0}};
\node (B) at (3,0) {\image{15}};
\node (C) at (6,0) {\image{45}};
\node (D) at (9,0) {\image{60}};
\node (E) at (12,0) {\image{90}};
\node (F) at (15,0) {\tilerot{YB}{B}{XA}{C}};
\draw[-to,very thick] (A) edge (B);
\draw[-to,very thick] (B) edge (C);
\draw[-to,very thick] (C) edge (D);
\draw[-to,very thick] (D) edge (E);
\draw[-to,very thick] (E) edge node {$\boxslash^2$} (F);
\end{tikzpicture}
\caption{A step by step visual explanation of the map
    $\theta:\textsc{Dominoes}^1(\T)\to\Scal$. A flat tile
    $(B,B,A,A)$ on top of two adjacent tiles is unfolded and merged
    with the left one with the fusion operation $\boxslash^2$.
    On tilings this operation is a shear conjugacy 
    $\Omega_\T\to\Omega_\Scal$ corresponding to the
    application of a shear transformation.}
    \label{fig:stepbystepshear}
\end{figure}
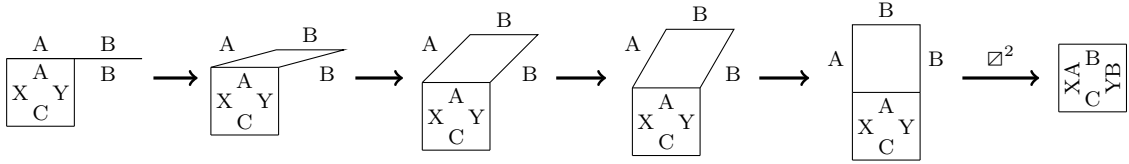

The map $\theta$ extends to tilings
as a map $\theta:\Omega_\T\to\Omega_\Scal$.
For any $x\in\Omega_\T$,
\begin{equation*}
\begin{array}{rccl}
    \theta(x):&\Z^2 & \to & \Scal\\
    & p & \mapsto & \theta(x_{Mp}, x_{Mp+\be_1})
\end{array}
\end{equation*}
where $M=\left(\begin{smallmatrix}
        1 & 1 \\
        0 & 1
\end{smallmatrix}\right)$.
The matrix $M$ is essential in the definition of $\theta$ because it assures
that $\theta(x)$ is a valid tiling of $\Z^2$ which is proved in the next lemma.
We also construct the map
\begin{equation*}
\begin{array}{rccc}
    \eta:& \Scal & \to & \T\\
    & \tilerot{YB}{B}{XA}{C}
    & \mapsto 
    & \tile{Y}{A}{X}{C}
\end{array}
\end{equation*}
The map $\eta$ extends to tilings 
    $\eta:\Omega_\Scal\to\Omega_\T$
as follows. For any $y\in\Omega_\Scal$,
\begin{equation}\label{eq:eta-on-tilings}
\begin{array}{rccl}
    \eta(y):&\Z^2 & \to & \T\\
    & p & \mapsto & \eta(y_{M^{-1}p})
\end{array}
\end{equation}
where $M=\left(\begin{smallmatrix}
        1 & 1 \\
        0 & 1
\end{smallmatrix}\right)$.
We show in the next lemma that $\eta$ is well-defined and is the inverse of $\theta$.

\begin{lemma}
    The map $\theta:\Omega_\T\to\Omega_\Scal$ is a shear conjugacy
    satisfying
    $\theta\circ\sigma^\bn=\sigma^{M^{-1}\bn}\circ\theta$
    with
$M=\left(\begin{smallmatrix}
        1 & 1 \\
        0 & 1
\end{smallmatrix}\right)$.
\end{lemma}

\begin{proof}
    First we show that $\theta$ is well-defined.
    Let $x\in\Omega_\T$ and $y=\theta(x)$.
    Let $p\in\Z^2$.
    We prove that colors on vertical edges match:
    \begin{align*}
    \scright(y_p) 
        &=\scright(x_{Mp})\cdot\sctop(x_{Mp+\be_1})\\
        &=\scleft(x_{Mp+\be_1})\cdot\sctop(x_{Mp+\be_1})\\
        &=\scleft(x_{M(p+\be_1)})\cdot\sctop(x_{M(p+\be_1)})
        =\scleft(y_{p+\be_1}).
    \end{align*}
    We prove that colors on horizontal edges match:
\begin{equation*}
    \sctop(y_p)
    = \sctop(x_{Mp+\be_1})
    =\scbottom(x_{Mp+\be_1+\be_2})
    =\scbottom(x_{M(p+\be_2)})
    =\scbottom(y_{p+\be_2}).
\end{equation*}

We prove that $\theta$ is one-to-one. Suppose that $x,x'\in\Omega_\T$ are two
tilings such that $x\neq x'$. Then there exists $p\in\Z^2$ such that
$x_p\neq x'_p$.  Let $y=\theta(x)$ and $y'=\theta(x')$. 
Since $\theta:(\textsc{Dominoes}^1(\T))\to\Scal$ is one-to-one on the first coordinate, 
    we have
    \begin{equation*}
        y_{M^{-1}p}
        = \theta(x_p, x_{p+\be_1})
        \neq \theta(x'_p, x'_{p+\be_1})
        = y'_{M^{-1}p}.
    \end{equation*}
    Therefore $\theta(x)\neq\theta(x')$.

We prove that the map $\eta:\Omega_\Scal\to\Omega_\T$ is well-defined.
To achieve this, we prove that colors on vertical edges match.
    Let $y\in\Omega_\Scal$, $x=\eta(y)$ and $p\in\Z^2$. We have
    \begin{align*}
    \scleft(x_p) \cdot \sctop(x_p) 
        &=\scleft(y_{M^{-1}p})\\
        &=\scright(y_{M^{-1}p-\be_1})\\
        &=\scright(y_{M^{-1}(p-\be_1)})\\
        &=\scright(x_{p-\be_1})\cdot\sctop(y_{M^{-1}(p-\be_1)})\\
        &=\scright(x_{p-\be_1})\cdot\scbottom(y_{M^{-1}(p-\be_1)+\be_2})\\
        &=\scright(x_{p-\be_1})\cdot\scbottom(y_{M^{-1}(p+\be_2)})\\
        &=\scright(x_{p-\be_1})\cdot\scbottom(x_{p+\be_2})
    \end{align*}
    from which we deduce
    \begin{align*}
        \scleft(x_p) = \scright(x_{p-\be_1})
        \qquad \text{ and } \qquad
        \sctop(x_p) = \scbottom(x_{p+\be_2}).
    \end{align*}
    Thus $x=\eta(y)\in\Omega_\T$ is a valid Wang tiling.

We now prove that $\theta$ is onto. Let $y\in\Omega_\Scal$
and take $x=\eta(y)\in\Omega_\T$.
    We have $\theta(x)=y$ since
    \begin{equation*}
        \theta(x)_p = \theta(x_{Mp}, x_{Mp+\be_1})
                    = \theta(\eta(y_{M^{-1}Mp}), y_{M^{-1}(Mp+\be_1)})
                    = \theta(\eta(y_{p}), \eta(y_{p+\be_1}))
                    = y_p.
    \end{equation*}
    Therefore $\theta$ is onto and $\eta$ is the inverse of $\theta$.

    We prove that $\theta$ and $\eta$ are continuous.
    If $x,x'\in\Omega_\T$ and $y,y'\in\Omega_\Scal$, then
    \begin{equation*}
        \delta(\theta(x), \theta(x')) < \sqrt{\delta(x,x')}
        \qquad \text{ and } \qquad
        \delta(\eta(y), \eta(y')) < \sqrt{\delta(y,y')}.
    \end{equation*}
    Thus if $(x_n)_{n\in\N}$ is a sequence of tilings $x_n\in\Omega_\T$ such
    that $x=\lim_{n\to\infty}x_n$ exists, then
    $\theta(x)=\lim_{n\to\infty}\theta(x_n)$ and $\theta$ is continuous.
    Similarly, $\eta$ is continuous.

    We prove that $\theta$ is a shear conjugacy.
    If $x\in\Omega_\T$, $\bn\in\Z^2$ and $p\in\Z^2$, then
    \begin{align*}
        \left[
            \sigma^{M^{-1}\bn}\circ\theta(x)
            \right](p)
        &= \left[ \theta(x) \right](p+M^{-1}\bn)
         = \theta(x_{M(p+M^{-1}\bn)},
                  x_{M(p+M^{-1}\bn)+\be_1})\\
        &= \theta(x_{Mp+\bn},
                  x_{Mp+\bn+\be_1})
         = \theta((\sigma^{\bn}x)_{Mp},
                  (\sigma^{\bn}x)_{Mp+\be_1})\\
        &= \left[ \theta(\sigma^{\bn}x) \right](p)
        = \left[ \theta\circ \sigma^{\bn}(x) \right](p).
    \end{align*}
    Thus the map $\theta:\Omega_\T\to\Omega_\Scal$ is a shear conjugacy
    satisfying
    $\theta\circ\sigma^\bn=\sigma^{M^{-1}\bn}\circ\theta$
    with
$M=\left(\begin{smallmatrix}
        1 & 1 \\
        0 & 1
\end{smallmatrix}\right)$.
\end{proof}

\begin{corollary}\label{cor:eta-6-to-5-is-conjugacy}
    The map $\eta:\Omega_\Scal\to\Omega_\T$ is a shear conjugacy
    satisfying
    $\eta\circ\sigma^\bn=\sigma^{M\bn}\circ\eta$
    with
$M=\left(\begin{smallmatrix}
        1 & 1 \\
        0 & 1
\end{smallmatrix}\right)$.
\end{corollary}

Of course, as a consequence, the map $\eta:\Omega_\Scal\to\Omega_\T$ is also
recognizable and onto up to a shift.
The algorithm to compute the set of Wang tiles $\Scal$ and the shear conjugacy
$\eta:\Omega_\Scal\to\Omega_\T$ 
that shears tiling by the matrix
$\left(\begin{smallmatrix}
1 & 1 \\
0 & 1
\end{smallmatrix}\right)$
is in Algorithm~\ref{alg:shear}.

\begin{algorithm}
    \caption[Shear Wang tiling by some matrix.]{Shear Wang tilings by the matrix
$\left(\begin{smallmatrix}
1 & 1 \\
0 & 1
\end{smallmatrix}\right)$.}
    \label{alg:shear}
  \begin{algorithmic}[1]
    \Require $\T$ is a set of Wang tiles;
             $r\in\N$ is a surrounding radius.
      \Function{Shear}{$\T$, $r$}
        \State $D_1 \gets \left\{(u,v)\in\T^2\mid \text{ domino } u\odot^1v \text{
             admits a surrounding of radius $r$ with tiles in $\T$}\right\}$
        \State $\Scal \gets \varnothing$
        \ForAll{$(u,v) \in D_1$}
            \State $w\gets (\sctop(v),\sctop(v),\sctop(u),\sctop(u))$
            \State $\Scal \gets \Scal \cup \{u\boxbar w\}$
        \EndFor
        \State \Return $\Scal$, $\eta:\Omega_\Scal\to\Omega_\T:
        u\boxbar w
        \mapsto
        u \text{ for each }
        u\boxbar w \in \Scal$
      \EndFunction
      \Ensure $\Scal$ is a set of Wang tiles;
              $\eta:\Omega_\Scal\to\Omega_\T$ is a shear conjugacy.
  \end{algorithmic}
\end{algorithm}

Let $\T_6=\theta(\textsc{Dominoes}^i(\T_5))$ and
$\eta:\Omega_6\to\Omega_5$ be defined as in
Equation~\eqref{eq:eta-on-tilings}
where $\Omega_6=\Omega_{\T_6}$
and $\Omega_5=\Omega_{\T_5}$.
In SageMath, we compute
    {\footnotesize
\begin{verbatim}
sage: # the following takes 6s with dancing_links, 3min 12s with Glucose, 22s with Gurobi
sage: T6,eta = T5.shear(radius=2, solver="dancing_links") 
\end{verbatim}}
The map $\eta:\Omega_6\to\Omega_5$ 
is a bijection preserving the indices of tiles:
\[
    \eta:\Omega_6\to\Omega_5:\left\{\omegaV\right.
\]
and the resulting set of Wang tiles $\T_6$ is:
\[
    \T_6=\left\{\raisebox{-17mm}{\includegraphics{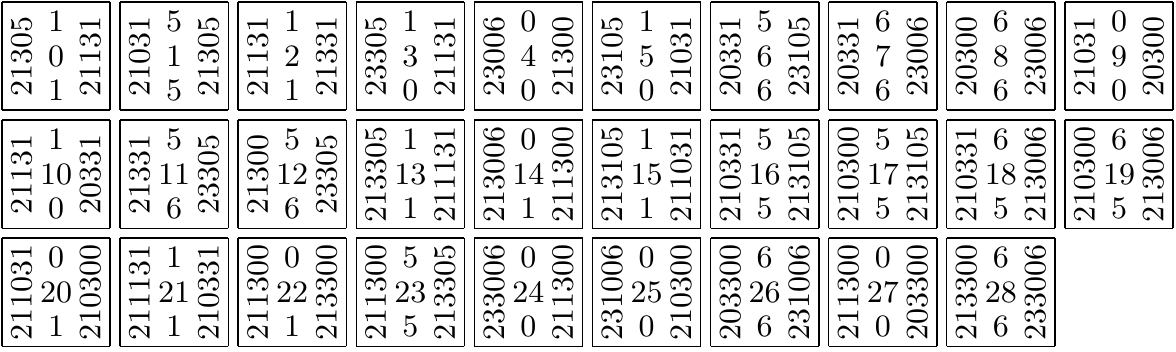}}\right\}
\]
which we can work with because it has markers.

\part{Desubstituting Wang tilings with markers again}\label{part:3}

\section{Desubstitution from $\Omega_6$ to $\Omega_{12}$:
$\Omega_6\xleftarrow{\omega_6}
\Omega_7\xleftarrow{\omega_7}
\Omega_8\xleftarrow{\omega_8}
\Omega_9\xleftarrow{\omega_9}
\Omega_{10}\xleftarrow{\omega_{10}}
\Omega_{11}\xleftarrow{\omega_{11}}
\Omega_{12}$}
\label{sec:desubstitution-6-to-12}

\begin{proposition}\label{prop:from-jeandel-rao-6-to-12}
    Let $\Omega_6$ be the Wang shift defined above from $\T_6$.
There exist sets of Wang tiles 
    $\T_7$, $\T_8$, $\T_9$, $\T_{10}$, $\T_{11}$ and $\T_{12}$
together with their associated Wang shifts 
    $\Omega_7$, $\Omega_8$, $\Omega_9$, $\Omega_{10}$, $\Omega_{11}$ and $\Omega_{12}$
and there exists a sequence of recognizable $2$-dimensional morphisms:
    \begin{equation*}
        \Omega_6 \xleftarrow{\omega_6}
        \Omega_7 \xleftarrow{\omega_7}
        \Omega_8 \xleftarrow{\omega_8}
        \Omega_9 \xleftarrow{\omega_9}
        \Omega_{10} \xleftarrow{\omega_{10}}
        \Omega_{11} \xleftarrow{\omega_{11}}
        \Omega_{12} 
    \end{equation*}
    that are onto up to a shift, i.e.,
    $\overline{\omega_i(\Omega_{i+1})}^\sigma=\Omega_{i}$
    for each $i\in\{6,7,8,9,10,11\}$.
\end{proposition}

\begin{proof}
    The proof is done by executing the function
    $\Call{FindMarkers}$ 
    followed by 
    $\Call{FindSubstitution}$
    and repeating this process.
    Each time $\Call{FindMarkers}$ 
    finds at least one subset of markers using a surrounding radius 
    of size at most 2.
    Thus using Theorem~\ref{thm:exist-homeo},
    we may find a desubstitution of tilings.
    The proof is done in SageMath \cite{sagemathv8.9}
    using \texttt{slabbe} optional package
    \cite{labbe_slabbe_0_6_2019}.

First, we desubstitute $\T_6$:
    {\footnotesize
\begin{verbatim}
sage: T6.find_markers(i=1, radius=1, solver="dancing_links")
[[1, 6, 7, 8, 11, 12, 16, 17, 18, 19, 23, 26, 28],
 [0, 3, 4, 5, 13, 14, 15, 24, 25],
 [2, 9, 10, 20, 21, 22, 27]]
sage: M6 = [1, 6, 7, 8, 11, 12, 16, 17, 18, 19, 23, 26, 28]
sage: T7,omega6 = T6.find_substitution(M=M6, i=1, radius=1, 
....:                                  side="left", solver="dancing_links")
\end{verbatim}
}
and we obtain:
    \[
\omega_6:\Omega_7\to\Omega_6:\left\{ \omegaVI\right.
    \]
    \[
\T_7 =\left\{\raisebox{-10mm}{\includegraphics{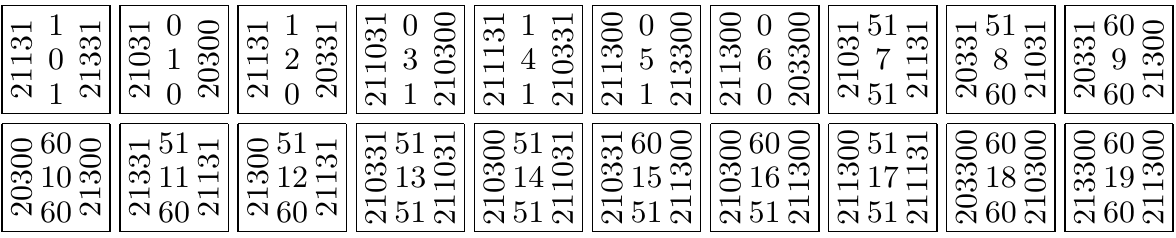}}\right\}
    \]

We desubstitute $\T_7$:
    {\footnotesize
\begin{verbatim}
sage: T7.find_markers(i=1, radius=1, solver="dancing_links")
[[0, 1, 2, 3, 4, 5, 6]]
sage: M7 = [0, 1, 2, 3, 4, 5, 6]
sage: T8,omega7 = T7.find_substitution(M=M7, i=1, radius=1, 
....:                                  side="right", solver="dancing_links")
\end{verbatim}
}
and we obtain:
    \[
\omega_7:\Omega_8\to\Omega_7:\left\{ \omegaVII\right.
    \]
    \[
\T_8 =\left\{\raisebox{-10mm}{\includegraphics{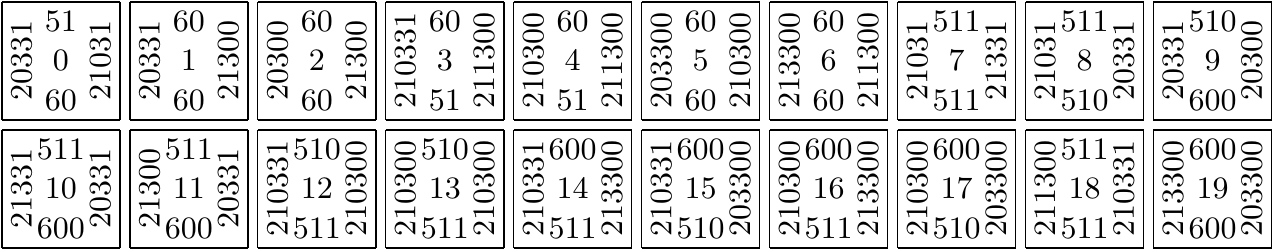}}\right\}
    \]

We desubstitute $\T_8$:
    {\footnotesize
\begin{verbatim}
sage: T8.find_markers(i=2, radius=2, solver="dancing_links")
[[0, 1, 2, 7, 8, 9, 10, 11]]
sage: M8 = [0, 1, 2, 7, 8, 9, 10, 11]
sage: T9,omega8 = T8.find_substitution(M=M8, i=2, radius=2, 
....:                                  side="right", solver="dancing_links")
\end{verbatim}
}
and we obtain:
    \[
\omega_8:\Omega_9\to\Omega_8:\left\{ \omegaVIII\right.
    \]
\[
    \T_9 =\left\{\footnotesize \TIXtable\right.
\]

We desubstitute $\T_9$:
    {\footnotesize
\begin{verbatim}
sage: T9.find_markers(i=1, radius=1, solver="dancing_links")
[[4, 6, 7, 15, 16, 18, 21],
 [0, 1, 2, 3, 9, 10, 11, 12, 13],
 [5, 8, 14, 17, 19, 20]]
sage: M9 = [0, 1, 2, 3, 9, 10, 11, 12, 13]
sage: T10,omega9 = T9.find_substitution(M=M9, i=1, radius=1, 
....:                                   side="right", solver="dancing_links")
\end{verbatim}
}
and we obtain:
    \[
\omega_9:\Omega_{10}\to\Omega_9:\left\{ \omegaIX\right.
    \]
\[
    \T_{10} =\left\{\footnotesize \TXtable\right.
\]

We desubstitute $\T_{10}$:
    {\footnotesize
\begin{verbatim}
sage: T10.find_markers(i=2, radius=2, solver="dancing_links")
[[0, 4, 5, 6, 7, 8]]
sage: M10 = [0, 4, 5, 6, 7, 8]
sage: T11,omega10 = T10.find_substitution(M=M10, i=2, radius=2, 
....:                                     side="right", solver="dancing_links")
\end{verbatim}
}
and we obtain:
    \[
\omega_{10}:\Omega_{11}\to\Omega_{10}:\left\{ \omegaX\right.
    \]
\[
    \T_{11} =\left\{\footnotesize \TXItable\right.
\]

We desubstitute $\T_{11}$:
    {\footnotesize
\begin{verbatim}
sage: T11.find_markers(i=1, radius=1, solver="dancing_links")
[[0, 1, 2, 9, 10, 11],
 [4, 6, 7, 12, 16, 17, 19],
 [3, 5, 8, 13, 14, 15, 18, 20]]
sage: M11 = [0, 1, 2, 9, 10, 11]
sage: T12,omega11 = T11.find_substitution(M=M11, i=1, radius=1, 
....:                                     side="right", solver="dancing_links")
\end{verbatim}
}
and we obtain:
    \[
\omega_{11}:\Omega_{12}\to\Omega_{11}:\left\{ \omegaXI\right.
    \]
\[
    \T_{12} =\left\{\footnotesize \TXIItablewithU\right.
\]
\end{proof}

\section{Proofs of main result}
\label{sec:proofs-main-results}

In \cite{MR3978536} the set of Wang tiles
\[
\U =\left\{\raisebox{-10mm}{\includegraphics{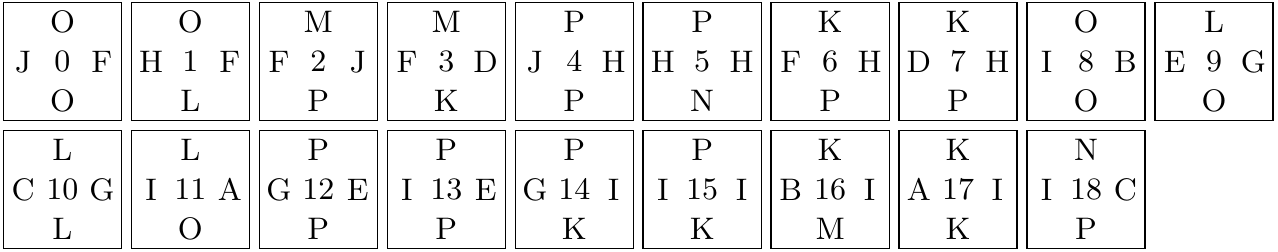}}\right\}
\]
was introduced. It was proved to be self-similar, aperiodic and minimal. In
this section, we prove that the desubstitution process we have done so far
starting from the Jeandel-Rao tiles $\T_0$ leads to the self-similar
set of tiles $\U$.

\begin{proposition}\label{prop:T12-and-U-equivalent}
    The set of Wang tiles $\T_{12}$ and $\U$ are equivalent.
\end{proposition}

\begin{proof}
    The correspondence between colors of $\T_{12}$ and colors of $\U$
    \begin{align*}
        & A\mapsto 21030021300211300 && K\mapsto 51160511\\
        & B\mapsto 21030021031211300 && L\mapsto 60060\\
        & C\mapsto 21330020331210331 && M\mapsto 51151511\\
        & D\mapsto 21033121331       && N\mapsto 60060511\\
        & E\mapsto 21030020331210331 && O\mapsto 51060\\
        & F\mapsto 21030021300       && P\mapsto 51060511\\
        & G\mapsto 21030021300210300 && \\
        & H\mapsto 21330020331       && \\
        & I\mapsto 21033121331210331 && \\
        & J\mapsto 21030020331       &&
    \end{align*}
	provides the bijection between $\T_{12}$ and $\U$
    which defines a bijection between $\Omega_\U$
    and $\Omega_{12}$:
\[
    \rho:\Omega_\U\to\Omega_{12}:\left\{\omegaUtoXII\right.
\]
\end{proof}

The self-similarity of $\Omega_\U$ is given by the following 2-dimensional
morphism:
\[
    \omega_{\U}:\Omega_\U\to\Omega_\U:\left\{\omegaUtoU\right.
\]

We may gather the proof of the main result and its corollaries.

\begin{proof}[Proof of Theorem~\ref{thm:introduction}]
    (i) Proved in Proposition~\ref{prop:from-jeandel-rao-0-to-4}.
    (ii) Proved in 
Proposition~\ref{prop:jmath-5-to-4-is-one-to-one-continuous}.
    (iii) Proved in 
Corollary~\ref{cor:eta-6-to-5-is-conjugacy}.
    (iv) Proved in Proposition~\ref{prop:from-jeandel-rao-6-to-12}.
    (v)
    From Proposition~\ref{prop:T12-and-U-equivalent},
    the set of Wang tiles $\T_{12}$ and $\U$ are equivalent.
    It was proved in \cite{MR3978536}
    that $\Omega_\U$ is self-similar, aperiodic and minimal.
\end{proof}

\begin{proof}[Proof of Corollary~\ref{cor:aperiodic-minimal-5-12}]
    From Theorem~\ref{thm:introduction},
    $\Omega_{12}$ is aperiodic and minimal.
    From Lemma~\ref{lem:minimal-implies-minimal}
    and Lemma~\ref{lem:aperiodic-implies-aperiodic},
    we conclude that
    $\Omega_{11}$,
    $\Omega_{10}$,
    $\Omega_{9}$,
    $\Omega_{8}$,
    $\Omega_{7}$ and
    $\Omega_{6}$ are aperiodic and minimal.

    From Corollary~\ref{cor:eta-6-to-5-is-conjugacy},
    the map $\eta:\Omega_6\to\Omega_5$ is a
    shear conjugacy. In particular, it is one-to-one, onto and
    recognizable.  
    From Lemma~\ref{lem:minimal-implies-minimal}
    and Lemma~\ref{lem:aperiodic-implies-aperiodic},
    we conclude that $\Omega_{5}$ is aperiodic and minimal.
\end{proof}

\begin{proof}[Proof of Corollary~\ref{cor:aperiodic-minimal-0-4}]
    From Proposition~\ref{prop:jmath-5-to-4-is-one-to-one-continuous},
    the map $\jmath:\T_5\to\T_4$ defines an embedding
    $\jmath:\Omega_5\to\Omega_4$.
    In particular, the map $\jmath$ is recognizable.
    From Lemma~\ref{lem:minimal-implies-minimal}
    and Lemma~\ref{lem:aperiodic-implies-aperiodic},
    we conclude that
    $X_4=\overline{\jmath(\Omega_{5})}^\sigma$ is an aperiodic and minimal
    subshift of $\Omega_4$.

    From Proposition~\ref{prop:from-jeandel-rao-0-to-4} and
    Corollary~\ref{cor:can-remove-2-tile}, we have
    that $\omega_i:\Omega_{i+1}\to\Omega_i$ is recognizable
    for every $i\in\{0,1,2,3\}$.
    From Lemma~\ref{lem:minimal-implies-minimal} and
    Lemma~\ref{lem:aperiodic-implies-aperiodic}, we conclude that
    $X_i=\overline{\omega_i(X_{i+1})}^\sigma$ is an aperiodic and minimal
    subshift of $\Omega_i$
    for every $i\in\{0,1,2,3\}$.

    The proof of nonemptyness of $\Omega_i\setminus X_i$ for every
    $i\in\{0,1,2,3\}$ is postponed to Proposition~\ref{prop:nonempty}.
\end{proof}

\begin{proof}[Proof of Theorem~\ref{thm:main-result-in-simple-terms}]
    It follows from Theorem~\ref{thm:introduction}.
\end{proof}

\begin{proposition}
    Each tile of the set of Wang tiles $\T_i$ appears with positive
    frequency in tilings in $X_i$
    for every $i\in\{0,1,\dots,12\}$.
\end{proposition}

\begin{proof}
    Let $v$ be the positive right eigenvector of the incidence matrix of the
    primitive substitution $\omega_\U$. It can be checked that the incidence
    matrix of $\left(\prod_{i=k}^{11}\omega_i\right)\rho$ maps the vector $v$
    to a vector with positive entries for each $k\in\{0,1,\dots,11\}$ where we
    suppose $\omega_4=\jmath$ and $\omega_5=\eta$ to simplify the notation of
    the product. This vector gives the relative frequency of each tile from
    $\T_i$ in tilings in $X_i$.
\end{proof}

The frequencies of tiles in tilings of $X_0\subset\Omega_0$ is in
Table~\ref{tab:frequency}.

\begin{table}[h]
\[
\begin{array}{c|c}
    \text{Tiles} & \text{Frequency}\\
    \hline
\raisebox{-3mm}{\JeandelRaoVII}
 & 5/(12\varphi+14) \approx 0.1496\\
\raisebox{-3mm}{\JeandelRaoO},
\raisebox{-3mm}{\JeandelRaoI},
\raisebox{-3mm}{\JeandelRaoIII},
\raisebox{-3mm}{\JeandelRaoVI},
\raisebox{-3mm}{\JeandelRaoIX}
 & 1/(2\varphi+6)   \approx 0.1083\\
\raisebox{-3mm}{\JeandelRaoV}
 & 1/(5\varphi+4)   \approx 0.0827\\
\raisebox{-3mm}{\JeandelRaoIV},
\raisebox{-3mm}{\JeandelRaoVIII},
\raisebox{-3mm}{\JeandelRaoX}
 & 1/(8\varphi+2)   \approx 0.0669\\
\raisebox{-3mm}{\JeandelRaoII}
 & 1/(18\varphi+10) \approx 0.0256\\
\end{array}
\] 
\caption{Frequency of Jeandel-Rao tiles in tilings in the proper subshift
    $X_0\subset\Omega_0$ in terms of the golden mean
    $\varphi=\frac{1+\sqrt{5}}{2}$.}
    \label{tab:frequency}
\end{table}

We now obtain independently the link between Jeandel-Rao aperiodic tilings and
the Fibonacci word. This link was used by Jeandel and Rao to prove aperiodicity
of their set of Wang tiles \cite{jeandel_aperiodic_2015}.

\begin{proposition}\label{prop:Fibonacci}
    The language $\L\subset\{4,5\}^\Z$ of the sequences of horizontal strips of
    height 4 or 5 in Jeandel-Rao tilings $X_0\subset\Omega_0$ is the language
    of the Fibonacci word, the fixed point of $4\mapsto 5, 5\mapsto 54$.
\end{proposition}

\begin{proof}
Let $\psi:\Omega_{12}\to\{a,b,c,d\}^{\Z^2}$ be the $2$-dimensional morphism
defined as
\[
\begin{array}{rcl}
    \psi:\T_{12} & \to & \{a,b,c,d\}\\
t & \mapsto &
\begin{cases}
a  & \text{ if } t \in \{12,13,14,15,16,17,18\},\\
b  & \text{ if } t \in \{2,3,4,5,6,7\},\\
c  & \text{ if } t \in \{8,9,10,11\},\\
d  & \text{ if } t \in \{0,1\}.
\end{cases}
\end{array}
\]
There exists a unique $2$-dimensional morphism
$\tau:\psi(\Omega_{12})\to\psi(\Omega_{12})$ which commute with the projection
$\psi$ that is $\psi\circ\tau=\omega_{12}\circ\psi$.
The morphism $\tau$ is given by the rule
\[
a \mapsto \left(\begin{array}{cc}
b & d \\
a & c
\end{array}\right),\quad
    b \mapsto \left(\begin{array}{cc}
a & c
\end{array}\right),\quad
    c \mapsto \left(\begin{array}{c}
b \\
a 
\end{array}\right),\quad
    d \mapsto \left(\begin{array}{c}
a
\end{array}\right).
\]
Vertically, the Fibonacci morphism $a\mapsto ab,b\mapsto a$ gives the structure of the
language of horizontal strips.
Then from Figure~\ref{fig:omega-12-to-5},
we see that tiles of $\T_{12}$ with index from 0 to 7 map to horizontal
strips of height $5+4$ and that those with index from 8 to 18 map to
horizontal strips of height $5+4+5$.
The image of the language of the Fibonacci morphism $a\mapsto ab,b\mapsto a$
by the map $a\mapsto 545, b\mapsto 54$ is exactly the language of
the Fibonacci morphism $5\mapsto 54,4\mapsto 5$.
\end{proof}

The minimal subshift $X_0$ is not equal to $\Omega_0$. Still we
believe that $X_0$ is the only minimal subshift of $\Omega_0$ and that
$\Omega_0\setminus X_0$ is a null set for any probability shift-invariant
measure on $\Omega_0$. The goal of the next two sections is to cover these
questions.

\part{Non-minimality of Jeandel-Rao tilings}\label{part:4}

\section{Eight tilings in $\Omega_\U$ fixed by the square of $\omega_\U$}
\label{sec:fix-points-Omega_12}

To understand Jeandel-Rao tilings $\Omega_0$ we need to understand $\Omega_\U$.
It was proved in \cite{MR3978536}
that the Wang shift $\Omega_\U$ is minimal, aperiodic and self-similar
and that $\omega_\U:\Omega_\U\to\Omega_\U$ 
is an expansive recognizable morphism that is onto up
to a shift. The subshift $\Omega_\U$ is the hidden internal self-similar
structure of $\Omega_0$. In this section, we describe the fixed points of
$\omega_\U$ which are helpful to define the tilings of $\Omega_0$ that we
cannot describe by the image of morphisms.


The first two images of 
$\left(\begin{smallmatrix}
17 & 13 \\
6 & 5
\end{smallmatrix}\right)$
under the $2$-dimensional morphism $\omega_\U$ are:
\[
\left(\begin{array}{r|r}
17 & 13 \\
\hline
6 & 5
\end{array}\right)
\xrightarrow{\omega_\U}
\left(\begin{array}{rr|rr}
4 & 1 & 6 & 1 \\
13 & 9 & 14 & 11 \\
\hline
15 & 8 & 16 & 8
\end{array}\right)
\xrightarrow{\omega_\U}
\left(\begin{array}{rrr|rrr}
17 & 8 & 16 & 15 & 8 & 16 \\
6 & 1 & 3 & 7 & 1 & 2 \\
14 & 11 & 17 & 13 & 9 & 14 \\
\hline
6 & 1 & 6 & 5 & 1 & 6 \\
12 & 9 & 14 & 18 & 10 & 14
\end{array}\right)
\]
Therefore, the square of the morphism $\omega_\U$ is prolongable on
$\left(\begin{smallmatrix}
17 & 13 \\
6 & 5
\end{smallmatrix}\right)$ 
in all four quadrants
(since the original patch occurs in the center of the last one)
and there exist
a unique tiling
$z:\Z^2\to\T_\U$ 
with $z=\omega_\U^2(z)\in\Omega_\U$
such that at the origin we have:
\[
\left(\begin{array}{rr}
z_{-\be_1} & z_{\zero} \\
z_{-\be_1-\be_2} & z_{-\be_2}
\end{array}\right)
=
\left(\begin{array}{rr}
17 & 13 \\
6 & 5
\end{array}\right).
\]

\begin{lemma}\label{lem:the8fixedpoints}
    The square of the morphism $\omega_\U$ has eight distinct fixed points
    $z=\omega_\U(z)\in\Omega_\U$ that can be generated from  the following
    values of $z$ at the origin:
    \[
    \arraycolsep=1.4pt
\left(\begin{array}{rr}
    z_{(-1,0)} & z_{\zero} \\
    z_{-(1,1)} & z_{(0,-1)}
\end{array}\right)
\in
        \left\{
\left(\begin{array}{rr}
9 & 14 \\
8 & 16
\end{array}\right),
\left(\begin{array}{rr}
9 & 14 \\
1 & 6
\end{array}\right),
\left(\begin{array}{rr}
17 & 13 \\
16 & 15
\end{array}\right),
\left(\begin{array}{rr}
17 & 13 \\
6 & 5
\end{array}\right),
\left(\begin{array}{rr}
16 & 15 \\
3 & 7
\end{array}\right),
\left(\begin{array}{rr}
10 & 12 \\
9 & 14
\end{array}\right),
\left(\begin{array}{rr}
16 & 13 \\
2 & 4
\end{array}\right),
\left(\begin{array}{rr}
10 & 14 \\
11 & 17
\end{array}\right)
    \right\}.
    \]
\end{lemma}

\begin{figure}
    \includegraphics[height=.95\textheight]{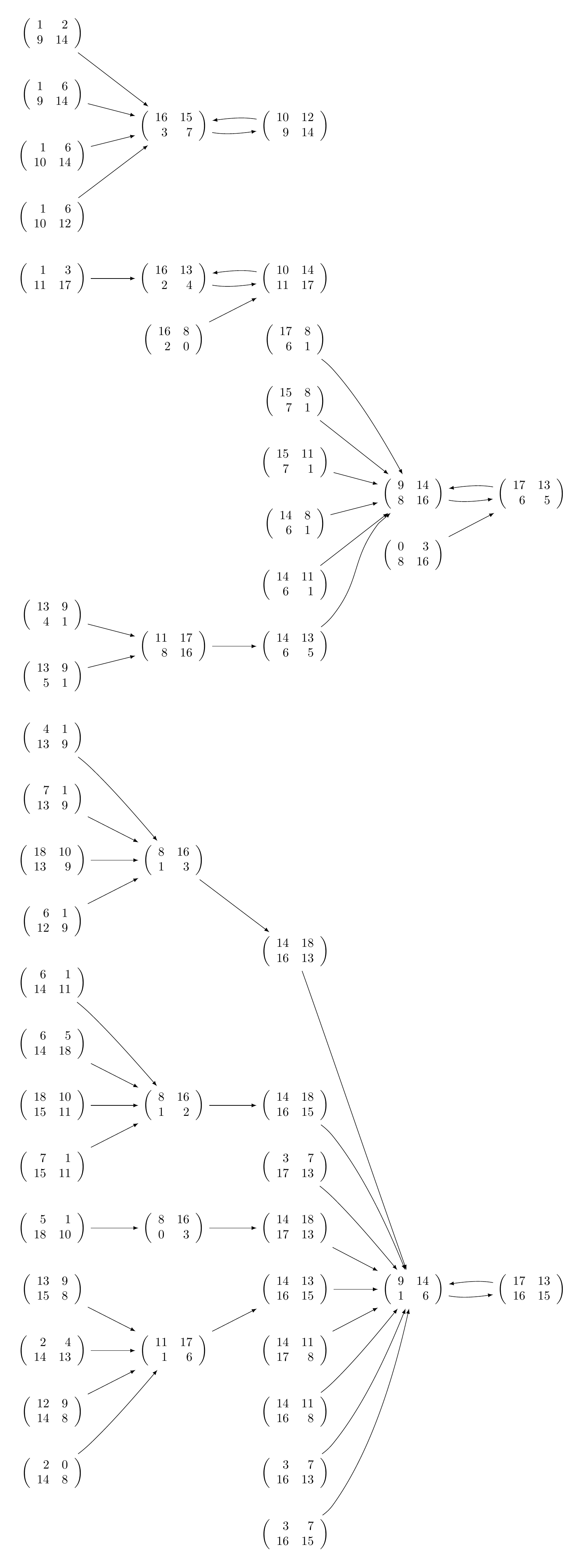}
    \caption{The eight fixed point of $\omega_\U$ are generated by the $2\times 2$ blocks
    that belong to a recurrent strongly connected component of this graph.}
    \label{fig:8fixedpoints}
\end{figure}

\begin{proof}
    The proof is done by applying the morphism $\omega_\U$ on each of the
    fifty $2\times 2$ factors of the language of $\Omega_\U$ already computed
    in \cite{MR3978536}. We get a graph shown
    in Figure~\ref{fig:8fixedpoints}. In the graph, there are 4 cycles of size
    2 and no other strongly connected components.
    Each of these 8 vertices correspond to a fixed point $z\in\Omega_\U$ of 
    the square of $\omega_\U$.
\end{proof}

\begin{lemma}\label{lem:rows-of-156-81516}
    Among the eight fixed points of $\omega_\U^2$ in $\Omega_\U$,
    only four contain a row made of tiles in $\{1,5,6\}$
    or a row made of tiles in $\{8,15,16\}$:
    the fixed points $z=(\omega_\U^2)^\infty(a)$
    generated by some
    \[
        a\in
    \arraycolsep=1.8pt
        \left\{
\left(\begin{array}{rr}
9 & 14 \\
8 & 16
\end{array}\right),
\left(\begin{array}{rr}
9 & 14 \\
1 & 6
\end{array}\right),
\left(\begin{array}{rr}
17 & 13 \\
16 & 15
\end{array}\right),
\left(\begin{array}{rr}
17 & 13 \\
6 & 5
\end{array}\right)
    \right\}.
    \]
Moreover, they contain exactly one of these row.
\end{lemma}

\begin{proof}
    The 8 fixed points are given in Lemma~\ref{lem:the8fixedpoints}.
    Only 1, 5, 6 and 10 contain a row over $\{1, 5, 6\}^*$ in their image under
    $\omega_\U$. 
    Only $\left(\begin{smallmatrix}
        9 & 14 \\
        1 & 6
\end{smallmatrix}\right)$
and $\left(\begin{smallmatrix}
        17 & 13 \\
        6 & 5
\end{smallmatrix}\right)$ contain a row over $\{1, 5, 6,10\}^*$.
    Since the exactly one row (the top one) of the images of $1,5,6$ under
    $\omega_\U$ is over $\{1,5,6\}^*$, this property is preserved under the
    application of $\omega_\U$. Similarly, we show that only
    $\left(\begin{smallmatrix}
        9 & 14 \\
        8 & 16
\end{smallmatrix}\right)$
and $\left(\begin{smallmatrix}
        17 & 13 \\
        16 & 15
\end{smallmatrix}\right)$ generate fixed point with a row over $\{8,15,16\}^*$.
\end{proof}

Finally we illustrate in Figure~\ref{fig:rauzy_graph_U} the language of
horizontal dominoes in $\Omega_\U$ as a Rauzy graph.

\begin{figure}
    \includegraphics[width=.70\linewidth]{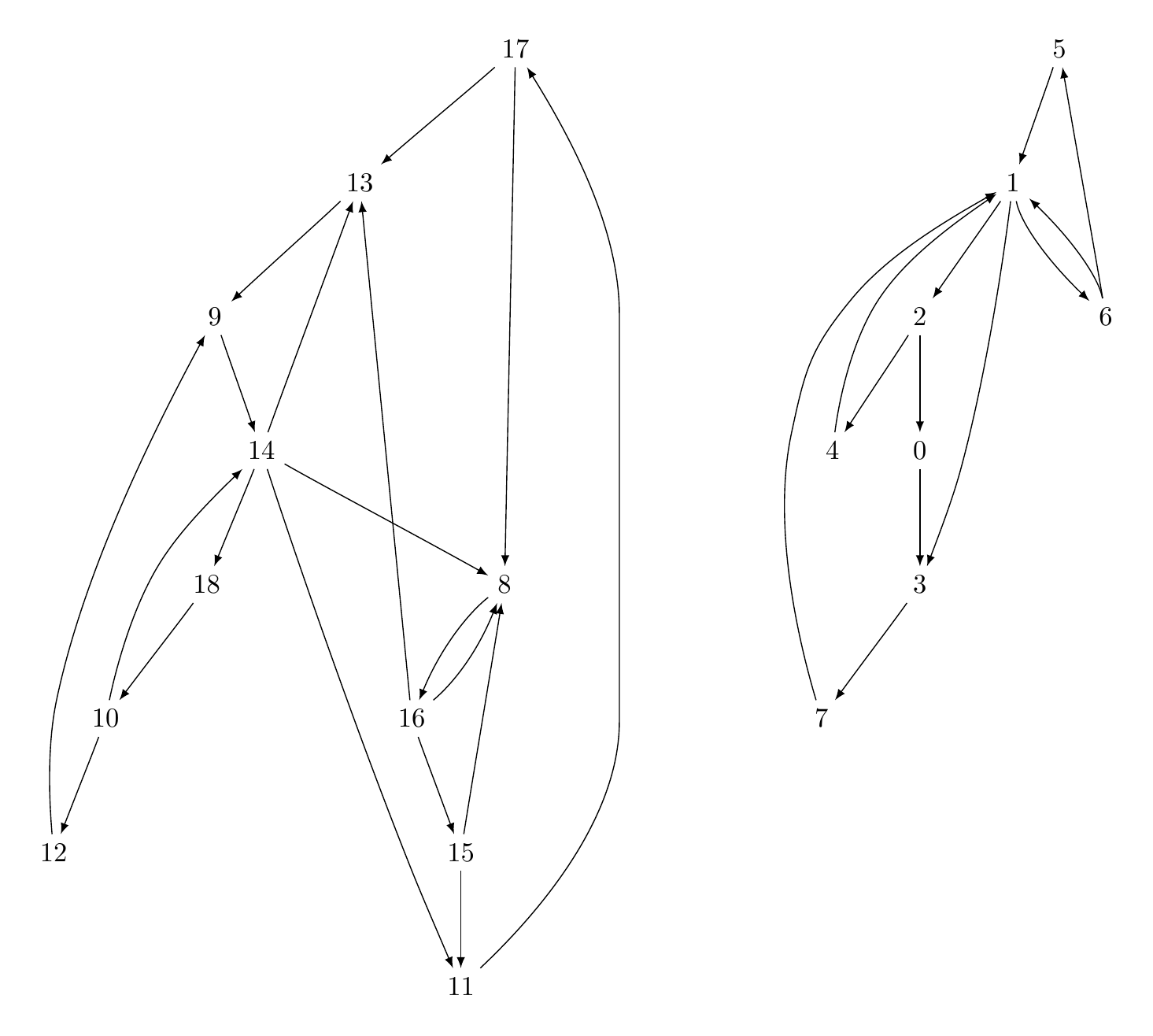}
\caption{The horizontal Rauzy graph for the tiles in $\Omega_\U$.
    There is an arrow $a\to b$ in the graph with $a,b\in\T_\U$ if and only if
    $a\odot^1 b$ is in the language of $\Omega_\U$.}
    \label{fig:rauzy_graph_U}
\end{figure}
    
\section{Understanding $\Omega_4\setminus X_4$}
\label{sec:fault lines}

Recall that $X_4=\overline{\jmath(\Omega_5)}^\sigma=\jmath(\Omega_5)$
as $\jmath(\Omega_5)$ is shift-invariant since $\jmath$ is a uniform morphism
mapping each letter on a letter.
The goal of this section is to show that
$\Omega_4\setminus X_4$
is nonempty and give a conjecture on its structure.
Informally, we believe that all tilings in $\Omega_4\setminus X_4$ are obtained
by sliding particular tilings in $X_4$ along a horizontal fault line of 0's or
1's.

The morphism $\eta\,
\omega_{6}\,
\omega_{7}\,
\omega_{8}\,
\omega_{9}\,
\omega_{10}\,
\omega_{11}\,\rho:\Omega_\U\to\Omega_5$
is illustrated in 
Figure~\ref{fig:omega-12-to-5}.
Recall that the application of the map $\jmath:\Omega_5\to\Omega_4$ replaces
horizontal colors 5 and 6 by 1 and 0 respectively.
Thus it can be seen on the figure that
    horizontal fault lines of 0's are present in the image of
    tiles numbered 1, 5 and 6 and horizontal fault lines of 1's are present in the images of
    tiles numbered 8, 15 and~16. In the next lemma, we formalize this observation.

Let $K\subseteq\Z$ be an interval of integers. A Wang tiling $t:\Z\times
K\to\T$ has a \emph{horizontal fault line} of color $a$ if there exists $k\in K$
such that $\scbottom(t(m,k))=a$ for all $m\in\Z$.

\begin{lemma}\label{lem:fault line-in-rows-iff}
    Let $u:\Z\to\T_\U$
    be a infinite horizontal row of some tiling in $\Omega_\U$.
\begin{enumerate}[\rm (i)]
\item $\jmath\,\eta\,
\omega_{6}\,
\omega_{7}\,
\omega_{8}\,
\omega_{9}\,
\omega_{10}\,
        \omega_{11}\,\rho(u)$
    contains a horizontal fault line of $0$'s
    if and only if $u\in\{1,5,6\}^\Z$.
\item $\jmath\,\eta\,
\omega_{6}\,
\omega_{7}\,
\omega_{8}\,
\omega_{9}\,
\omega_{10}\,
        \omega_{11}\,\rho(u)$
    contains a horizontal fault line of $1$'s
    if and only if $u\in\{8,15,16\}^\Z$.
\end{enumerate}
\end{lemma}

\begin{proof}
    (i) If $u\in\{1,5,6\}^\Z$, then $\jmath\,\eta\,
\omega_{6}\,
\omega_{7}\,
\omega_{8}\,
\omega_{9}\,
\omega_{10}\,
        \omega_{11}\,\rho(u)$
    contains a horizontal fault line of $0$'s.
    If $\jmath\,\eta\,
\omega_{6}\,
\omega_{7}\,
\omega_{8}\,
\omega_{9}\,
\omega_{10}\,
        \omega_{11}\,\rho(u)$
    contains a horizontal fault line of $0$'s, then
    we see from Figure~\ref{fig:omega-12-to-5} that
    $u\in\{0,1,4,5,6,9,10\}^\Z$.
    In the horizontal language of $\Omega_\U$ (see Figure~\ref{fig:rauzy_graph_U}),
    the tile 10 is preceded by 18 which has no fault line of 0's and
    the tile 9 is followed by 14 which has no fault line of 0's.
    Moreover, the tiles 0 and 4 are preceded by tile 2 which has no fault line of
    0's. We conclude that $u\in\{1,5,6\}^\Z$.

    (ii) If $u\in\{8,15,16\}^\Z$, then $\jmath\,\eta\,
\omega_{6}\,
\omega_{7}\,
\omega_{8}\,
\omega_{9}\,
\omega_{10}\,
        \omega_{11}\,\rho(u)$
    contains a horizontal fault line of $1$'s.
    If $\jmath\,\eta\,
\omega_{6}\,
\omega_{7}\,
\omega_{8}\,
\omega_{9}\,
\omega_{10}\,
        \omega_{11}\,\rho(u)$
    contains a horizontal fault line of $1$'s, then
    we see from Figure~\ref{fig:omega-12-to-5} that
    $u\in\{8,13,15,16\}^\Z$.
    In the language of $\Omega_\U$ (see Figure~\ref{fig:rauzy_graph_U}),
    the tile 13 is always followed by 9 which has no fault line of 1's.
    We conclude that $u\in\{8,15,16\}^\Z$.
\end{proof}

We can now use the result of the previous section on the fixed points of
$\omega_\U$ to construct tilings in $\Omega_4$ with horizontal fault lines.

\begin{lemma}\label{lem:has-horizontal-fault line}
Let $a\in
    \left\{
\left(\begin{smallmatrix}
9 & 14 \\
8 & 16
\end{smallmatrix}\right),
\left(\begin{smallmatrix}
9 & 14 \\
1 & 6
\end{smallmatrix}\right),
\left(\begin{smallmatrix}
17 & 13 \\
16 & 15
\end{smallmatrix}\right),
\left(\begin{smallmatrix}
17 & 13 \\
6 & 5
\end{smallmatrix}\right)
    \right\}$ and
$z=\lim_{n\to\infty}(\omega_\U)^{2n}(a)\in\Omega_\U$ be the fixed point of
    $\omega_\U$ with $a$ at origin.
Let
\[
y= \jmath\,\eta\,
\omega_{6}\,
\omega_{7}\,
\omega_{8}\,
\omega_{9}\,
\omega_{10}\,
    \omega_{11}\,\rho(z)\in\Omega_4.
\]
If $a\in
    \left\{
\left(\begin{smallmatrix}
9 & 14 \\
1 & 6
\end{smallmatrix}\right),
\left(\begin{smallmatrix}
17 & 13 \\
6 & 5
\end{smallmatrix}\right)
    \right\}$ then $y$ has a horizontal fault line of 0's below the row $k=-1$.
If $a\in
    \left\{
\left(\begin{smallmatrix}
9 & 14 \\
8 & 16
\end{smallmatrix}\right),
\left(\begin{smallmatrix}
17 & 13 \\
16 & 15
\end{smallmatrix}\right)
    \right\}$ then $y$ has a horizontal fault line of 1's below the row $k=-1$.
\end{lemma}

\begin{proof}
    The proof follows from 
    Lemma~\ref{lem:rows-of-156-81516} and
    Lemma~\ref{lem:fault line-in-rows-iff}.
\end{proof}

Consider now two well-chosen tilings in $\Omega_4$:
\begin{align*}
y_0&= \jmath\,\eta\,
\omega_{6}\,
\omega_{7}\,
\omega_{8}\,
\omega_{9}\,
\omega_{10}\,
    \omega_{11}\,\rho(z_0)\in\Omega_4
\quad
    \text{ with }
\quad
z_0=\lim_{n\to\infty}(\omega_\U)^{2n}
\left(\begin{smallmatrix}
9 & 14 \\
1 & 6
\end{smallmatrix}\right)\in\Omega_\U \text{ and }\\
y_1&= \jmath\,\eta\,
\omega_{6}\,
\omega_{7}\,
\omega_{8}\,
\omega_{9}\,
\omega_{10}\,
    \omega_{11}\,\rho(z_1)\in\Omega_4
\quad
    \text{ with }
\quad
z_1=\lim_{n\to\infty}(\omega_\U)^{2n}
\left(\begin{smallmatrix}
9 & 14 \\
8 & 16
\end{smallmatrix}\right)\in\Omega_\U.
\end{align*}
For any Wang tiling $y:\Z\to\T$,
we define $y^+$ to be the tiling $y$ restricted to the upper half-plane
$\Z\times(\N-1)$ and $y^-$ to be the tiling $y$ restricted to
the lower half-plane $\Z\times(-2-\N)$:
\[
y^+ = y|_{\Z\times(\N-1)}
\qquad
\text{ and }
\qquad
y^- = y|_{\Z\times(-2-\N)}.
\]
It follows from Lemma~\ref{lem:has-horizontal-fault line} that
the tiling $y_0^+$ (resp. $y_1^+$) is a tiling of the upper half-plane with
only color 0 (resp. 1) on the bottom. The tiling $y_0^-$ (resp. $y_1^-$) is a
tiling of the lower half-plane with only color 0 (resp. 1) on the top. 

\begin{remark}
    Obviously, the tilings $z_0$ and $z_1$ agree on the upper half-plane
    $\{(m,n)\in\Z^2\mid n \geq 0\}$ but surprisingly,
    they also agree on the lower half-plane
    $\{(m,n)\in\Z^2\mid n \leq -3\}$.
    We do not provide a proof of this here as it is not needed for the
    next results.
\end{remark}

\begin{figure}
\begin{center}
    \includegraphics[width=.80\linewidth]{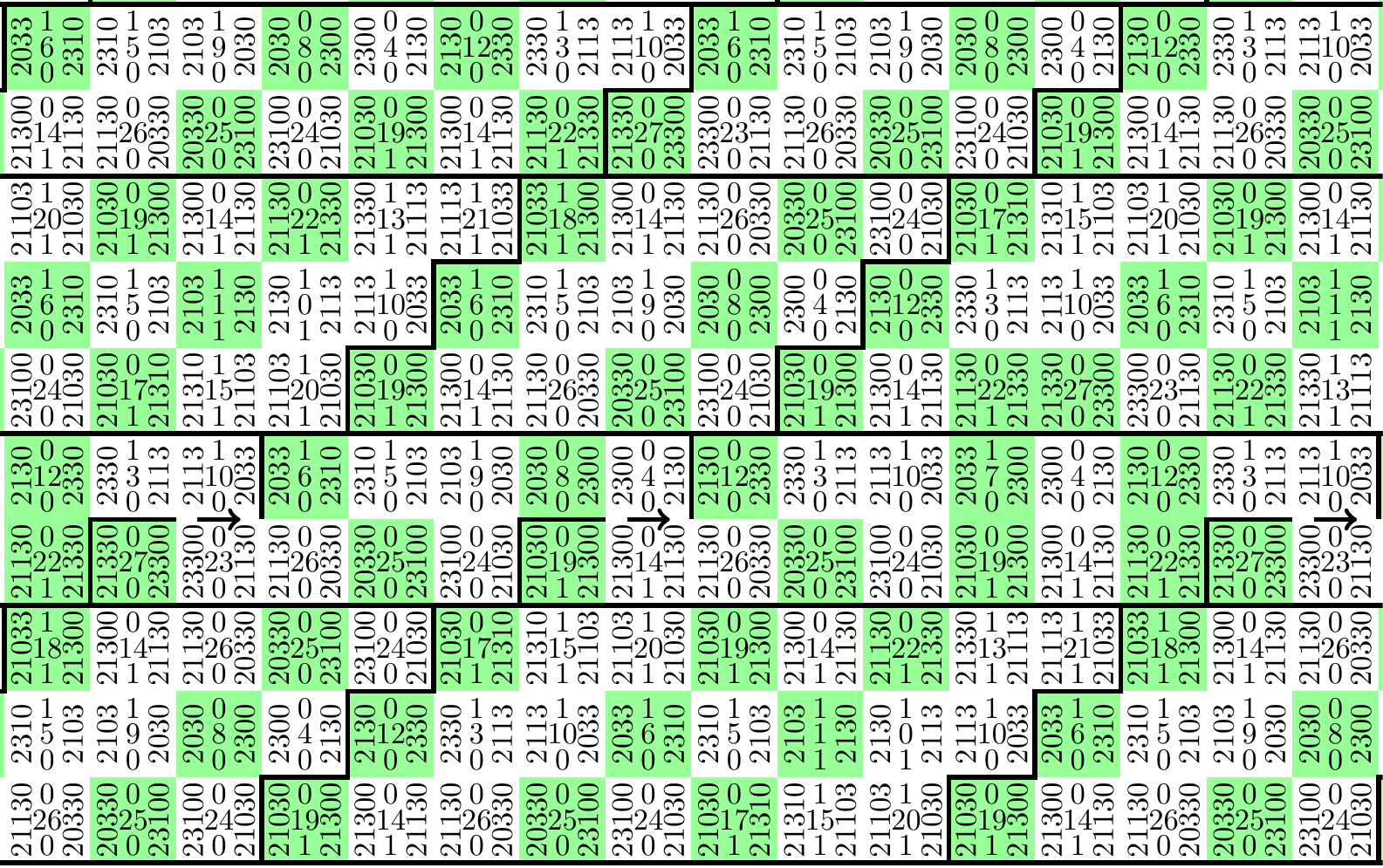}
\end{center}
    \caption{The tiling $x = y_0^-\odot^2
    \left(\sigma^{-\be_1}y_0^+\right)\in\Omega_4$ 
    illustrated in the window $[-8,7]\times[-5,4]$. 
    The image under the map $\jmath$ of the marker tiles 
    shown in green in Figure~\ref{fig:tiling_with_T5},
    i.e., the tiles of $\T_4$ with indices
        $\{0, 3, 4, 5, 13, 14, 15, 23, 24\}$,
    are shown with green background here.
    Those \emph{kind of diagonal markers} appear on lines of slope $+1$
    but those diagonal lines are broken by the shift of the upper half by one unit
    to the right. Hence $x\notin X_4$.}
    \label{fig:fault line}
\end{figure}

\begin{proposition}\label{prop:nonempty}
    The set $\Omega_4\setminus X_4$ is nonempty.
\end{proposition}

\begin{proof}
    Let
    $x = y_0^-\odot^2
    \left(\sigma^{-\be_1}y_0^+\right)$.
    Because of the horizontal fault line of 0's, adjacent edges in the tiling $x$
    are still valid. Thus we have $x\in\Omega_4$ is a valid Wang tiling.
    But $x\notin X_4$.
    Indeed, 
    any tiling in $ X_4$
    is the image under $\jmath\circ\eta$ of some tiling in $\Omega_6$.
    The image of the marker tiles in $M_6$
    must appear on diagonal lines of slope $1$ in the tiling $x$
    (recall Figure~\ref{fig:tiling_with_T5}).
    The fact that the upper half-tiling $y^+$ is shifted by one horizontal
    unit breaks the property of diagonal line of markers in $x$
    (see Figure~\ref{fig:fault line}).
    Therefore $x\notin X_4$.
\end{proof}

We can construct an uncountable number of tilings in
$\Omega_4\setminus X_4$.
For $i\in\{0,1\}$, let $Y_i^+$ (resp. $Y_i^-$) be the subshift generated by the
orbit of $y_i^+$ (resp.  $y_i^-$) under the action of the horizontal shift
$\sigma^{\be_1}$:
\begin{align*}
Y_i^+ &= \textsc{Closure}\left(\left\{\sigma^{k\be_1}y_i^+\mid k\in\Z\right\}\right),\\
Y_i^- &= \textsc{Closure}\left(\left\{\sigma^{k\be_1}y_i^-\mid k\in\Z\right\}\right).
\end{align*}
Any product $u^-\odot^2 u^+$ with $u^-\in Y_i^-$ and $u^+\in Y_i^+$ belongs to
$\Omega_4$ since they match along a horizontal fault line of $i$'s with
$i\in\{0,1\}$.

\begin{proposition}\label{prop:has-measure-0}
    For $i\in\{0,1\}$, 
    the subset 
    of tilings in $\Omega_4$ obtained by sliding half-plane along a horizontal fault line of $i$'s
    \[
        \overline{Y_i^-\odot^2 Y_i^+}^\sigma
        \subset\Omega_4
    \]
    has measure 0.
\end{proposition}

\begin{proof}
    Let $\mu$ be a shift invariant probability measure on $\Omega_4$, i.e.,
    $\mu(\sigma^\bk X)=\mu(X)$ for every measurable subset $X\subset\Omega_4$
    and every $\bk\in\Z^d$.
    Let $i\in\{0,1\}$.
    We have the following disjoint union
    \[
        \overline{Y_i^-\odot^2 Y_i^+}^\sigma
        =
        \bigsqcup_{k\in\Z}
        \sigma^{k\be_2} \left(Y_i^-\odot^2 Y_i^+\right).
    \]
    The union is disjoint since 
    it follows from Lemma~\ref{lem:rows-of-156-81516} that
    each of these tilings has at most one horizontal fault line
    of $i$'s.
    The set $Y_i^-\odot^2 Y_i^+$ is closed and thus measurable
    with $\mu$-measure $\kappa$. Thus
        $\mu\left(\overline{Y_i^-\odot^2 Y_i^+}^\sigma\right)
        =
        \sum_{k\in\Z}\kappa$
    which is bounded only if $\kappa=0$.
    We conclude that
        $\mu\left(\overline{Y_i^-\odot^2 Y_i^+}^\sigma\right)=0$.
\end{proof}

\begin{conjecture}\label{conj:fault line}
    A tiling in $\Omega_4$ has at most one horizontal fault line
    and if it has one then it is a fault line of 
    0's or 1's.
    In other words,
    \[
        \Omega_4\setminus X_4
        =
        \left(
        \overline{Y_0^-\odot^2 Y_0^+}^\sigma
        \cup
        \overline{Y_1^-\odot^2 Y_1^+}^\sigma
        \right)
        \setminus
        \overline{\{y_0,y_1,y_2,y_3\}}^\sigma.
    \]
where $y_0$ and $y_1$ are defined above and
\begin{align*}
y_2&= \jmath\,\eta\,
\omega_{6}\,
\omega_{7}\,
\omega_{8}\,
\omega_{9}\,
\omega_{10}\,
    \omega_{11}\,\rho(z_2)\in\Omega_4
\quad
    \text{ with }
\quad
z_2=\lim_{n\to\infty}(\omega_\U)^{2n}
\left(\begin{smallmatrix}
17 & 13 \\
16 & 15
\end{smallmatrix}\right)\in\Omega_\U \text{ and }\\
y_3&= \jmath\,\eta\,
\omega_{6}\,
\omega_{7}\,
\omega_{8}\,
\omega_{9}\,
\omega_{10}\,
    \omega_{11}\,\rho(z_3)\in\Omega_4
\quad
    \text{ with }
\quad
z_3=\lim_{n\to\infty}(\omega_\U)^{2n}
\left(\begin{smallmatrix}
17 & 13 \\
6 & 5
\end{smallmatrix}\right)\in\Omega_\U.
\end{align*}
\end{conjecture}

To prove that conjecture, which implies
Conjecture~\ref{conj:complement-measure-zero-Omega0},
we need to prove that the transducer associated to
the set of Wang tiles $\T_8$ (with each occurrence of 5 and 6 in horizontal colors
replaced by 1 and 0 respectively) is \emph{synchronized} for some
definition of synchronization. Informally, if a word $v=v_0v_1\dots
v_k=v'_0v'_1\dots v'_k$ is simultaneously the output of some computation 
\[
a_0\xrightarrow{u_0|v_0}
a_1\xrightarrow{u_1|v_1}
a_2
\dots
\xrightarrow{u_{k-1}|v_{k-1}}
a_k
\]
and the input of some computation
\[
b_0\xrightarrow{v'_0|w_0}
b_1\xrightarrow{v'_1|w_1}
b_2
\dots
\xrightarrow{v'_{k'-1}|w_{k'-1}}
b_{k'}
\]
then $k=k'$ and $v_i=v'_i$ for every $i$ with $1\leq i\leq k$,
then we say that the transducer is \emph{synchronized} on the word $v$.
We need to show that the only arbitrarily large words for which the transducer
$\T_8$ is not synchronized are those of the form $00000\dots$ or $111111\dots$.
Some partial results were found in that direction but they were insufficient to
prove the conjecture.

\part*{References and appendix}

\bibliographystyle{myalpha} 
\bibliography{biblio}

\newpage
\section*{Appendix}

\begin{figure}[h]
    \begin{center}
\includegraphics{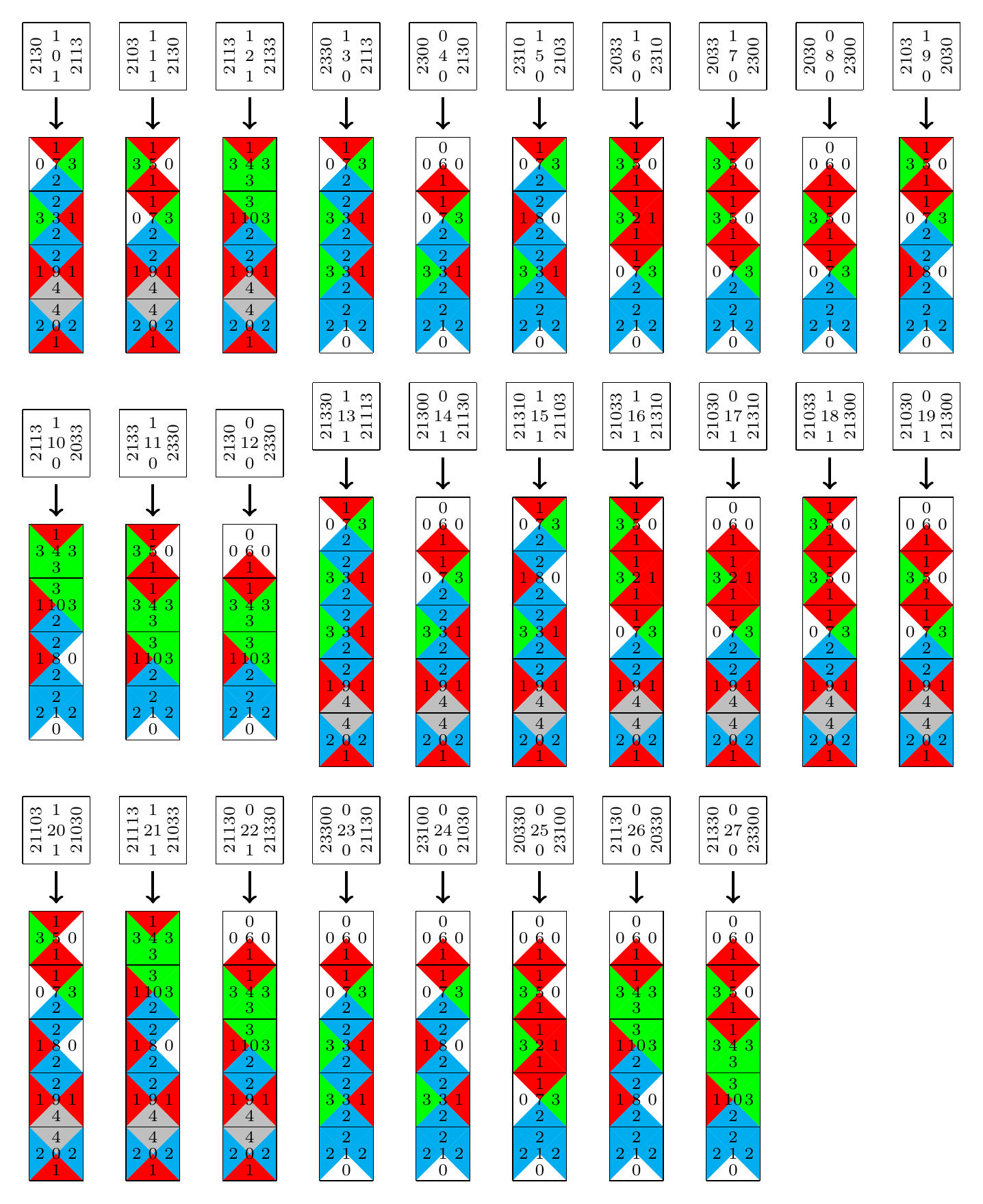}
    \end{center}
    \caption{The morphism 
    $\omega_0\,\omega_1\,\omega_2\,\omega_3:\Omega_4\to\Omega_0$
    is recognizable and onto up to a shift.}
    \label{fig:omega-4-to-0}
\end{figure}

\begin{figure}[p]
    \includegraphics[width=.95\linewidth]{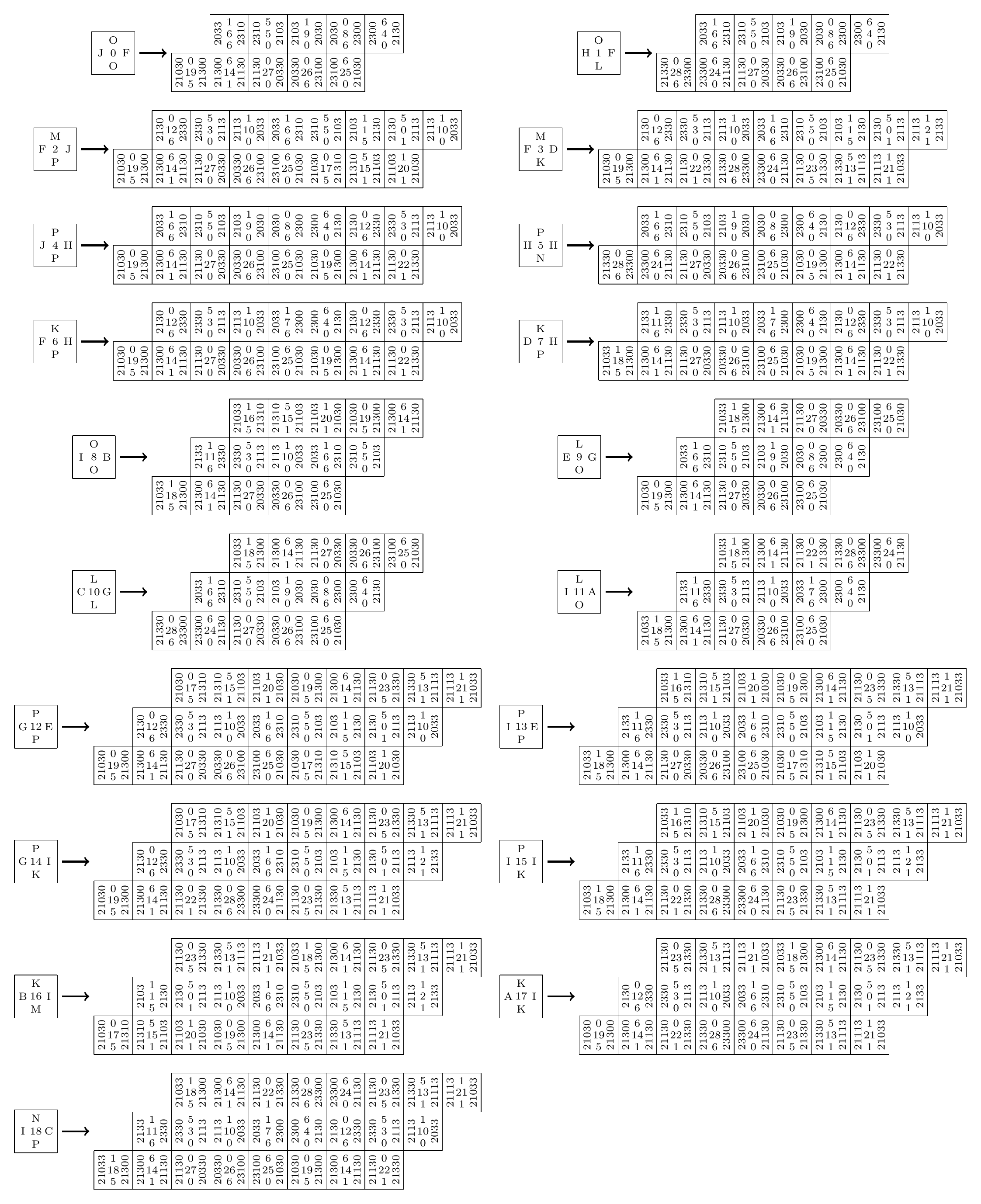}
    \caption{The map
    $\eta\,\omega_6\,\omega_7\,\omega_8
    \,\omega_9\,\omega_{10}\,\omega_{11}\,\rho:\Omega_\U\to\Omega_5$
    is recognizable and onto up to a shift. The map
    $\jmath:\Omega_5\to\Omega_4$ then replaces horizontal colors 5 and 6 by 1 and
    0, respectively, which creates horizontal fault lines of 0's in the image of
    tiles numbered 1, 5 and 6 and horizontal fault lines of 1's in the images of
    tiles numbered 8, 15 and~16.}
    \label{fig:omega-12-to-5}
\end{figure}

\begin{figure}
    \begin{center}
        \includegraphics[width=\linewidth]{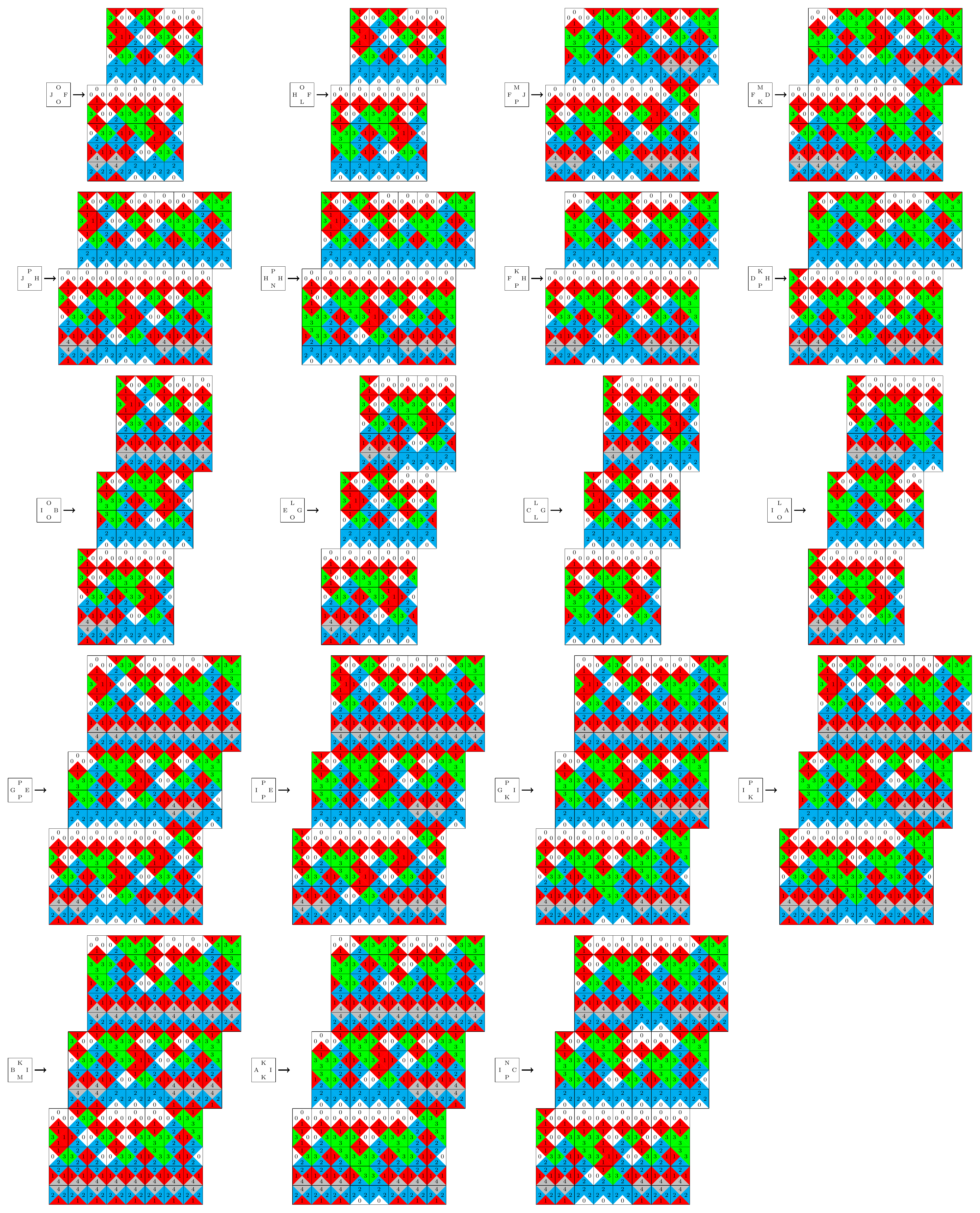}
    \end{center}
    \caption{The map
    $\omega_0\,\omega_1\,\omega_2\,\omega_3\,\jmath\,
    \eta\,\omega_6\,\omega_7\,\omega_8
    \,\omega_9\,\omega_{10}\,\omega_{11}\,\rho:\Omega_\U\to\Omega_0$.}
    \label{fig:omega-12-to-0}
\end{figure}

\end{document}

%% file: article2_macro_jeandel_rao.tex
\newcommand\JeandelRaoO{
\begin{tikzpicture}
[scale=0.900000000000000]
\tikzstyle{every node}=[font=\footnotesize]
\fill[cyan] (1, 0) -- (0.5, 0.5) -- (1, 1);
\fill[lightgray] (0, 1) -- (0.5, 0.5) -- (1, 1);
\fill[cyan] (0, 0) -- (0.5, 0.5) -- (0, 1);
\fill[red] (0, 0) -- (0.5, 0.5) -- (1, 0);
\draw (1, 0) -- ++ (0,1);
\draw (0, 1) -- ++ (1,0);
\draw (0, 0) -- ++ (0,1);
\draw (0, 0) -- ++ (1,0);
\node[rotate=0,black] at (0.800000000000000, 0.5) {2};
\node[rotate=0,black] at (0.5, 0.800000000000000) {4};
\node[rotate=0,black] at (0.200000000000000, 0.5) {2};
\node[rotate=0,black] at (0.5, 0.200000000000000) {1};
\end{tikzpicture}
} 
\newcommand\JeandelRaoI{
\begin{tikzpicture}
[scale=0.900000000000000]
\tikzstyle{every node}=[font=\footnotesize]
\fill[cyan] (1, 0) -- (0.5, 0.5) -- (1, 1);
\fill[cyan] (0, 1) -- (0.5, 0.5) -- (1, 1);
\fill[cyan] (0, 0) -- (0.5, 0.5) -- (0, 1);
\fill[white] (0, 0) -- (0.5, 0.5) -- (1, 0);
\draw (1, 0) -- ++ (0,1);
\draw (0, 1) -- ++ (1,0);
\draw (0, 0) -- ++ (0,1);
\draw (0, 0) -- ++ (1,0);
\node[rotate=0,black] at (0.800000000000000, 0.5) {2};
\node[rotate=0,black] at (0.5, 0.800000000000000) {2};
\node[rotate=0,black] at (0.200000000000000, 0.5) {2};
\node[rotate=0,black] at (0.5, 0.200000000000000) {0};
\end{tikzpicture}
} 
\newcommand\JeandelRaoII{
\begin{tikzpicture}
[scale=0.900000000000000]
\tikzstyle{every node}=[font=\footnotesize]
\fill[red] (1, 0) -- (0.5, 0.5) -- (1, 1);
\fill[red] (0, 1) -- (0.5, 0.5) -- (1, 1);
\fill[green] (0, 0) -- (0.5, 0.5) -- (0, 1);
\fill[red] (0, 0) -- (0.5, 0.5) -- (1, 0);
\draw (1, 0) -- ++ (0,1);
\draw (0, 1) -- ++ (1,0);
\draw (0, 0) -- ++ (0,1);
\draw (0, 0) -- ++ (1,0);
\node[rotate=0,black] at (0.800000000000000, 0.5) {1};
\node[rotate=0,black] at (0.5, 0.800000000000000) {1};
\node[rotate=0,black] at (0.200000000000000, 0.5) {3};
\node[rotate=0,black] at (0.5, 0.200000000000000) {1};
\end{tikzpicture}
} 
\newcommand\JeandelRaoIII{
\begin{tikzpicture}
[scale=0.900000000000000]
\tikzstyle{every node}=[font=\footnotesize]
\fill[red] (1, 0) -- (0.5, 0.5) -- (1, 1);
\fill[cyan] (0, 1) -- (0.5, 0.5) -- (1, 1);
\fill[green] (0, 0) -- (0.5, 0.5) -- (0, 1);
\fill[cyan] (0, 0) -- (0.5, 0.5) -- (1, 0);
\draw (1, 0) -- ++ (0,1);
\draw (0, 1) -- ++ (1,0);
\draw (0, 0) -- ++ (0,1);
\draw (0, 0) -- ++ (1,0);
\node[rotate=0,black] at (0.800000000000000, 0.5) {1};
\node[rotate=0,black] at (0.5, 0.800000000000000) {2};
\node[rotate=0,black] at (0.200000000000000, 0.5) {3};
\node[rotate=0,black] at (0.5, 0.200000000000000) {2};
\end{tikzpicture}
} 
\newcommand\JeandelRaoIV{
\begin{tikzpicture}
[scale=0.900000000000000]
\tikzstyle{every node}=[font=\footnotesize]
\fill[green] (1, 0) -- (0.5, 0.5) -- (1, 1);
\fill[red] (0, 1) -- (0.5, 0.5) -- (1, 1);
\fill[green] (0, 0) -- (0.5, 0.5) -- (0, 1);
\fill[green] (0, 0) -- (0.5, 0.5) -- (1, 0);
\draw (1, 0) -- ++ (0,1);
\draw (0, 1) -- ++ (1,0);
\draw (0, 0) -- ++ (0,1);
\draw (0, 0) -- ++ (1,0);
\node[rotate=0,black] at (0.800000000000000, 0.5) {3};
\node[rotate=0,black] at (0.5, 0.800000000000000) {1};
\node[rotate=0,black] at (0.200000000000000, 0.5) {3};
\node[rotate=0,black] at (0.5, 0.200000000000000) {3};
\end{tikzpicture}
} 
\newcommand\JeandelRaoV{
\begin{tikzpicture}
[scale=0.900000000000000]
\tikzstyle{every node}=[font=\footnotesize]
\fill[white] (1, 0) -- (0.5, 0.5) -- (1, 1);
\fill[red] (0, 1) -- (0.5, 0.5) -- (1, 1);
\fill[green] (0, 0) -- (0.5, 0.5) -- (0, 1);
\fill[red] (0, 0) -- (0.5, 0.5) -- (1, 0);
\draw (1, 0) -- ++ (0,1);
\draw (0, 1) -- ++ (1,0);
\draw (0, 0) -- ++ (0,1);
\draw (0, 0) -- ++ (1,0);
\node[rotate=0,black] at (0.800000000000000, 0.5) {0};
\node[rotate=0,black] at (0.5, 0.800000000000000) {1};
\node[rotate=0,black] at (0.200000000000000, 0.5) {3};
\node[rotate=0,black] at (0.5, 0.200000000000000) {1};
\end{tikzpicture}
} 
\newcommand\JeandelRaoVI{
\begin{tikzpicture}
[scale=0.900000000000000]
\tikzstyle{every node}=[font=\footnotesize]
\fill[white] (1, 0) -- (0.5, 0.5) -- (1, 1);
\fill[white] (0, 1) -- (0.5, 0.5) -- (1, 1);
\fill[white] (0, 0) -- (0.5, 0.5) -- (0, 1);
\fill[red] (0, 0) -- (0.5, 0.5) -- (1, 0);
\draw (1, 0) -- ++ (0,1);
\draw (0, 1) -- ++ (1,0);
\draw (0, 0) -- ++ (0,1);
\draw (0, 0) -- ++ (1,0);
\node[rotate=0,black] at (0.800000000000000, 0.5) {0};
\node[rotate=0,black] at (0.5, 0.800000000000000) {0};
\node[rotate=0,black] at (0.200000000000000, 0.5) {0};
\node[rotate=0,black] at (0.5, 0.200000000000000) {1};
\end{tikzpicture}
} 
\newcommand\JeandelRaoVII{
\begin{tikzpicture}
[scale=0.900000000000000]
\tikzstyle{every node}=[font=\footnotesize]
\fill[green] (1, 0) -- (0.5, 0.5) -- (1, 1);
\fill[red] (0, 1) -- (0.5, 0.5) -- (1, 1);
\fill[white] (0, 0) -- (0.5, 0.5) -- (0, 1);
\fill[cyan] (0, 0) -- (0.5, 0.5) -- (1, 0);
\draw (1, 0) -- ++ (0,1);
\draw (0, 1) -- ++ (1,0);
\draw (0, 0) -- ++ (0,1);
\draw (0, 0) -- ++ (1,0);
\node[rotate=0,black] at (0.800000000000000, 0.5) {3};
\node[rotate=0,black] at (0.5, 0.800000000000000) {1};
\node[rotate=0,black] at (0.200000000000000, 0.5) {0};
\node[rotate=0,black] at (0.5, 0.200000000000000) {2};
\end{tikzpicture}
} 
\newcommand\JeandelRaoVIII{
\begin{tikzpicture}
[scale=0.900000000000000]
\tikzstyle{every node}=[font=\footnotesize]
\fill[white] (1, 0) -- (0.5, 0.5) -- (1, 1);
\fill[cyan] (0, 1) -- (0.5, 0.5) -- (1, 1);
\fill[red] (0, 0) -- (0.5, 0.5) -- (0, 1);
\fill[cyan] (0, 0) -- (0.5, 0.5) -- (1, 0);
\draw (1, 0) -- ++ (0,1);
\draw (0, 1) -- ++ (1,0);
\draw (0, 0) -- ++ (0,1);
\draw (0, 0) -- ++ (1,0);
\node[rotate=0,black] at (0.800000000000000, 0.5) {0};
\node[rotate=0,black] at (0.5, 0.800000000000000) {2};
\node[rotate=0,black] at (0.200000000000000, 0.5) {1};
\node[rotate=0,black] at (0.5, 0.200000000000000) {2};
\end{tikzpicture}
} 
\newcommand\JeandelRaoIX{
\begin{tikzpicture}
[scale=0.900000000000000]
\tikzstyle{every node}=[font=\footnotesize]
\fill[red] (1, 0) -- (0.5, 0.5) -- (1, 1);
\fill[cyan] (0, 1) -- (0.5, 0.5) -- (1, 1);
\fill[red] (0, 0) -- (0.5, 0.5) -- (0, 1);
\fill[lightgray] (0, 0) -- (0.5, 0.5) -- (1, 0);
\draw (1, 0) -- ++ (0,1);
\draw (0, 1) -- ++ (1,0);
\draw (0, 0) -- ++ (0,1);
\draw (0, 0) -- ++ (1,0);
\node[rotate=0,black] at (0.800000000000000, 0.5) {1};
\node[rotate=0,black] at (0.5, 0.800000000000000) {2};
\node[rotate=0,black] at (0.200000000000000, 0.5) {1};
\node[rotate=0,black] at (0.5, 0.200000000000000) {4};
\end{tikzpicture}
} 
\newcommand\JeandelRaoX{
\begin{tikzpicture}
[scale=0.900000000000000]
\tikzstyle{every node}=[font=\footnotesize]
\fill[green] (1, 0) -- (0.5, 0.5) -- (1, 1);
\fill[green] (0, 1) -- (0.5, 0.5) -- (1, 1);
\fill[red] (0, 0) -- (0.5, 0.5) -- (0, 1);
\fill[cyan] (0, 0) -- (0.5, 0.5) -- (1, 0);
\draw (1, 0) -- ++ (0,1);
\draw (0, 1) -- ++ (1,0);
\draw (0, 0) -- ++ (0,1);
\draw (0, 0) -- ++ (1,0);
\node[rotate=0,black] at (0.800000000000000, 0.5) {3};
\node[rotate=0,black] at (0.5, 0.800000000000000) {3};
\node[rotate=0,black] at (0.200000000000000, 0.5) {1};
\node[rotate=0,black] at (0.5, 0.200000000000000) {2};
\end{tikzpicture}
} 

%% file: article2_T0_tiles_no_id.tex
\begin{tikzpicture}
[scale=1]
\tikzstyle{every node}=[font=\footnotesize]
\fill[cyan] (1.0, 0.0) -- (0.5, 0.5) -- (1.0, 1.0);
\fill[lightgray] (0.0, 1.0) -- (0.5, 0.5) -- (1.0, 1.0);
\fill[cyan] (0.0, 0.0) -- (0.5, 0.5) -- (0.0, 1.0);
\fill[red] (0.0, 0.0) -- (0.5, 0.5) -- (1.0, 0.0);
\draw (1.0, 0.0) -- ++ (0,1);
\draw (0.0, 1.0) -- ++ (1,0);
\draw (0.0, 0.0) -- ++ (0,1);
\draw (0.0, 0.0) -- ++ (1,0);
\node[rotate=0,black] at (0.8, 0.5) {2};
\node[rotate=0,black] at (0.5, 0.8) {4};
\node[rotate=0,black] at (0.2, 0.5) {2};
\node[rotate=0,black] at (0.5, 0.2) {1};
\fill[cyan] (2.1, 0.0) -- (1.6, 0.5) -- (2.1, 1.0);
\fill[cyan] (1.1, 1.0) -- (1.6, 0.5) -- (2.1, 1.0);
\fill[cyan] (1.1, 0.0) -- (1.6, 0.5) -- (1.1, 1.0);
\fill[white] (1.1, 0.0) -- (1.6, 0.5) -- (2.1, 0.0);
\draw (2.1, 0.0) -- ++ (0,1);
\draw (1.1, 1.0) -- ++ (1,0);
\draw (1.1, 0.0) -- ++ (0,1);
\draw (1.1, 0.0) -- ++ (1,0);
\node[rotate=0,black] at (1.9000000000000001, 0.5) {2};
\node[rotate=0,black] at (1.6, 0.8) {2};
\node[rotate=0,black] at (1.3, 0.5) {2};
\node[rotate=0,black] at (1.6, 0.2) {0};
\fill[red] (3.2, 0.0) -- (2.7, 0.5) -- (3.2, 1.0);
\fill[red] (2.2, 1.0) -- (2.7, 0.5) -- (3.2, 1.0);
\fill[green] (2.2, 0.0) -- (2.7, 0.5) -- (2.2, 1.0);
\fill[red] (2.2, 0.0) -- (2.7, 0.5) -- (3.2, 0.0);
\draw (3.2, 0.0) -- ++ (0,1);
\draw (2.2, 1.0) -- ++ (1,0);
\draw (2.2, 0.0) -- ++ (0,1);
\draw (2.2, 0.0) -- ++ (1,0);
\node[rotate=0,black] at (3.0, 0.5) {1};
\node[rotate=0,black] at (2.7, 0.8) {1};
\node[rotate=0,black] at (2.4000000000000004, 0.5) {3};
\node[rotate=0,black] at (2.7, 0.2) {1};
\fill[red] (4.300000000000001, 0.0) -- (3.8000000000000003, 0.5) -- (4.300000000000001, 1.0);
\fill[cyan] (3.3000000000000003, 1.0) -- (3.8000000000000003, 0.5) -- (4.300000000000001, 1.0);
\fill[green] (3.3000000000000003, 0.0) -- (3.8000000000000003, 0.5) -- (3.3000000000000003, 1.0);
\fill[cyan] (3.3000000000000003, 0.0) -- (3.8000000000000003, 0.5) -- (4.300000000000001, 0.0);
\draw (4.300000000000001, 0.0) -- ++ (0,1);
\draw (3.3000000000000003, 1.0) -- ++ (1,0);
\draw (3.3000000000000003, 0.0) -- ++ (0,1);
\draw (3.3000000000000003, 0.0) -- ++ (1,0);
\node[rotate=0,black] at (4.1000000000000005, 0.5) {1};
\node[rotate=0,black] at (3.8000000000000003, 0.8) {2};
\node[rotate=0,black] at (3.5000000000000004, 0.5) {3};
\node[rotate=0,black] at (3.8000000000000003, 0.2) {2};
\fill[green] (5.4, 0.0) -- (4.9, 0.5) -- (5.4, 1.0);
\fill[red] (4.4, 1.0) -- (4.9, 0.5) -- (5.4, 1.0);
\fill[green] (4.4, 0.0) -- (4.9, 0.5) -- (4.4, 1.0);
\fill[green] (4.4, 0.0) -- (4.9, 0.5) -- (5.4, 0.0);
\draw (5.4, 0.0) -- ++ (0,1);
\draw (4.4, 1.0) -- ++ (1,0);
\draw (4.4, 0.0) -- ++ (0,1);
\draw (4.4, 0.0) -- ++ (1,0);
\node[rotate=0,black] at (5.2, 0.5) {3};
\node[rotate=0,black] at (4.9, 0.8) {1};
\node[rotate=0,black] at (4.6000000000000005, 0.5) {3};
\node[rotate=0,black] at (4.9, 0.2) {3};
\fill[white] (6.5, 0.0) -- (6.0, 0.5) -- (6.5, 1.0);
\fill[red] (5.5, 1.0) -- (6.0, 0.5) -- (6.5, 1.0);
\fill[green] (5.5, 0.0) -- (6.0, 0.5) -- (5.5, 1.0);
\fill[red] (5.5, 0.0) -- (6.0, 0.5) -- (6.5, 0.0);
\draw (6.5, 0.0) -- ++ (0,1);
\draw (5.5, 1.0) -- ++ (1,0);
\draw (5.5, 0.0) -- ++ (0,1);
\draw (5.5, 0.0) -- ++ (1,0);
\node[rotate=0,black] at (6.3, 0.5) {0};
\node[rotate=0,black] at (6.0, 0.8) {1};
\node[rotate=0,black] at (5.7, 0.5) {3};
\node[rotate=0,black] at (6.0, 0.2) {1};
\fill[white] (7.6000000000000005, 0.0) -- (7.1000000000000005, 0.5) -- (7.6000000000000005, 1.0);
\fill[white] (6.6000000000000005, 1.0) -- (7.1000000000000005, 0.5) -- (7.6000000000000005, 1.0);
\fill[white] (6.6000000000000005, 0.0) -- (7.1000000000000005, 0.5) -- (6.6000000000000005, 1.0);
\fill[red] (6.6000000000000005, 0.0) -- (7.1000000000000005, 0.5) -- (7.6000000000000005, 0.0);
\draw (7.6000000000000005, 0.0) -- ++ (0,1);
\draw (6.6000000000000005, 1.0) -- ++ (1,0);
\draw (6.6000000000000005, 0.0) -- ++ (0,1);
\draw (6.6000000000000005, 0.0) -- ++ (1,0);
\node[rotate=0,black] at (7.4, 0.5) {0};
\node[rotate=0,black] at (7.1000000000000005, 0.8) {0};
\node[rotate=0,black] at (6.800000000000001, 0.5) {0};
\node[rotate=0,black] at (7.1000000000000005, 0.2) {1};
\fill[green] (8.700000000000001, 0.0) -- (8.200000000000001, 0.5) -- (8.700000000000001, 1.0);
\fill[red] (7.700000000000001, 1.0) -- (8.200000000000001, 0.5) -- (8.700000000000001, 1.0);
\fill[white] (7.700000000000001, 0.0) -- (8.200000000000001, 0.5) -- (7.700000000000001, 1.0);
\fill[cyan] (7.700000000000001, 0.0) -- (8.200000000000001, 0.5) -- (8.700000000000001, 0.0);
\draw (8.700000000000001, 0.0) -- ++ (0,1);
\draw (7.700000000000001, 1.0) -- ++ (1,0);
\draw (7.700000000000001, 0.0) -- ++ (0,1);
\draw (7.700000000000001, 0.0) -- ++ (1,0);
\node[rotate=0,black] at (8.500000000000002, 0.5) {3};
\node[rotate=0,black] at (8.200000000000001, 0.8) {1};
\node[rotate=0,black] at (7.900000000000001, 0.5) {0};
\node[rotate=0,black] at (8.200000000000001, 0.2) {2};
\fill[white] (9.8, 0.0) -- (9.3, 0.5) -- (9.8, 1.0);
\fill[cyan] (8.8, 1.0) -- (9.3, 0.5) -- (9.8, 1.0);
\fill[red] (8.8, 0.0) -- (9.3, 0.5) -- (8.8, 1.0);
\fill[cyan] (8.8, 0.0) -- (9.3, 0.5) -- (9.8, 0.0);
\draw (9.8, 0.0) -- ++ (0,1);
\draw (8.8, 1.0) -- ++ (1,0);
\draw (8.8, 0.0) -- ++ (0,1);
\draw (8.8, 0.0) -- ++ (1,0);
\node[rotate=0,black] at (9.600000000000001, 0.5) {0};
\node[rotate=0,black] at (9.3, 0.8) {2};
\node[rotate=0,black] at (9.0, 0.5) {1};
\node[rotate=0,black] at (9.3, 0.2) {2};
\fill[red] (10.9, 0.0) -- (10.4, 0.5) -- (10.9, 1.0);
\fill[cyan] (9.9, 1.0) -- (10.4, 0.5) -- (10.9, 1.0);
\fill[red] (9.9, 0.0) -- (10.4, 0.5) -- (9.9, 1.0);
\fill[lightgray] (9.9, 0.0) -- (10.4, 0.5) -- (10.9, 0.0);
\draw (10.9, 0.0) -- ++ (0,1);
\draw (9.9, 1.0) -- ++ (1,0);
\draw (9.9, 0.0) -- ++ (0,1);
\draw (9.9, 0.0) -- ++ (1,0);
\node[rotate=0,black] at (10.700000000000001, 0.5) {1};
\node[rotate=0,black] at (10.4, 0.8) {2};
\node[rotate=0,black] at (10.1, 0.5) {1};
\node[rotate=0,black] at (10.4, 0.2) {4};
\fill[green] (12.0, 0.0) -- (11.5, 0.5) -- (12.0, 1.0);
\fill[green] (11.0, 1.0) -- (11.5, 0.5) -- (12.0, 1.0);
\fill[red] (11.0, 0.0) -- (11.5, 0.5) -- (11.0, 1.0);
\fill[cyan] (11.0, 0.0) -- (11.5, 0.5) -- (12.0, 0.0);
\draw (12.0, 0.0) -- ++ (0,1);
\draw (11.0, 1.0) -- ++ (1,0);
\draw (11.0, 0.0) -- ++ (0,1);
\draw (11.0, 0.0) -- ++ (1,0);
\node[rotate=0,black] at (11.8, 0.5) {3};
\node[rotate=0,black] at (11.5, 0.8) {3};
\node[rotate=0,black] at (11.2, 0.5) {1};
\node[rotate=0,black] at (11.5, 0.2) {2};
\end{tikzpicture}

%% file: article2_T0_6x1_tiling.tex
\begin{tikzpicture}[scale=1]
\tikzstyle{every node}=[font=\footnotesize]
\fill[green] (1, 0) -- (0.5, 0.5) -- (1, 1);
\fill[red] (0, 1) -- (0.5, 0.5) -- (1, 1);
\fill[white] (0, 0) -- (0.5, 0.5) -- (0, 1);
\fill[cyan] (0, 0) -- (0.5, 0.5) -- (1, 0);
\draw (0, 1) -- ++ (1,0);
\draw (0, 0) -- ++ (0,1);
\draw (0, 0) -- ++ (1,0);
\node[rotate=0,black] at (0.8, 0.5) {3};
\node[rotate=0,black] at (0.5, 0.8) {1};
\node[rotate=0,black] at (0.2, 0.5) {0};
\node[rotate=0,black] at (0.5, 0.2) {2};
\fill[green] (2, 0) -- (1.5, 0.5) -- (2, 1);
\fill[red] (1, 1) -- (1.5, 0.5) -- (2, 1);
\fill[green] (1, 0) -- (1.5, 0.5) -- (1, 1);
\fill[green] (1, 0) -- (1.5, 0.5) -- (2, 0);
\draw (1, 1) -- ++ (1,0);
\draw (1, 0) -- ++ (0,1);
\draw (1, 0) -- ++ (1,0);
\node[rotate=0,black] at (1.8, 0.5) {3};
\node[rotate=0,black] at (1.5, 0.8) {1};
\node[rotate=0,black] at (1.2, 0.5) {3};
\node[rotate=0,black] at (1.5, 0.2) {3};
\fill[red] (3, 0) -- (2.5, 0.5) -- (3, 1);
\fill[cyan] (2, 1) -- (2.5, 0.5) -- (3, 1);
\fill[green] (2, 0) -- (2.5, 0.5) -- (2, 1);
\fill[cyan] (2, 0) -- (2.5, 0.5) -- (3, 0);
\draw (2, 1) -- ++ (1,0);
\draw (2, 0) -- ++ (0,1);
\draw (2, 0) -- ++ (1,0);
\node[rotate=0,black] at (2.8, 0.5) {1};
\node[rotate=0,black] at (2.5, 0.8) {2};
\node[rotate=0,black] at (2.2, 0.5) {3};
\node[rotate=0,black] at (2.5, 0.2) {2};
\fill[green] (4, 0) -- (3.5, 0.5) -- (4, 1);
\fill[green] (3, 1) -- (3.5, 0.5) -- (4, 1);
\fill[red] (3, 0) -- (3.5, 0.5) -- (3, 1);
\fill[cyan] (3, 0) -- (3.5, 0.5) -- (4, 0);
\draw (3, 1) -- ++ (1,0);
\draw (3, 0) -- ++ (0,1);
\draw (3, 0) -- ++ (1,0);
\node[rotate=0,black] at (3.8, 0.5) {3};
\node[rotate=0,black] at (3.5, 0.8) {3};
\node[rotate=0,black] at (3.2, 0.5) {1};
\node[rotate=0,black] at (3.5, 0.2) {2};
\fill[red] (5, 0) -- (4.5, 0.5) -- (5, 1);
\fill[red] (4, 1) -- (4.5, 0.5) -- (5, 1);
\fill[green] (4, 0) -- (4.5, 0.5) -- (4, 1);
\fill[red] (4, 0) -- (4.5, 0.5) -- (5, 0);
\draw (4, 1) -- ++ (1,0);
\draw (4, 0) -- ++ (0,1);
\draw (4, 0) -- ++ (1,0);
\node[rotate=0,black] at (4.8, 0.5) {1};
\node[rotate=0,black] at (4.5, 0.8) {1};
\node[rotate=0,black] at (4.2, 0.5) {3};
\node[rotate=0,black] at (4.5, 0.2) {1};
\fill[white] (6, 0) -- (5.5, 0.5) -- (6, 1);
\fill[cyan] (5, 1) -- (5.5, 0.5) -- (6, 1);
\fill[red] (5, 0) -- (5.5, 0.5) -- (5, 1);
\fill[cyan] (5, 0) -- (5.5, 0.5) -- (6, 0);
\draw (6, 0) -- ++ (0,1);
\draw (5, 1) -- ++ (1,0);
\draw (5, 0) -- ++ (0,1);
\draw (5, 0) -- ++ (1,0);
\node[rotate=0,black] at (5.8, 0.5) {0};
\node[rotate=0,black] at (5.5, 0.8) {2};
\node[rotate=0,black] at (5.2, 0.5) {1};
\node[rotate=0,black] at (5.5, 0.2) {2};
\end{tikzpicture}

%% file: article2_T0_5x5_tiling_noid_nocolor.tex
\begin{tikzpicture}[scale=1]
\tikzstyle{every node}=[font=\footnotesize]
\draw (0, 0) -- ++ (0,1);
\draw (0, 0) -- ++ (1,0);
\node[rotate=0,black] at (0.8, 0.5) {3};
\node[rotate=0,black] at (0.5, 0.8) {1};
\node[rotate=0,black] at (0.2, 0.5) {0};
\node[rotate=0,black] at (0.5, 0.2) {2};
\draw (0, 1) -- ++ (0,1);
\draw (0, 1) -- ++ (1,0);
\node[rotate=0,black] at (0.8, 1.5) {0};
\node[rotate=0,black] at (0.5, 1.8) {0};
\node[rotate=0,black] at (0.2, 1.5) {0};
\node[rotate=0,black] at (0.5, 1.2) {1};
\draw (0, 2) -- ++ (0,1);
\draw (0, 2) -- ++ (1,0);
\node[rotate=0,black] at (0.8, 2.5) {2};
\node[rotate=0,black] at (0.5, 2.8) {2};
\node[rotate=0,black] at (0.2, 2.5) {2};
\node[rotate=0,black] at (0.5, 2.2) {0};
\draw (0, 3) -- ++ (0,1);
\draw (0, 3) -- ++ (1,0);
\node[rotate=0,black] at (0.8, 3.5) {3};
\node[rotate=0,black] at (0.5, 3.8) {1};
\node[rotate=0,black] at (0.2, 3.5) {0};
\node[rotate=0,black] at (0.5, 3.2) {2};
\draw (0, 5) -- ++ (1,0);
\draw (0, 4) -- ++ (0,1);
\draw (0, 4) -- ++ (1,0);
\node[rotate=0,black] at (0.8, 4.5) {1};
\node[rotate=0,black] at (0.5, 4.8) {1};
\node[rotate=0,black] at (0.2, 4.5) {3};
\node[rotate=0,black] at (0.5, 4.2) {1};
\draw (1, 0) -- ++ (0,1);
\draw (1, 0) -- ++ (1,0);
\node[rotate=0,black] at (1.8, 0.5) {3};
\node[rotate=0,black] at (1.5, 0.8) {1};
\node[rotate=0,black] at (1.2, 0.5) {3};
\node[rotate=0,black] at (1.5, 0.2) {3};
\draw (1, 1) -- ++ (0,1);
\draw (1, 1) -- ++ (1,0);
\node[rotate=0,black] at (1.8, 1.5) {0};
\node[rotate=0,black] at (1.5, 1.8) {0};
\node[rotate=0,black] at (1.2, 1.5) {0};
\node[rotate=0,black] at (1.5, 1.2) {1};
\draw (1, 2) -- ++ (0,1);
\draw (1, 2) -- ++ (1,0);
\node[rotate=0,black] at (1.8, 2.5) {2};
\node[rotate=0,black] at (1.5, 2.8) {2};
\node[rotate=0,black] at (1.2, 2.5) {2};
\node[rotate=0,black] at (1.5, 2.2) {0};
\draw (1, 3) -- ++ (0,1);
\draw (1, 3) -- ++ (1,0);
\node[rotate=0,black] at (1.8, 3.5) {1};
\node[rotate=0,black] at (1.5, 3.8) {2};
\node[rotate=0,black] at (1.2, 3.5) {3};
\node[rotate=0,black] at (1.5, 3.2) {2};
\draw (1, 5) -- ++ (1,0);
\draw (1, 4) -- ++ (0,1);
\draw (1, 4) -- ++ (1,0);
\node[rotate=0,black] at (1.8, 4.5) {0};
\node[rotate=0,black] at (1.5, 4.8) {2};
\node[rotate=0,black] at (1.2, 4.5) {1};
\node[rotate=0,black] at (1.5, 4.2) {2};
\draw (2, 0) -- ++ (0,1);
\draw (2, 0) -- ++ (1,0);
\node[rotate=0,black] at (2.8, 0.5) {1};
\node[rotate=0,black] at (2.5, 0.8) {2};
\node[rotate=0,black] at (2.2, 0.5) {3};
\node[rotate=0,black] at (2.5, 0.2) {2};
\draw (2, 1) -- ++ (0,1);
\draw (2, 1) -- ++ (1,0);
\node[rotate=0,black] at (2.8, 1.5) {3};
\node[rotate=0,black] at (2.5, 1.8) {1};
\node[rotate=0,black] at (2.2, 1.5) {0};
\node[rotate=0,black] at (2.5, 1.2) {2};
\draw (2, 2) -- ++ (0,1);
\draw (2, 2) -- ++ (1,0);
\node[rotate=0,black] at (2.8, 2.5) {2};
\node[rotate=0,black] at (2.5, 2.8) {4};
\node[rotate=0,black] at (2.2, 2.5) {2};
\node[rotate=0,black] at (2.5, 2.2) {1};
\draw (2, 3) -- ++ (0,1);
\draw (2, 3) -- ++ (1,0);
\node[rotate=0,black] at (2.8, 3.5) {1};
\node[rotate=0,black] at (2.5, 3.8) {2};
\node[rotate=0,black] at (2.2, 3.5) {1};
\node[rotate=0,black] at (2.5, 3.2) {4};
\draw (2, 5) -- ++ (1,0);
\draw (2, 4) -- ++ (0,1);
\draw (2, 4) -- ++ (1,0);
\node[rotate=0,black] at (2.8, 4.5) {3};
\node[rotate=0,black] at (2.5, 4.8) {1};
\node[rotate=0,black] at (2.2, 4.5) {0};
\node[rotate=0,black] at (2.5, 4.2) {2};
\draw (3, 0) -- ++ (0,1);
\draw (3, 0) -- ++ (1,0);
\node[rotate=0,black] at (3.8, 0.5) {3};
\node[rotate=0,black] at (3.5, 0.8) {3};
\node[rotate=0,black] at (3.2, 0.5) {1};
\node[rotate=0,black] at (3.5, 0.2) {2};
\draw (3, 1) -- ++ (0,1);
\draw (3, 1) -- ++ (1,0);
\node[rotate=0,black] at (3.8, 1.5) {3};
\node[rotate=0,black] at (3.5, 1.8) {1};
\node[rotate=0,black] at (3.2, 1.5) {3};
\node[rotate=0,black] at (3.5, 1.2) {3};
\draw (3, 2) -- ++ (0,1);
\draw (3, 2) -- ++ (1,0);
\node[rotate=0,black] at (3.8, 2.5) {2};
\node[rotate=0,black] at (3.5, 2.8) {4};
\node[rotate=0,black] at (3.2, 2.5) {2};
\node[rotate=0,black] at (3.5, 2.2) {1};
\draw (3, 3) -- ++ (0,1);
\draw (3, 3) -- ++ (1,0);
\node[rotate=0,black] at (3.8, 3.5) {1};
\node[rotate=0,black] at (3.5, 3.8) {2};
\node[rotate=0,black] at (3.2, 3.5) {1};
\node[rotate=0,black] at (3.5, 3.2) {4};
\draw (3, 5) -- ++ (1,0);
\draw (3, 4) -- ++ (0,1);
\draw (3, 4) -- ++ (1,0);
\node[rotate=0,black] at (3.8, 4.5) {1};
\node[rotate=0,black] at (3.5, 4.8) {2};
\node[rotate=0,black] at (3.2, 4.5) {3};
\node[rotate=0,black] at (3.5, 4.2) {2};
\draw (5, 0) -- ++ (0,1);
\draw (4, 0) -- ++ (0,1);
\draw (4, 0) -- ++ (1,0);
\node[rotate=0,black] at (4.8, 0.5) {0};
\node[rotate=0,black] at (4.5, 0.8) {1};
\node[rotate=0,black] at (4.2, 0.5) {3};
\node[rotate=0,black] at (4.5, 0.2) {1};
\draw (5, 1) -- ++ (0,1);
\draw (4, 1) -- ++ (0,1);
\draw (4, 1) -- ++ (1,0);
\node[rotate=0,black] at (4.8, 1.5) {0};
\node[rotate=0,black] at (4.5, 1.8) {1};
\node[rotate=0,black] at (4.2, 1.5) {3};
\node[rotate=0,black] at (4.5, 1.2) {1};
\draw (5, 2) -- ++ (0,1);
\draw (4, 2) -- ++ (0,1);
\draw (4, 2) -- ++ (1,0);
\node[rotate=0,black] at (4.8, 2.5) {2};
\node[rotate=0,black] at (4.5, 2.8) {4};
\node[rotate=0,black] at (4.2, 2.5) {2};
\node[rotate=0,black] at (4.5, 2.2) {1};
\draw (5, 3) -- ++ (0,1);
\draw (4, 3) -- ++ (0,1);
\draw (4, 3) -- ++ (1,0);
\node[rotate=0,black] at (4.8, 3.5) {1};
\node[rotate=0,black] at (4.5, 3.8) {2};
\node[rotate=0,black] at (4.2, 3.5) {1};
\node[rotate=0,black] at (4.5, 3.2) {4};
\draw (5, 4) -- ++ (0,1);
\draw (4, 5) -- ++ (1,0);
\draw (4, 4) -- ++ (0,1);
\draw (4, 4) -- ++ (1,0);
\node[rotate=0,black] at (4.8, 4.5) {3};
\node[rotate=0,black] at (4.5, 4.8) {3};
\node[rotate=0,black] at (4.2, 4.5) {1};
\node[rotate=0,black] at (4.5, 4.2) {2};
\end{tikzpicture}

%% file: article2_T0_5x5_matrix.tex
\left(\begin{array}{rrrrr}
2 & 8 & 7 & 3 & 10 \\
7 & 3 & 9 & 9 & 9 \\
1 & 1 & 0 & 0 & 0 \\
6 & 6 & 7 & 4 & 5 \\
7 & 4 & 3 & 10 & 5
\end{array}\right)

%% file: article2_T0_5x5_tiling_nocolor.tex
\begin{tikzpicture}[scale=1]
\tikzstyle{every node}=[font=\footnotesize]
\node[] at (0.5, 0.5) {7};
\draw (0, 0) -- ++ (0,1);
\draw (0, 0) -- ++ (1,0);
\node[rotate=0,black] at (0.8, 0.5) {3};
\node[rotate=0,black] at (0.5, 0.8) {1};
\node[rotate=0,black] at (0.2, 0.5) {0};
\node[rotate=0,black] at (0.5, 0.2) {2};
\node[] at (0.5, 1.5) {6};
\draw (0, 1) -- ++ (0,1);
\draw (0, 1) -- ++ (1,0);
\node[rotate=0,black] at (0.8, 1.5) {0};
\node[rotate=0,black] at (0.5, 1.8) {0};
\node[rotate=0,black] at (0.2, 1.5) {0};
\node[rotate=0,black] at (0.5, 1.2) {1};
\node[] at (0.5, 2.5) {1};
\draw (0, 2) -- ++ (0,1);
\draw (0, 2) -- ++ (1,0);
\node[rotate=0,black] at (0.8, 2.5) {2};
\node[rotate=0,black] at (0.5, 2.8) {2};
\node[rotate=0,black] at (0.2, 2.5) {2};
\node[rotate=0,black] at (0.5, 2.2) {0};
\node[] at (0.5, 3.5) {7};
\draw (0, 3) -- ++ (0,1);
\draw (0, 3) -- ++ (1,0);
\node[rotate=0,black] at (0.8, 3.5) {3};
\node[rotate=0,black] at (0.5, 3.8) {1};
\node[rotate=0,black] at (0.2, 3.5) {0};
\node[rotate=0,black] at (0.5, 3.2) {2};
\node[] at (0.5, 4.5) {2};
\draw (0, 5) -- ++ (1,0);
\draw (0, 4) -- ++ (0,1);
\draw (0, 4) -- ++ (1,0);
\node[rotate=0,black] at (0.8, 4.5) {1};
\node[rotate=0,black] at (0.5, 4.8) {1};
\node[rotate=0,black] at (0.2, 4.5) {3};
\node[rotate=0,black] at (0.5, 4.2) {1};
\node[] at (1.5, 0.5) {4};
\draw (1, 0) -- ++ (0,1);
\draw (1, 0) -- ++ (1,0);
\node[rotate=0,black] at (1.8, 0.5) {3};
\node[rotate=0,black] at (1.5, 0.8) {1};
\node[rotate=0,black] at (1.2, 0.5) {3};
\node[rotate=0,black] at (1.5, 0.2) {3};
\node[] at (1.5, 1.5) {6};
\draw (1, 1) -- ++ (0,1);
\draw (1, 1) -- ++ (1,0);
\node[rotate=0,black] at (1.8, 1.5) {0};
\node[rotate=0,black] at (1.5, 1.8) {0};
\node[rotate=0,black] at (1.2, 1.5) {0};
\node[rotate=0,black] at (1.5, 1.2) {1};
\node[] at (1.5, 2.5) {1};
\draw (1, 2) -- ++ (0,1);
\draw (1, 2) -- ++ (1,0);
\node[rotate=0,black] at (1.8, 2.5) {2};
\node[rotate=0,black] at (1.5, 2.8) {2};
\node[rotate=0,black] at (1.2, 2.5) {2};
\node[rotate=0,black] at (1.5, 2.2) {0};
\node[] at (1.5, 3.5) {3};
\draw (1, 3) -- ++ (0,1);
\draw (1, 3) -- ++ (1,0);
\node[rotate=0,black] at (1.8, 3.5) {1};
\node[rotate=0,black] at (1.5, 3.8) {2};
\node[rotate=0,black] at (1.2, 3.5) {3};
\node[rotate=0,black] at (1.5, 3.2) {2};
\node[] at (1.5, 4.5) {8};
\draw (1, 5) -- ++ (1,0);
\draw (1, 4) -- ++ (0,1);
\draw (1, 4) -- ++ (1,0);
\node[rotate=0,black] at (1.8, 4.5) {0};
\node[rotate=0,black] at (1.5, 4.8) {2};
\node[rotate=0,black] at (1.2, 4.5) {1};
\node[rotate=0,black] at (1.5, 4.2) {2};
\node[] at (2.5, 0.5) {3};
\draw (2, 0) -- ++ (0,1);
\draw (2, 0) -- ++ (1,0);
\node[rotate=0,black] at (2.8, 0.5) {1};
\node[rotate=0,black] at (2.5, 0.8) {2};
\node[rotate=0,black] at (2.2, 0.5) {3};
\node[rotate=0,black] at (2.5, 0.2) {2};
\node[] at (2.5, 1.5) {7};
\draw (2, 1) -- ++ (0,1);
\draw (2, 1) -- ++ (1,0);
\node[rotate=0,black] at (2.8, 1.5) {3};
\node[rotate=0,black] at (2.5, 1.8) {1};
\node[rotate=0,black] at (2.2, 1.5) {0};
\node[rotate=0,black] at (2.5, 1.2) {2};
\node[] at (2.5, 2.5) {0};
\draw (2, 2) -- ++ (0,1);
\draw (2, 2) -- ++ (1,0);
\node[rotate=0,black] at (2.8, 2.5) {2};
\node[rotate=0,black] at (2.5, 2.8) {4};
\node[rotate=0,black] at (2.2, 2.5) {2};
\node[rotate=0,black] at (2.5, 2.2) {1};
\node[] at (2.5, 3.5) {9};
\draw (2, 3) -- ++ (0,1);
\draw (2, 3) -- ++ (1,0);
\node[rotate=0,black] at (2.8, 3.5) {1};
\node[rotate=0,black] at (2.5, 3.8) {2};
\node[rotate=0,black] at (2.2, 3.5) {1};
\node[rotate=0,black] at (2.5, 3.2) {4};
\node[] at (2.5, 4.5) {7};
\draw (2, 5) -- ++ (1,0);
\draw (2, 4) -- ++ (0,1);
\draw (2, 4) -- ++ (1,0);
\node[rotate=0,black] at (2.8, 4.5) {3};
\node[rotate=0,black] at (2.5, 4.8) {1};
\node[rotate=0,black] at (2.2, 4.5) {0};
\node[rotate=0,black] at (2.5, 4.2) {2};
\node[] at (3.5, 0.5) {10};
\draw (3, 0) -- ++ (0,1);
\draw (3, 0) -- ++ (1,0);
\node[rotate=0,black] at (3.8, 0.5) {3};
\node[rotate=0,black] at (3.5, 0.8) {3};
\node[rotate=0,black] at (3.2, 0.5) {1};
\node[rotate=0,black] at (3.5, 0.2) {2};
\node[] at (3.5, 1.5) {4};
\draw (3, 1) -- ++ (0,1);
\draw (3, 1) -- ++ (1,0);
\node[rotate=0,black] at (3.8, 1.5) {3};
\node[rotate=0,black] at (3.5, 1.8) {1};
\node[rotate=0,black] at (3.2, 1.5) {3};
\node[rotate=0,black] at (3.5, 1.2) {3};
\node[] at (3.5, 2.5) {0};
\draw (3, 2) -- ++ (0,1);
\draw (3, 2) -- ++ (1,0);
\node[rotate=0,black] at (3.8, 2.5) {2};
\node[rotate=0,black] at (3.5, 2.8) {4};
\node[rotate=0,black] at (3.2, 2.5) {2};
\node[rotate=0,black] at (3.5, 2.2) {1};
\node[] at (3.5, 3.5) {9};
\draw (3, 3) -- ++ (0,1);
\draw (3, 3) -- ++ (1,0);
\node[rotate=0,black] at (3.8, 3.5) {1};
\node[rotate=0,black] at (3.5, 3.8) {2};
\node[rotate=0,black] at (3.2, 3.5) {1};
\node[rotate=0,black] at (3.5, 3.2) {4};
\node[] at (3.5, 4.5) {3};
\draw (3, 5) -- ++ (1,0);
\draw (3, 4) -- ++ (0,1);
\draw (3, 4) -- ++ (1,0);
\node[rotate=0,black] at (3.8, 4.5) {1};
\node[rotate=0,black] at (3.5, 4.8) {2};
\node[rotate=0,black] at (3.2, 4.5) {3};
\node[rotate=0,black] at (3.5, 4.2) {2};
\node[] at (4.5, 0.5) {5};
\draw (5, 0) -- ++ (0,1);
\draw (4, 0) -- ++ (0,1);
\draw (4, 0) -- ++ (1,0);
\node[rotate=0,black] at (4.8, 0.5) {0};
\node[rotate=0,black] at (4.5, 0.8) {1};
\node[rotate=0,black] at (4.2, 0.5) {3};
\node[rotate=0,black] at (4.5, 0.2) {1};
\node[] at (4.5, 1.5) {5};
\draw (5, 1) -- ++ (0,1);
\draw (4, 1) -- ++ (0,1);
\draw (4, 1) -- ++ (1,0);
\node[rotate=0,black] at (4.8, 1.5) {0};
\node[rotate=0,black] at (4.5, 1.8) {1};
\node[rotate=0,black] at (4.2, 1.5) {3};
\node[rotate=0,black] at (4.5, 1.2) {1};
\node[] at (4.5, 2.5) {0};
\draw (5, 2) -- ++ (0,1);
\draw (4, 2) -- ++ (0,1);
\draw (4, 2) -- ++ (1,0);
\node[rotate=0,black] at (4.8, 2.5) {2};
\node[rotate=0,black] at (4.5, 2.8) {4};
\node[rotate=0,black] at (4.2, 2.5) {2};
\node[rotate=0,black] at (4.5, 2.2) {1};
\node[] at (4.5, 3.5) {9};
\draw (5, 3) -- ++ (0,1);
\draw (4, 3) -- ++ (0,1);
\draw (4, 3) -- ++ (1,0);
\node[rotate=0,black] at (4.8, 3.5) {1};
\node[rotate=0,black] at (4.5, 3.8) {2};
\node[rotate=0,black] at (4.2, 3.5) {1};
\node[rotate=0,black] at (4.5, 3.2) {4};
\node[] at (4.5, 4.5) {10};
\draw (5, 4) -- ++ (0,1);
\draw (4, 5) -- ++ (1,0);
\draw (4, 4) -- ++ (0,1);
\draw (4, 4) -- ++ (1,0);
\node[rotate=0,black] at (4.8, 4.5) {3};
\node[rotate=0,black] at (4.5, 4.8) {3};
\node[rotate=0,black] at (4.2, 4.5) {1};
\node[rotate=0,black] at (4.5, 4.2) {2};
\end{tikzpicture}

%% file: article2_T0_0_9_domino.tex
\begin{tikzpicture}[scale=1]
\tikzstyle{every node}=[font=\footnotesize]
\fill[cyan] (1, 0) -- (0.5, 0.5) -- (1, 1);
\fill[lightgray] (0, 1) -- (0.5, 0.5) -- (1, 1);
\fill[cyan] (0, 0) -- (0.5, 0.5) -- (0, 1);
\fill[red] (0, 0) -- (0.5, 0.5) -- (1, 0);
\node[] at (0.5, 0.5) {0};
\draw (1, 0) -- ++ (0,1);
\draw (0, 0) -- ++ (0,1);
\draw (0, 0) -- ++ (1,0);
\node[rotate=0,black] at (0.8, 0.5) {2};
\node[rotate=0,black] at (0.5, 0.8) {4};
\node[rotate=0,black] at (0.2, 0.5) {2};
\node[rotate=0,black] at (0.5, 0.2) {1};
\fill[red] (1, 1) -- (0.5, 1.5) -- (1, 2);
\fill[cyan] (0, 2) -- (0.5, 1.5) -- (1, 2);
\fill[red] (0, 1) -- (0.5, 1.5) -- (0, 2);
\fill[lightgray] (0, 1) -- (0.5, 1.5) -- (1, 1);
\node[] at (0.5, 1.5) {9};
\draw (1, 1) -- ++ (0,1);
\draw (0, 2) -- ++ (1,0);
\draw (0, 1) -- ++ (0,1);
\draw (0, 1) -- ++ (1,0);
\node[rotate=0,black] at (0.8, 1.5) {1};
\node[rotate=0,black] at (0.5, 1.8) {2};
\node[rotate=0,black] at (0.2, 1.5) {1};
\node[rotate=0,black] at (0.5, 1.2) {4};
\end{tikzpicture}

%% file: article2_T0_1_3_domino.tex
\begin{tikzpicture}[scale=1]
\tikzstyle{every node}=[font=\footnotesize]
\fill[cyan] (1, 0) -- (0.5, 0.5) -- (1, 1);
\fill[cyan] (0, 1) -- (0.5, 0.5) -- (1, 1);
\fill[cyan] (0, 0) -- (0.5, 0.5) -- (0, 1);
\fill[white] (0, 0) -- (0.5, 0.5) -- (1, 0);
\node[] at (0.5, 0.5) {1};
\draw (1, 0) -- ++ (0,1);
\draw (0, 0) -- ++ (0,1);
\draw (0, 0) -- ++ (1,0);
\node[rotate=0,black] at (0.8, 0.5) {2};
\node[rotate=0,black] at (0.5, 0.8) {2};
\node[rotate=0,black] at (0.2, 0.5) {2};
\node[rotate=0,black] at (0.5, 0.2) {0};
\fill[red] (1, 1) -- (0.5, 1.5) -- (1, 2);
\fill[cyan] (0, 2) -- (0.5, 1.5) -- (1, 2);
\fill[green] (0, 1) -- (0.5, 1.5) -- (0, 2);
\fill[cyan] (0, 1) -- (0.5, 1.5) -- (1, 1);
\node[] at (0.5, 1.5) {3};
\draw (1, 1) -- ++ (0,1);
\draw (0, 2) -- ++ (1,0);
\draw (0, 1) -- ++ (0,1);
\draw (0, 1) -- ++ (1,0);
\node[rotate=0,black] at (0.8, 1.5) {1};
\node[rotate=0,black] at (0.5, 1.8) {2};
\node[rotate=0,black] at (0.2, 1.5) {3};
\node[rotate=0,black] at (0.5, 1.2) {2};
\end{tikzpicture}

%% file: article2_T0_1_7_domino.tex
\begin{tikzpicture}[scale=1]
\tikzstyle{every node}=[font=\footnotesize]
\fill[cyan] (1, 0) -- (0.5, 0.5) -- (1, 1);
\fill[cyan] (0, 1) -- (0.5, 0.5) -- (1, 1);
\fill[cyan] (0, 0) -- (0.5, 0.5) -- (0, 1);
\fill[white] (0, 0) -- (0.5, 0.5) -- (1, 0);
\node[] at (0.5, 0.5) {1};
\draw (1, 0) -- ++ (0,1);
\draw (0, 0) -- ++ (0,1);
\draw (0, 0) -- ++ (1,0);
\node[rotate=0,black] at (0.8, 0.5) {2};
\node[rotate=0,black] at (0.5, 0.8) {2};
\node[rotate=0,black] at (0.2, 0.5) {2};
\node[rotate=0,black] at (0.5, 0.2) {0};
\fill[green] (1, 1) -- (0.5, 1.5) -- (1, 2);
\fill[red] (0, 2) -- (0.5, 1.5) -- (1, 2);
\fill[white] (0, 1) -- (0.5, 1.5) -- (0, 2);
\fill[cyan] (0, 1) -- (0.5, 1.5) -- (1, 1);
\node[] at (0.5, 1.5) {7};
\draw (1, 1) -- ++ (0,1);
\draw (0, 2) -- ++ (1,0);
\draw (0, 1) -- ++ (0,1);
\draw (0, 1) -- ++ (1,0);
\node[rotate=0,black] at (0.8, 1.5) {3};
\node[rotate=0,black] at (0.5, 1.8) {1};
\node[rotate=0,black] at (0.2, 1.5) {0};
\node[rotate=0,black] at (0.5, 1.2) {2};
\end{tikzpicture}

%% file: article2_T0_1_8_domino.tex
\begin{tikzpicture}[scale=1]
\tikzstyle{every node}=[font=\footnotesize]
\fill[cyan] (1, 0) -- (0.5, 0.5) -- (1, 1);
\fill[cyan] (0, 1) -- (0.5, 0.5) -- (1, 1);
\fill[cyan] (0, 0) -- (0.5, 0.5) -- (0, 1);
\fill[white] (0, 0) -- (0.5, 0.5) -- (1, 0);
\node[] at (0.5, 0.5) {1};
\draw (1, 0) -- ++ (0,1);
\draw (0, 0) -- ++ (0,1);
\draw (0, 0) -- ++ (1,0);
\node[rotate=0,black] at (0.8, 0.5) {2};
\node[rotate=0,black] at (0.5, 0.8) {2};
\node[rotate=0,black] at (0.2, 0.5) {2};
\node[rotate=0,black] at (0.5, 0.2) {0};
\fill[white] (1, 1) -- (0.5, 1.5) -- (1, 2);
\fill[cyan] (0, 2) -- (0.5, 1.5) -- (1, 2);
\fill[red] (0, 1) -- (0.5, 1.5) -- (0, 2);
\fill[cyan] (0, 1) -- (0.5, 1.5) -- (1, 1);
\node[] at (0.5, 1.5) {8};
\draw (1, 1) -- ++ (0,1);
\draw (0, 2) -- ++ (1,0);
\draw (0, 1) -- ++ (0,1);
\draw (0, 1) -- ++ (1,0);
\node[rotate=0,black] at (0.8, 1.5) {0};
\node[rotate=0,black] at (0.5, 1.8) {2};
\node[rotate=0,black] at (0.2, 1.5) {1};
\node[rotate=0,black] at (0.5, 1.2) {2};
\end{tikzpicture}

%% file: article2_T0_1_10_domino.tex
\begin{tikzpicture}[scale=1]
\tikzstyle{every node}=[font=\footnotesize]
\fill[cyan] (1, 0) -- (0.5, 0.5) -- (1, 1);
\fill[cyan] (0, 1) -- (0.5, 0.5) -- (1, 1);
\fill[cyan] (0, 0) -- (0.5, 0.5) -- (0, 1);
\fill[white] (0, 0) -- (0.5, 0.5) -- (1, 0);
\node[] at (0.5, 0.5) {1};
\draw (1, 0) -- ++ (0,1);
\draw (0, 0) -- ++ (0,1);
\draw (0, 0) -- ++ (1,0);
\node[rotate=0,black] at (0.8, 0.5) {2};
\node[rotate=0,black] at (0.5, 0.8) {2};
\node[rotate=0,black] at (0.2, 0.5) {2};
\node[rotate=0,black] at (0.5, 0.2) {0};
\fill[green] (1, 1) -- (0.5, 1.5) -- (1, 2);
\fill[green] (0, 2) -- (0.5, 1.5) -- (1, 2);
\fill[red] (0, 1) -- (0.5, 1.5) -- (0, 2);
\fill[cyan] (0, 1) -- (0.5, 1.5) -- (1, 1);
\node[] at (0.5, 1.5) {10};
\draw (1, 1) -- ++ (0,1);
\draw (0, 2) -- ++ (1,0);
\draw (0, 1) -- ++ (0,1);
\draw (0, 1) -- ++ (1,0);
\node[rotate=0,black] at (0.8, 1.5) {3};
\node[rotate=0,black] at (0.5, 1.8) {3};
\node[rotate=0,black] at (0.2, 1.5) {1};
\node[rotate=0,black] at (0.5, 1.2) {2};
\end{tikzpicture}

%% file: article2_omega_X.tex
\newcommand\omegaO{\arraycolsep=1.4pt
 \begin{array}{lllllll}
0\mapsto \left(2\right)
,&
1\mapsto \left(3\right)
,&
2\mapsto \left(4\right)
,&
3\mapsto \left(5\right)
,&
4\mapsto \left(6\right)
,&
5\mapsto \left(7\right)
,&
6\mapsto \left(8\right)
,\\
7\mapsto \left(10\right)
,&
8\mapsto \left(\begin{array}{r}
9 \\
0
\end{array}\right)
,&
9\mapsto \left(\begin{array}{r}
3 \\
1
\end{array}\right)
,&
10\mapsto \left(\begin{array}{r}
7 \\
1
\end{array}\right)
,&
11\mapsto \left(\begin{array}{r}
8 \\
1
\end{array}\right)
,&
12\mapsto \left(\begin{array}{r}
10 \\
1
\end{array}\right)
.
\end{array}}
\newcommand\omegaI{\arraycolsep=1.4pt
 \begin{array}{llllll}
0\mapsto \left(0\right)
,&
1\mapsto \left(1\right)
,&
2\mapsto \left(2\right)
,&
3\mapsto \left(3\right)
,&
4\mapsto \left(4\right)
,&
5\mapsto \left(5\right)
,\\
6\mapsto \left(6\right)
,&
7\mapsto \left(7\right)
,&
8\mapsto \left(\begin{array}{r}
1 \\
8
\end{array}\right)
,&
9\mapsto \left(\begin{array}{r}
5 \\
8
\end{array}\right)
,&
10\mapsto \left(\begin{array}{r}
6 \\
8
\end{array}\right)
,&
11\mapsto \left(\begin{array}{r}
7 \\
8
\end{array}\right)
,\\
12\mapsto \left(\begin{array}{r}
1 \\
9
\end{array}\right)
,&
13\mapsto \left(\begin{array}{r}
5 \\
9
\end{array}\right)
,&
14\mapsto \left(\begin{array}{r}
6 \\
9
\end{array}\right)
,&
15\mapsto \left(\begin{array}{r}
0 \\
10
\end{array}\right)
,&
16\mapsto \left(\begin{array}{r}
3 \\
10
\end{array}\right)
,&
17\mapsto \left(\begin{array}{r}
5 \\
11
\end{array}\right)
,\\
18\mapsto \left(\begin{array}{r}
7 \\
11
\end{array}\right)
,&
19\mapsto \left(\begin{array}{r}
2 \\
12
\end{array}\right)
.
\end{array}}
\newcommand\omegaII{\arraycolsep=1.4pt
 \begin{array}{llllll}
0\mapsto \left(2\right)
,&
1\mapsto \left(3\right)
,&
2\mapsto \left(4\right)
,&
3\mapsto \left(5\right)
,&
4\mapsto \left(\begin{array}{r}
1 \\
8
\end{array}\right)
,&
5\mapsto \left(\begin{array}{r}
5 \\
8
\end{array}\right)
,\\
6\mapsto \left(\begin{array}{r}
6 \\
8
\end{array}\right)
,&
7\mapsto \left(\begin{array}{r}
0 \\
9
\end{array}\right)
,&
8\mapsto \left(\begin{array}{r}
3 \\
9
\end{array}\right)
,&
9\mapsto \left(\begin{array}{r}
5 \\
10
\end{array}\right)
,&
10\mapsto \left(\begin{array}{r}
7 \\
10
\end{array}\right)
,&
11\mapsto \left(\begin{array}{r}
2 \\
11
\end{array}\right)
,\\
12\mapsto \left(\begin{array}{r}
5 \\
12
\end{array}\right)
,&
13\mapsto \left(\begin{array}{r}
6 \\
12
\end{array}\right)
,&
14\mapsto \left(\begin{array}{r}
4 \\
13
\end{array}\right)
,&
15\mapsto \left(\begin{array}{r}
5 \\
14
\end{array}\right)
,&
16\mapsto \left(\begin{array}{r}
3 \\
15
\end{array}\right)
,&
17\mapsto \left(\begin{array}{r}
3 \\
16
\end{array}\right)
,\\
18\mapsto \left(\begin{array}{r}
4 \\
16
\end{array}\right)
,&
19\mapsto \left(\begin{array}{r}
3 \\
17
\end{array}\right)
,&
20\mapsto \left(\begin{array}{r}
2 \\
18
\end{array}\right)
,&
21\mapsto \left(\begin{array}{r}
0 \\
19
\end{array}\right)
,&
22\mapsto \left(\begin{array}{r}
3 \\
19
\end{array}\right)
,&
23\mapsto \left(\begin{array}{r}
4 \\
19
\end{array}\right)
.
\end{array}}
\newcommand\omegaIIIp{\arraycolsep=1.4pt
 \begin{array}{llllll}
0\mapsto \left(5\right)
,&
1\mapsto \left(8\right)
,&
2\mapsto \left(11\right)
,&
3\mapsto \left(12\right)
,&
4\mapsto \left(14\right)
,&
5\mapsto \left(15\right)
,\\
6\mapsto \left(16\right)
,&
7\mapsto \left(17\right)
,&
8\mapsto \left(18\right)
,&
9\mapsto \left(19\right)
,&
10\mapsto \left(20\right)
,&
11\mapsto \left(22\right)
,\\
12\mapsto \left(23\right)
,&
13\mapsto \left(\begin{array}{r}
3 \\
4
\end{array}\right)
,&
14\mapsto \left(\begin{array}{r}
2 \\
5
\end{array}\right)
,&
15\mapsto \left(\begin{array}{r}
3 \\
6
\end{array}\right)
,&
16\mapsto \left(\begin{array}{r}
1 \\
7
\end{array}\right)
,&
17\mapsto \left(\begin{array}{r}
2 \\
7
\end{array}\right)
,\\
18\mapsto \left(\begin{array}{r}
1 \\
8
\end{array}\right)
,&
19\mapsto \left(\begin{array}{r}
2 \\
8
\end{array}\right)
,&
20\mapsto \left(\begin{array}{r}
1 \\
9
\end{array}\right)
,&
21\mapsto \left(\begin{array}{r}
0 \\
10
\end{array}\right)
,&
22\mapsto \left(\begin{array}{r}
2 \\
11
\end{array}\right)
,&
23\mapsto \left(\begin{array}{r}
2 \\
12
\end{array}\right)
,\\
24\mapsto \left(\begin{array}{r}
3 \\
13
\end{array}\right)
,&
25\mapsto \left(\begin{array}{r}
2 \\
15
\end{array}\right)
,&
26\mapsto \left(\begin{array}{r}
2 \\
16
\end{array}\right)
,&
27\mapsto \left(\begin{array}{r}
2 \\
20
\end{array}\right)
,&
28\mapsto \left(\begin{array}{r}
2 \\
21
\end{array}\right)
,&
29\mapsto \left(\begin{array}{r}
2 \\
22
\end{array}\right)
.
\end{array}}
\newcommand\iotaIII{\arraycolsep=1.4pt
 \begin{array}{lllllll}
0\mapsto \left(0\right)
,&
1\mapsto \left(1\right)
,&
2\mapsto \left(2\right)
,&
3\mapsto \left(3\right)
,&
4\mapsto \left(4\right)
,&
5\mapsto \left(5\right)
,&
6\mapsto \left(6\right)
,\\
7\mapsto \left(7\right)
,&
8\mapsto \left(8\right)
,&
9\mapsto \left(9\right)
,&
10\mapsto \left(10\right)
,&
11\mapsto \left(11\right)
,&
12\mapsto \left(12\right)
,&
13\mapsto \left(13\right)
,\\
14\mapsto \left(14\right)
,&
15\mapsto \left(15\right)
,&
16\mapsto \left(16\right)
,&
17\mapsto \left(17\right)
,&
18\mapsto \left(18\right)
,&
19\mapsto \left(19\right)
,&
20\mapsto \left(20\right)
,\\
21\mapsto \left(21\right)
,&
22\mapsto \left(22\right)
,&
23\mapsto \left(23\right)
,&
24\mapsto \left(25\right)
,&
25\mapsto \left(26\right)
,&
26\mapsto \left(27\right)
,&
27\mapsto \left(29\right)
.
\end{array}}
\newcommand\omegaOtoIV{\arraycolsep=1.4pt
 \begin{array}{lllllll}
0\mapsto \left(\begin{array}{r}
7 \\
3 \\
9 \\
0
\end{array}\right)
,&
1\mapsto \left(\begin{array}{r}
5 \\
7 \\
9 \\
0
\end{array}\right)
,&
2\mapsto \left(\begin{array}{r}
4 \\
10 \\
9 \\
0
\end{array}\right)
,&
3\mapsto \left(\begin{array}{r}
7 \\
3 \\
3 \\
1
\end{array}\right)
,&
4\mapsto \left(\begin{array}{r}
6 \\
7 \\
3 \\
1
\end{array}\right)
,&
5\mapsto \left(\begin{array}{r}
7 \\
8 \\
3 \\
1
\end{array}\right)
,&
6\mapsto \left(\begin{array}{r}
5 \\
2 \\
7 \\
1
\end{array}\right)
,\\
7\mapsto \left(\begin{array}{r}
5 \\
5 \\
7 \\
1
\end{array}\right)
,&
8\mapsto \left(\begin{array}{r}
6 \\
5 \\
7 \\
1
\end{array}\right)
,&
9\mapsto \left(\begin{array}{r}
5 \\
7 \\
8 \\
1
\end{array}\right)
,&
10\mapsto \left(\begin{array}{r}
4 \\
10 \\
8 \\
1
\end{array}\right)
,&
11\mapsto \left(\begin{array}{r}
5 \\
4 \\
10 \\
1
\end{array}\right)
,&
12\mapsto \left(\begin{array}{r}
6 \\
4 \\
10 \\
1
\end{array}\right)
,&
13\mapsto \left(\begin{array}{r}
7 \\
3 \\
3 \\
9 \\
0
\end{array}\right)
,\\
14\mapsto \left(\begin{array}{r}
6 \\
7 \\
3 \\
9 \\
0
\end{array}\right)
,&
15\mapsto \left(\begin{array}{r}
7 \\
8 \\
3 \\
9 \\
0
\end{array}\right)
,&
16\mapsto \left(\begin{array}{r}
5 \\
2 \\
7 \\
9 \\
0
\end{array}\right)
,&
17\mapsto \left(\begin{array}{r}
6 \\
2 \\
7 \\
9 \\
0
\end{array}\right)
,&
18\mapsto \left(\begin{array}{r}
5 \\
5 \\
7 \\
9 \\
0
\end{array}\right)
,&
19\mapsto \left(\begin{array}{r}
6 \\
5 \\
7 \\
9 \\
0
\end{array}\right)
,&
20\mapsto \left(\begin{array}{r}
5 \\
7 \\
8 \\
9 \\
0
\end{array}\right)
,\\
21\mapsto \left(\begin{array}{r}
4 \\
10 \\
8 \\
9 \\
0
\end{array}\right)
,&
22\mapsto \left(\begin{array}{r}
6 \\
4 \\
10 \\
9 \\
0
\end{array}\right)
,&
23\mapsto \left(\begin{array}{r}
6 \\
7 \\
3 \\
3 \\
1
\end{array}\right)
,&
24\mapsto \left(\begin{array}{r}
6 \\
7 \\
8 \\
3 \\
1
\end{array}\right)
,&
25\mapsto \left(\begin{array}{r}
6 \\
5 \\
2 \\
7 \\
1
\end{array}\right)
,&
26\mapsto \left(\begin{array}{r}
6 \\
4 \\
10 \\
8 \\
1
\end{array}\right)
,&
27\mapsto \left(\begin{array}{r}
6 \\
5 \\
4 \\
10 \\
1
\end{array}\right)
.
\end{array}}
\newcommand\omegaIV{\arraycolsep=1.4pt
 \begin{array}{lllllll}
0\mapsto \left(0\right)
,&
1\mapsto \left(1\right)
,&
2\mapsto \left(2\right)
,&
3\mapsto \left(3\right)
,&
4\mapsto \left(4\right)
,&
5\mapsto \left(5\right)
,&
6\mapsto \left(6\right)
,\\
7\mapsto \left(7\right)
,&
8\mapsto \left(8\right)
,&
9\mapsto \left(9\right)
,&
10\mapsto \left(10\right)
,&
11\mapsto \left(11\right)
,&
12\mapsto \left(12\right)
,&
13\mapsto \left(13\right)
,\\
14\mapsto \left(14\right)
,&
15\mapsto \left(15\right)
,&
16\mapsto \left(16\right)
,&
17\mapsto \left(17\right)
,&
18\mapsto \left(18\right)
,&
19\mapsto \left(19\right)
,&
20\mapsto \left(20\right)
,\\
21\mapsto \left(21\right)
,&
22\mapsto \left(22\right)
,&
23\mapsto \left(22\right)
,&
24\mapsto \left(23\right)
,&
25\mapsto \left(24\right)
,&
26\mapsto \left(25\right)
,&
27\mapsto \left(26\right)
,\\
28\mapsto \left(27\right)
.
\end{array}}
\newcommand\omegaV{\arraycolsep=1.4pt
 \begin{array}{llllll}
0\mapsto \left(0\right)
,&
1\mapsto \left(1\right)
,&
2\mapsto \left(2\right)
,&
3\mapsto \left(3\right)
,&
4\mapsto \left(4\right)
,&
5\mapsto \left(5\right)
,\\
6\mapsto \left(6\right)
,&
7\mapsto \left(7\right)
,&
8\mapsto \left(8\right)
,&
9\mapsto \left(9\right)
,&
10\mapsto \left(10\right)
,&
11\mapsto \left(11\right)
,\\
12\mapsto \left(12\right)
,&
13\mapsto \left(13\right)
,&
14\mapsto \left(14\right)
,&
15\mapsto \left(15\right)
,&
16\mapsto \left(16\right)
,&
17\mapsto \left(17\right)
,\\
18\mapsto \left(18\right)
,&
19\mapsto \left(19\right)
,&
20\mapsto \left(20\right)
,&
21\mapsto \left(21\right)
,&
22\mapsto \left(22\right)
,&
23\mapsto \left(23\right)
,\\
24\mapsto \left(24\right)
,&
25\mapsto \left(25\right)
,&
26\mapsto \left(26\right)
,&
27\mapsto \left(27\right)
,&
28\mapsto \left(28\right)
.
\end{array}}
\newcommand\omegaVI{\arraycolsep=1.4pt
 \begin{array}{llll}
0\mapsto \left(2\right)
,&
1\mapsto \left(9\right)
,&
2\mapsto \left(10\right)
,&
3\mapsto \left(20\right)
,\\
4\mapsto \left(21\right)
,&
5\mapsto \left(22\right)
,&
6\mapsto \left(27\right)
,&
7\mapsto \left(1,\,0\right)
,\\
8\mapsto \left(6,\,5\right)
,&
9\mapsto \left(7,\,4\right)
,&
10\mapsto \left(8,\,4\right)
,&
11\mapsto \left(11,\,3\right)
,\\
12\mapsto \left(12,\,3\right)
,&
13\mapsto \left(16,\,15\right)
,&
14\mapsto \left(17,\,15\right)
,&
15\mapsto \left(18,\,14\right)
,\\
16\mapsto \left(19,\,14\right)
,&
17\mapsto \left(23,\,13\right)
,&
18\mapsto \left(26,\,25\right)
,&
19\mapsto \left(28,\,24\right)
.
\end{array}}
\newcommand\omegaVII{\arraycolsep=1.4pt
 \begin{array}{llll}
0\mapsto \left(8\right)
,&
1\mapsto \left(9\right)
,&
2\mapsto \left(10\right)
,&
3\mapsto \left(15\right)
,\\
4\mapsto \left(16\right)
,&
5\mapsto \left(18\right)
,&
6\mapsto \left(19\right)
,&
7\mapsto \left(7,\,0\right)
,\\
8\mapsto \left(7,\,2\right)
,&
9\mapsto \left(8,\,1\right)
,&
10\mapsto \left(11,\,2\right)
,&
11\mapsto \left(12,\,2\right)
,\\
12\mapsto \left(13,\,3\right)
,&
13\mapsto \left(14,\,3\right)
,&
14\mapsto \left(15,\,5\right)
,&
15\mapsto \left(15,\,6\right)
,\\
16\mapsto \left(16,\,5\right)
,&
17\mapsto \left(16,\,6\right)
,&
18\mapsto \left(17,\,4\right)
,&
19\mapsto \left(19,\,6\right)
.
\end{array}}
\newcommand\omegaVIII{\arraycolsep=1.4pt
 \begin{array}{llllll}
0\mapsto \left(3\right)
,&
1\mapsto \left(4\right)
,&
2\mapsto \left(5\right)
,&
3\mapsto \left(6\right)
,&
4\mapsto \left(12\right)
,&
5\mapsto \left(13\right)
,\\
6\mapsto \left(14\right)
,&
7\mapsto \left(15\right)
,&
8\mapsto \left(18\right)
,&
9\mapsto \left(\begin{array}{r}
0 \\
4
\end{array}\right)
,&
10\mapsto \left(\begin{array}{r}
0 \\
5
\end{array}\right)
,&
11\mapsto \left(\begin{array}{r}
1 \\
5
\end{array}\right)
,\\
12\mapsto \left(\begin{array}{r}
2 \\
5
\end{array}\right)
,&
13\mapsto \left(\begin{array}{r}
0 \\
6
\end{array}\right)
,&
14\mapsto \left(\begin{array}{r}
8 \\
13
\end{array}\right)
,&
15\mapsto \left(\begin{array}{r}
10 \\
14
\end{array}\right)
,&
16\mapsto \left(\begin{array}{r}
10 \\
15
\end{array}\right)
,&
17\mapsto \left(\begin{array}{r}
11 \\
16
\end{array}\right)
,\\
18\mapsto \left(\begin{array}{r}
9 \\
17
\end{array}\right)
,&
19\mapsto \left(\begin{array}{r}
11 \\
17
\end{array}\right)
,&
20\mapsto \left(\begin{array}{r}
7 \\
18
\end{array}\right)
,&
21\mapsto \left(\begin{array}{r}
9 \\
19
\end{array}\right)
.
\end{array}}
\newcommand\omegaIX{\arraycolsep=1.4pt
 \begin{array}{llll}
0\mapsto \left(8\right)
,&
1\mapsto \left(14\right)
,&
2\mapsto \left(17\right)
,&
3\mapsto \left(20\right)
,\\
4\mapsto \left(4,\,1\right)
,&
5\mapsto \left(5,\,1\right)
,&
6\mapsto \left(6,\,3\right)
,&
7\mapsto \left(7,\,2\right)
,\\
8\mapsto \left(8,\,0\right)
,&
9\mapsto \left(14,\,9\right)
,&
10\mapsto \left(15,\,13\right)
,&
11\mapsto \left(16,\,10\right)
,\\
12\mapsto \left(16,\,11\right)
,&
13\mapsto \left(17,\,13\right)
,&
14\mapsto \left(18,\,12\right)
,&
15\mapsto \left(19,\,10\right)
,\\
16\mapsto \left(19,\,11\right)
,&
17\mapsto \left(21,\,12\right)
.
\end{array}}
\newcommand\omegaX{\arraycolsep=1.4pt
 \begin{array}{llllll}
0\mapsto \left(1\right)
,&
1\mapsto \left(2\right)
,&
2\mapsto \left(3\right)
,&
3\mapsto \left(12\right)
,&
4\mapsto \left(13\right)
,&
5\mapsto \left(14\right)
,\\
6\mapsto \left(15\right)
,&
7\mapsto \left(16\right)
,&
8\mapsto \left(17\right)
,&
9\mapsto \left(\begin{array}{r}
0 \\
1
\end{array}\right)
,&
10\mapsto \left(\begin{array}{r}
0 \\
2
\end{array}\right)
,&
11\mapsto \left(\begin{array}{r}
0 \\
3
\end{array}\right)
,\\
12\mapsto \left(\begin{array}{r}
8 \\
9
\end{array}\right)
,&
13\mapsto \left(\begin{array}{r}
4 \\
10
\end{array}\right)
,&
14\mapsto \left(\begin{array}{r}
4 \\
11
\end{array}\right)
,&
15\mapsto \left(\begin{array}{r}
6 \\
12
\end{array}\right)
,&
16\mapsto \left(\begin{array}{r}
5 \\
13
\end{array}\right)
,&
17\mapsto \left(\begin{array}{r}
8 \\
13
\end{array}\right)
,\\
18\mapsto \left(\begin{array}{r}
7 \\
14
\end{array}\right)
,&
19\mapsto \left(\begin{array}{r}
5 \\
15
\end{array}\right)
,&
20\mapsto \left(\begin{array}{r}
7 \\
17
\end{array}\right)
.
\end{array}}
\newcommand\omegaXI{\arraycolsep=1.4pt
 \begin{array}{llll}
0\mapsto \left(5\right)
,&
1\mapsto \left(8\right)
,&
2\mapsto \left(14\right)
,&
3\mapsto \left(15\right)
,\\
4\mapsto \left(18\right)
,&
5\mapsto \left(20\right)
,&
6\mapsto \left(3,\,1\right)
,&
7\mapsto \left(4,\,2\right)
,\\
8\mapsto \left(5,\,1\right)
,&
9\mapsto \left(6,\,0\right)
,&
10\mapsto \left(7,\,1\right)
,&
11\mapsto \left(8,\,1\right)
,\\
12\mapsto \left(12,\,11\right)
,&
13\mapsto \left(13,\,11\right)
,&
14\mapsto \left(14,\,9\right)
,&
15\mapsto \left(15,\,10\right)
,\\
16\mapsto \left(16,\,11\right)
,&
17\mapsto \left(17,\,11\right)
,&
18\mapsto \left(19,\,9\right)
.
\end{array}}
\newcommand\omegaXII{\arraycolsep=1.4pt
 \begin{array}{lllll}
0\mapsto \left(8\right)
,&
1\mapsto \left(9\right)
,&
2\mapsto \left(11\right)
,&
3\mapsto \left(13\right)
,&
4\mapsto \left(14\right)
,\\
5\mapsto \left(15\right)
,&
6\mapsto \left(16\right)
,&
7\mapsto \left(17\right)
,&
8\mapsto \left(\begin{array}{r}
0 \\
8
\end{array}\right)
,&
9\mapsto \left(\begin{array}{r}
1 \\
9
\end{array}\right)
,\\
10\mapsto \left(\begin{array}{r}
1 \\
10
\end{array}\right)
,&
11\mapsto \left(\begin{array}{r}
1 \\
11
\end{array}\right)
,&
12\mapsto \left(\begin{array}{r}
6 \\
12
\end{array}\right)
,&
13\mapsto \left(\begin{array}{r}
4 \\
13
\end{array}\right)
,&
14\mapsto \left(\begin{array}{r}
7 \\
13
\end{array}\right)
,\\
15\mapsto \left(\begin{array}{r}
2 \\
14
\end{array}\right)
,&
16\mapsto \left(\begin{array}{r}
6 \\
14
\end{array}\right)
,&
17\mapsto \left(\begin{array}{r}
7 \\
15
\end{array}\right)
,&
18\mapsto \left(\begin{array}{r}
3 \\
16
\end{array}\right)
,&
19\mapsto \left(\begin{array}{r}
3 \\
17
\end{array}\right)
,\\
20\mapsto \left(\begin{array}{r}
5 \\
18
\end{array}\right)
.
\end{array}}
\newcommand\omegaXIII{\arraycolsep=1.4pt
 \begin{array}{lllll}
0\mapsto \left(6\right)
,&
1\mapsto \left(7\right)
,&
2\mapsto \left(15\right)
,&
3\mapsto \left(16\right)
,&
4\mapsto \left(18\right)
,\\
5\mapsto \left(19\right)
,&
6\mapsto \left(3,\,1\right)
,&
7\mapsto \left(4,\,0\right)
,&
8\mapsto \left(5,\,0\right)
,&
9\mapsto \left(5,\,2\right)
,\\
10\mapsto \left(6,\,0\right)
,&
11\mapsto \left(7,\,0\right)
,&
12\mapsto \left(12,\,9\right)
,&
13\mapsto \left(13,\,9\right)
,&
14\mapsto \left(14,\,9\right)
,\\
15\mapsto \left(15,\,8\right)
,&
16\mapsto \left(16,\,11\right)
,&
17\mapsto \left(17,\,11\right)
,&
18\mapsto \left(20,\,10\right)
.
\end{array}}
\newcommand\omegaUtoXII{\arraycolsep=1.4pt
 \begin{array}{lllll}
0\mapsto \left(0\right)
,&
1\mapsto \left(1\right)
,&
2\mapsto \left(9\right)
,&
3\mapsto \left(7\right)
,&
4\mapsto \left(8\right)
,\\
5\mapsto \left(11\right)
,&
6\mapsto \left(10\right)
,&
7\mapsto \left(6\right)
,&
8\mapsto \left(2\right)
,&
9\mapsto \left(4\right)
,\\
10\mapsto \left(5\right)
,&
11\mapsto \left(3\right)
,&
12\mapsto \left(18\right)
,&
13\mapsto \left(14\right)
,&
14\mapsto \left(16\right)
,\\
15\mapsto \left(13\right)
,&
16\mapsto \left(12\right)
,&
17\mapsto \left(17\right)
,&
18\mapsto \left(15\right)
.
\end{array}}
\newcommand\omegaUtoU{\arraycolsep=1.4pt
 \begin{array}{lllll}
0\mapsto \left(17\right)
,&
1\mapsto \left(16\right)
,&
2\mapsto \left(15,\,11\right)
,&
3\mapsto \left(13,\,9\right)
,&
4\mapsto \left(17,\,8\right)
,\\
5\mapsto \left(16,\,8\right)
,&
6\mapsto \left(15,\,8\right)
,&
7\mapsto \left(14,\,8\right)
,&
8\mapsto \left(\begin{array}{r}
6 \\
14
\end{array}\right)
,&
9\mapsto \left(\begin{array}{r}
3 \\
17
\end{array}\right)
,\\
10\mapsto \left(\begin{array}{r}
3 \\
16
\end{array}\right)
,&
11\mapsto \left(\begin{array}{r}
2 \\
14
\end{array}\right)
,&
12\mapsto \left(\begin{array}{rr}
7 & 1 \\
15 & 11
\end{array}\right)
,&
13\mapsto \left(\begin{array}{rr}
6 & 1 \\
14 & 11
\end{array}\right)
,&
14\mapsto \left(\begin{array}{rr}
7 & 1 \\
13 & 9
\end{array}\right)
,\\
15\mapsto \left(\begin{array}{rr}
6 & 1 \\
12 & 9
\end{array}\right)
,&
16\mapsto \left(\begin{array}{rr}
5 & 1 \\
18 & 10
\end{array}\right)
,&
17\mapsto \left(\begin{array}{rr}
4 & 1 \\
13 & 9
\end{array}\right)
,&
18\mapsto \left(\begin{array}{rr}
2 & 0 \\
14 & 8
\end{array}\right)
.
\end{array}}

%% file: article2_tiles_table.tex
\newcommand\TOtable{\begin{tabular}{lllll}
Id & Right & Top & Left & Bottom \\ \hline
$0$ & 2 & 4 & 2 & 1 \\
$1$ & 2 & 2 & 2 & 0 \\
$2$ & 1 & 1 & 3 & 1 \\
$3$ & 1 & 2 & 3 & 2 \\
$4$ & 3 & 1 & 3 & 3 \\
$5$ & 0 & 1 & 3 & 1 \\
$6$ & 0 & 0 & 0 & 1 \\
$7$ & 3 & 1 & 0 & 2 \\
$8$ & 0 & 2 & 1 & 2 \\
$9$ & 1 & 2 & 1 & 4 \\
$10$ & 3 & 3 & 1 & 2 \\
\end{tabular}}
\newcommand\TItable{\begin{tabular}{lllll}
Id & Right & Top & Left & Bottom \\ \hline
$0$ & 1 & 1 & 3 & 1 \\
$1$ & 1 & 2 & 3 & 2 \\
$2$ & 3 & 1 & 3 & 3 \\
$3$ & 0 & 1 & 3 & 1 \\
$4$ & 0 & 0 & 0 & 1 \\
$5$ & 3 & 1 & 0 & 2 \\
$6$ & 0 & 2 & 1 & 2 \\
$7$ & 3 & 3 & 1 & 2 \\
$8$ & 21 & 2 & 21 & 1 \\
$9$ & 21 & 2 & 23 & 0 \\
$10$ & 23 & 1 & 20 & 0 \\
$11$ & 20 & 2 & 21 & 0 \\
$12$ & 23 & 3 & 21 & 0 \\
\end{tabular}}
\newcommand\TIItable{\begin{tabular}{lllll}
Id & Right & Top & Left & Bottom \\ \hline
$0$ & 1 & 1 & 3 & 1 \\
$1$ & 1 & 2 & 3 & 2 \\
$2$ & 3 & 1 & 3 & 3 \\
$3$ & 0 & 1 & 3 & 1 \\
$4$ & 0 & 0 & 0 & 1 \\
$5$ & 3 & 1 & 0 & 2 \\
$6$ & 0 & 2 & 1 & 2 \\
$7$ & 3 & 3 & 1 & 2 \\
$8$ & 211 & 2 & 213 & 1 \\
$9$ & 213 & 1 & 210 & 1 \\
$10$ & 210 & 2 & 211 & 1 \\
$11$ & 213 & 3 & 211 & 1 \\
$12$ & 211 & 2 & 233 & 0 \\
$13$ & 213 & 1 & 230 & 0 \\
$14$ & 210 & 2 & 231 & 0 \\
$15$ & 231 & 1 & 203 & 0 \\
$16$ & 230 & 1 & 203 & 0 \\
$17$ & 203 & 1 & 210 & 0 \\
$18$ & 203 & 3 & 211 & 0 \\
$19$ & 233 & 1 & 213 & 0 \\
\end{tabular}}
\newcommand\TIIItable{\begin{tabular}{lllll}
Id & Right & Top & Left & Bottom \\ \hline
$0$ & 3 & 1 & 3 & 3 \\
$1$ & 0 & 1 & 3 & 1 \\
$2$ & 0 & 0 & 0 & 1 \\
$3$ & 3 & 1 & 0 & 2 \\
$4$ & 2111 & 2 & 2133 & 1 \\
$5$ & 2113 & 1 & 2130 & 1 \\
$6$ & 2110 & 2 & 2131 & 1 \\
$7$ & 2131 & 1 & 2103 & 1 \\
$8$ & 2130 & 1 & 2103 & 1 \\
$9$ & 2103 & 1 & 2110 & 1 \\
$10$ & 2103 & 3 & 2111 & 1 \\
$11$ & 2133 & 1 & 2113 & 1 \\
$12$ & 2113 & 1 & 2330 & 0 \\
$13$ & 2110 & 2 & 2331 & 0 \\
$14$ & 2130 & 0 & 2300 & 0 \\
$15$ & 2103 & 1 & 2310 & 0 \\
$16$ & 2310 & 1 & 2033 & 0 \\
$17$ & 2300 & 1 & 2033 & 0 \\
$18$ & 2300 & 0 & 2030 & 0 \\
$19$ & 2030 & 1 & 2103 & 0 \\
$20$ & 2033 & 1 & 2113 & 0 \\
$21$ & 2331 & 1 & 2133 & 0 \\
$22$ & 2330 & 1 & 2133 & 0 \\
$23$ & 2330 & 0 & 2130 & 0 \\
\end{tabular}}
\newcommand\TIVptable{\begin{tabular}{lllll}
Id & Right & Top & Left & Bottom \\ \hline
$0$ & 2113 & 1 & 2130 & 1 \\
$1$ & 2130 & 1 & 2103 & 1 \\
$2$ & 2133 & 1 & 2113 & 1 \\
$3$ & 2113 & 1 & 2330 & 0 \\
$4$ & 2130 & 0 & 2300 & 0 \\
$5$ & 2103 & 1 & 2310 & 0 \\
$6$ & 2310 & 1 & 2033 & 0 \\
$7$ & 2300 & 1 & 2033 & 0 \\
$8$ & 2300 & 0 & 2030 & 0 \\
$9$ & 2030 & 1 & 2103 & 0 \\
$10$ & 2033 & 1 & 2113 & 0 \\
$11$ & 2330 & 1 & 2133 & 0 \\
$12$ & 2330 & 0 & 2130 & 0 \\
$13$ & 21113 & 1 & 21330 & 1 \\
$14$ & 21130 & 0 & 21300 & 1 \\
$15$ & 21103 & 1 & 21310 & 1 \\
$16$ & 21310 & 1 & 21033 & 1 \\
$17$ & 21310 & 0 & 21030 & 1 \\
$18$ & 21300 & 1 & 21033 & 1 \\
$19$ & 21300 & 0 & 21030 & 1 \\
$20$ & 21030 & 1 & 21103 & 1 \\
$21$ & 21033 & 1 & 21113 & 1 \\
$22$ & 21330 & 0 & 21130 & 1 \\
$23$ & 21130 & 0 & 23300 & 0 \\
$24$ & 21103 & 1 & 23310 & 0 \\
$25$ & 21030 & 0 & 23100 & 0 \\
$26$ & 23100 & 0 & 20330 & 0 \\
$27$ & 20330 & 0 & 21130 & 0 \\
$28$ & 23310 & 0 & 21330 & 0 \\
$29$ & 23300 & 0 & 21330 & 0 \\
\end{tabular}}
\newcommand\TIVtable{\begin{tabular}{lllll}
Id & Right & Top & Left & Bottom \\ \hline
$0$ & 2113 & 1 & 2130 & 1 \\
$1$ & 2130 & 1 & 2103 & 1 \\
$2$ & 2133 & 1 & 2113 & 1 \\
$3$ & 2113 & 1 & 2330 & 0 \\
$4$ & 2130 & 0 & 2300 & 0 \\
$5$ & 2103 & 1 & 2310 & 0 \\
$6$ & 2310 & 1 & 2033 & 0 \\
$7$ & 2300 & 1 & 2033 & 0 \\
$8$ & 2300 & 0 & 2030 & 0 \\
$9$ & 2030 & 1 & 2103 & 0 \\
$10$ & 2033 & 1 & 2113 & 0 \\
$11$ & 2330 & 1 & 2133 & 0 \\
$12$ & 2330 & 0 & 2130 & 0 \\
$13$ & 21113 & 1 & 21330 & 1 \\
$14$ & 21130 & 0 & 21300 & 1 \\
$15$ & 21103 & 1 & 21310 & 1 \\
$16$ & 21310 & 1 & 21033 & 1 \\
$17$ & 21310 & 0 & 21030 & 1 \\
$18$ & 21300 & 1 & 21033 & 1 \\
$19$ & 21300 & 0 & 21030 & 1 \\
$20$ & 21030 & 1 & 21103 & 1 \\
$21$ & 21033 & 1 & 21113 & 1 \\
$22$ & 21330 & 0 & 21130 & 1 \\
$23$ & 21130 & 0 & 23300 & 0 \\
$24$ & 21030 & 0 & 23100 & 0 \\
$25$ & 23100 & 0 & 20330 & 0 \\
$26$ & 20330 & 0 & 21130 & 0 \\
$27$ & 23300 & 0 & 21330 & 0 \\
\end{tabular}}
\newcommand\TVtable{\begin{tabular}{lllll}
Id & Right & Top & Left & Bottom \\ \hline
$0$ & 2113 & 5 & 2130 & 1 \\
$1$ & 2130 & 1 & 2103 & 5 \\
$2$ & 2133 & 1 & 2113 & 1 \\
$3$ & 2113 & 5 & 2330 & 0 \\
$4$ & 2130 & 6 & 2300 & 0 \\
$5$ & 2103 & 5 & 2310 & 0 \\
$6$ & 2310 & 1 & 2033 & 6 \\
$7$ & 2300 & 1 & 2033 & 6 \\
$8$ & 2300 & 0 & 2030 & 6 \\
$9$ & 2030 & 1 & 2103 & 0 \\
$10$ & 2033 & 1 & 2113 & 0 \\
$11$ & 2330 & 1 & 2133 & 6 \\
$12$ & 2330 & 0 & 2130 & 6 \\
$13$ & 21113 & 5 & 21330 & 1 \\
$14$ & 21130 & 6 & 21300 & 1 \\
$15$ & 21103 & 5 & 21310 & 1 \\
$16$ & 21310 & 1 & 21033 & 5 \\
$17$ & 21310 & 0 & 21030 & 5 \\
$18$ & 21300 & 1 & 21033 & 5 \\
$19$ & 21300 & 0 & 21030 & 5 \\
$20$ & 21030 & 1 & 21103 & 1 \\
$21$ & 21033 & 1 & 21113 & 1 \\
$22$ & 21330 & 0 & 21130 & 1 \\
$23$ & 21330 & 0 & 21130 & 5 \\
$24$ & 21130 & 6 & 23300 & 0 \\
$25$ & 21030 & 6 & 23100 & 0 \\
$26$ & 23100 & 0 & 20330 & 6 \\
$27$ & 20330 & 0 & 21130 & 0 \\
$28$ & 23300 & 0 & 21330 & 6 \\
\end{tabular}}
\newcommand\TVItable{\begin{tabular}{lllll}
Id & Right & Top & Left & Bottom \\ \hline
$0$ & 21131 & 1 & 21305 & 1 \\
$1$ & 21305 & 5 & 21031 & 5 \\
$2$ & 21331 & 1 & 21131 & 1 \\
$3$ & 21131 & 1 & 23305 & 0 \\
$4$ & 21300 & 0 & 23006 & 0 \\
$5$ & 21031 & 1 & 23105 & 0 \\
$6$ & 23105 & 5 & 20331 & 6 \\
$7$ & 23006 & 6 & 20331 & 6 \\
$8$ & 23006 & 6 & 20300 & 6 \\
$9$ & 20300 & 0 & 21031 & 0 \\
$10$ & 20331 & 1 & 21131 & 0 \\
$11$ & 23305 & 5 & 21331 & 6 \\
$12$ & 23305 & 5 & 21300 & 6 \\
$13$ & 211131 & 1 & 213305 & 1 \\
$14$ & 211300 & 0 & 213006 & 1 \\
$15$ & 211031 & 1 & 213105 & 1 \\
$16$ & 213105 & 5 & 210331 & 5 \\
$17$ & 213105 & 5 & 210300 & 5 \\
$18$ & 213006 & 6 & 210331 & 5 \\
$19$ & 213006 & 6 & 210300 & 5 \\
$20$ & 210300 & 0 & 211031 & 1 \\
$21$ & 210331 & 1 & 211131 & 1 \\
$22$ & 213300 & 0 & 211300 & 1 \\
$23$ & 213305 & 5 & 211300 & 5 \\
$24$ & 211300 & 0 & 233006 & 0 \\
$25$ & 210300 & 0 & 231006 & 0 \\
$26$ & 231006 & 6 & 203300 & 6 \\
$27$ & 203300 & 0 & 211300 & 0 \\
$28$ & 233006 & 6 & 213300 & 6 \\
\end{tabular}}
\newcommand\TVIItable{\begin{tabular}{lllll}
Id & Right & Top & Left & Bottom \\ \hline
$0$ & 21331 & 1 & 21131 & 1 \\
$1$ & 20300 & 0 & 21031 & 0 \\
$2$ & 20331 & 1 & 21131 & 0 \\
$3$ & 210300 & 0 & 211031 & 1 \\
$4$ & 210331 & 1 & 211131 & 1 \\
$5$ & 213300 & 0 & 211300 & 1 \\
$6$ & 203300 & 0 & 211300 & 0 \\
$7$ & 21131 & 51 & 21031 & 51 \\
$8$ & 21031 & 51 & 20331 & 60 \\
$9$ & 21300 & 60 & 20331 & 60 \\
$10$ & 21300 & 60 & 20300 & 60 \\
$11$ & 21131 & 51 & 21331 & 60 \\
$12$ & 21131 & 51 & 21300 & 60 \\
$13$ & 211031 & 51 & 210331 & 51 \\
$14$ & 211031 & 51 & 210300 & 51 \\
$15$ & 211300 & 60 & 210331 & 51 \\
$16$ & 211300 & 60 & 210300 & 51 \\
$17$ & 211131 & 51 & 211300 & 51 \\
$18$ & 210300 & 60 & 203300 & 60 \\
$19$ & 211300 & 60 & 213300 & 60 \\
\end{tabular}}
\newcommand\TVIIItable{\begin{tabular}{lllll}
Id & Right & Top & Left & Bottom \\ \hline
$0$ & 21031 & 51 & 20331 & 60 \\
$1$ & 21300 & 60 & 20331 & 60 \\
$2$ & 21300 & 60 & 20300 & 60 \\
$3$ & 211300 & 60 & 210331 & 51 \\
$4$ & 211300 & 60 & 210300 & 51 \\
$5$ & 210300 & 60 & 203300 & 60 \\
$6$ & 211300 & 60 & 213300 & 60 \\
$7$ & 21331 & 511 & 21031 & 511 \\
$8$ & 20331 & 511 & 21031 & 510 \\
$9$ & 20300 & 510 & 20331 & 600 \\
$10$ & 20331 & 511 & 21331 & 600 \\
$11$ & 20331 & 511 & 21300 & 600 \\
$12$ & 210300 & 510 & 210331 & 511 \\
$13$ & 210300 & 510 & 210300 & 511 \\
$14$ & 213300 & 600 & 210331 & 511 \\
$15$ & 203300 & 600 & 210331 & 510 \\
$16$ & 213300 & 600 & 210300 & 511 \\
$17$ & 203300 & 600 & 210300 & 510 \\
$18$ & 210331 & 511 & 211300 & 511 \\
$19$ & 203300 & 600 & 213300 & 600 \\
\end{tabular}}
\newcommand\TIXtable{\begin{tabular}{lllll}
Id & Right & Top & Left & Bottom \\ \hline
$0$ & 211300 & 60 & 210331 & 51 \\
$1$ & 211300 & 60 & 210300 & 51 \\
$2$ & 210300 & 60 & 203300 & 60 \\
$3$ & 211300 & 60 & 213300 & 60 \\
$4$ & 210300 & 510 & 210331 & 511 \\
$5$ & 210300 & 510 & 210300 & 511 \\
$6$ & 213300 & 600 & 210331 & 511 \\
$7$ & 203300 & 600 & 210331 & 510 \\
$8$ & 210331 & 511 & 211300 & 511 \\
$9$ & 21130021031 & 51 & 21030020331 & 51 \\
$10$ & 21030021031 & 51 & 20330020331 & 60 \\
$11$ & 21030021300 & 60 & 20330020331 & 60 \\
$12$ & 21030021300 & 60 & 20330020300 & 60 \\
$13$ & 21130021031 & 51 & 21330020331 & 60 \\
$14$ & 21030020331 & 511 & 21030021031 & 511 \\
$15$ & 21330020331 & 511 & 21033121331 & 511 \\
$16$ & 20330020331 & 511 & 21033121331 & 510 \\
$17$ & 21330020331 & 511 & 21030021300 & 511 \\
$18$ & 20330020300 & 510 & 21030020331 & 510 \\
$19$ & 20330020331 & 511 & 21030021300 & 510 \\
$20$ & 21033121331 & 511 & 21130021031 & 511 \\
$21$ & 20330020300 & 510 & 21330020331 & 600 \\
\end{tabular}}
\newcommand\TXtable{\begin{tabular}{lllll}
Id & Right & Top & Left & Bottom \\ \hline
$0$ & 210331 & 511 & 211300 & 511 \\
$1$ & 21030020331 & 511 & 21030021031 & 511 \\
$2$ & 21330020331 & 511 & 21030021300 & 511 \\
$3$ & 21033121331 & 511 & 21130021031 & 511 \\
$4$ & 211300 & 51060 & 210331 & 51151 \\
$5$ & 211300 & 51060 & 210300 & 51151 \\
$6$ & 211300 & 60060 & 210331 & 51160 \\
$7$ & 210300 & 60060 & 210331 & 51060 \\
$8$ & 211300 & 51160 & 211300 & 51151 \\
$9$ & 21130021031 & 51151 & 21030021031 & 51151 \\
$10$ & 21130021031 & 51151 & 21033121331 & 51160 \\
$11$ & 21030021031 & 51151 & 21033121331 & 51060 \\
$12$ & 21030021300 & 51160 & 21033121331 & 51060 \\
$13$ & 21130021031 & 51151 & 21030021300 & 51160 \\
$14$ & 21030021300 & 51060 & 21030020331 & 51060 \\
$15$ & 21030021031 & 51151 & 21030021300 & 51060 \\
$16$ & 21030021300 & 51160 & 21030021300 & 51060 \\
$17$ & 21030021300 & 51060 & 21330020331 & 60060 \\
\end{tabular}}
\newcommand\TXItable{\begin{tabular}{lllll}
Id & Right & Top & Left & Bottom \\ \hline
$0$ & 21030020331 & 511 & 21030021031 & 511 \\
$1$ & 21330020331 & 511 & 21030021300 & 511 \\
$2$ & 21033121331 & 511 & 21130021031 & 511 \\
$3$ & 21030021300 & 51160 & 21033121331 & 51060 \\
$4$ & 21130021031 & 51151 & 21030021300 & 51160 \\
$5$ & 21030021300 & 51060 & 21030020331 & 51060 \\
$6$ & 21030021031 & 51151 & 21030021300 & 51060 \\
$7$ & 21030021300 & 51160 & 21030021300 & 51060 \\
$8$ & 21030021300 & 51060 & 21330020331 & 60060 \\
$9$ & 21030020331210331 & 511 & 21030021031211300 & 511 \\
$10$ & 21330020331210331 & 511 & 21030021300211300 & 511 \\
$11$ & 21033121331210331 & 511 & 21130021031211300 & 511 \\
$12$ & 21130021031211300 & 51160 & 21030021031211300 & 51151 \\
$13$ & 21130021031211300 & 51060 & 21033121331210331 & 51160 \\
$14$ & 21030021031211300 & 51060 & 21033121331210331 & 51060 \\
$15$ & 21030021300211300 & 60060 & 21033121331210331 & 51060 \\
$16$ & 21130021031211300 & 51060 & 21030021300210300 & 51160 \\
$17$ & 21130021031211300 & 51160 & 21030021300211300 & 51160 \\
$18$ & 21030021300210300 & 60060 & 21030020331210331 & 51060 \\
$19$ & 21030021031211300 & 51060 & 21030021300210300 & 51060 \\
$20$ & 21030021300210300 & 60060 & 21330020331210331 & 60060 \\
\end{tabular}}
\newcommand\TXIItable{\begin{tabular}{lllll}
Id & Right & Top & Left & Bottom \\ \hline
$0$ & 21030021300 & 51060 & 21030020331 & 51060 \\
$1$ & 21030021300 & 51060 & 21330020331 & 60060 \\
$2$ & 21030021031211300 & 51060 & 21033121331210331 & 51060 \\
$3$ & 21030021300211300 & 60060 & 21033121331210331 & 51060 \\
$4$ & 21030021300210300 & 60060 & 21030020331210331 & 51060 \\
$5$ & 21030021300210300 & 60060 & 21330020331210331 & 60060 \\
$6$ & 21330020331 & 51160511 & 21033121331 & 51060511 \\
$7$ & 21033121331 & 51151511 & 21030021300 & 51160511 \\
$8$ & 21330020331 & 51060511 & 21030020331 & 51060511 \\
$9$ & 21030020331 & 51151511 & 21030021300 & 51060511 \\
$10$ & 21330020331 & 51160511 & 21030021300 & 51060511 \\
$11$ & 21330020331 & 51060511 & 21330020331 & 60060511 \\
$12$ & 21033121331210331 & 51160511 & 21030021031211300 & 51151511 \\
$13$ & 21033121331210331 & 51060511 & 21033121331210331 & 51160511 \\
$14$ & 21030020331210331 & 51060511 & 21033121331210331 & 51060511 \\
$15$ & 21330020331210331 & 60060511 & 21033121331210331 & 51060511 \\
$16$ & 21033121331210331 & 51060511 & 21030021300210300 & 51160511 \\
$17$ & 21033121331210331 & 51160511 & 21030021300211300 & 51160511 \\
$18$ & 21030020331210331 & 51060511 & 21030021300210300 & 51060511 \\
\end{tabular}}
\newcommand\TXIIItable{\begin{tabular}{lllll}
Id & Right & Top & Left & Bottom \\ \hline
$0$ & B & O & I & O \\
$1$ & G & L & E & O \\
$2$ & A & L & I & O \\
$3$ & E & P & I & P \\
$4$ & I & P & G & K \\
$5$ & I & P & I & K \\
$6$ & I & K & B & M \\
$7$ & I & K & A & K \\
$8$ & BF & O & IJ & O \\
$9$ & GF & O & EH & O \\
$10$ & GF & O & CH & L \\
$11$ & AF & O & IH & O \\
$12$ & EH & K & GF & P \\
$13$ & EH & P & IJ & P \\
$14$ & EH & K & ID & P \\
$15$ & IJ & M & GF & K \\
$16$ & IH & K & GF & K \\
$17$ & IH & K & ID & K \\
$18$ & ID & M & BF & M \\
$19$ & ID & M & AF & K \\
$20$ & CH & P & IH & P \\
\end{tabular}}
\newcommand\TXIVtable{\begin{tabular}{lllll}
Id & Right & Top & Left & Bottom \\ \hline
$0$ & I & K & B & M \\
$1$ & I & K & A & K \\
$2$ & IJ & M & GF & K \\
$3$ & IH & K & GF & K \\
$4$ & ID & M & BF & M \\
$5$ & ID & M & AF & K \\
$6$ & G & PL & I & PO \\
$7$ & B & PO & G & KO \\
$8$ & B & PO & I & KO \\
$9$ & A & PL & I & KO \\
$10$ & B & KO & B & MO \\
$11$ & B & KO & A & KO \\
$12$ & GF & KO & GF & PO \\
$13$ & GF & PO & IJ & PO \\
$14$ & GF & KO & ID & PO \\
$15$ & BF & MO & GF & KO \\
$16$ & AF & KO & GF & KO \\
$17$ & AF & KO & ID & KO \\
$18$ & GF & PO & IH & PL \\
\end{tabular}}
\newcommand\TXIItablewithU{\begin{tabular}{lllll}
Id & Right & Top & Left & Bottom \\ \hline
$0$ & F=21030021300 & O=51060 & J=21030020331 & O=51060 \\
$1$ & F=21030021300 & O=51060 & H=21330020331 & L=60060 \\
$2$ & B=21030021031211300 & O=51060 & I=21033121331210331 & O=51060 \\
$3$ & A=21030021300211300 & L=60060 & I=21033121331210331 & O=51060 \\
$4$ & G=21030021300210300 & L=60060 & E=21030020331210331 & O=51060 \\
$5$ & G=21030021300210300 & L=60060 & C=21330020331210331 & L=60060 \\
$6$ & H=21330020331 & K=51160511 & D=21033121331 & P=51060511 \\
$7$ & D=21033121331 & M=51151511 & F=21030021300 & K=51160511 \\
$8$ & H=21330020331 & P=51060511 & J=21030020331 & P=51060511 \\
$9$ & J=21030020331 & M=51151511 & F=21030021300 & P=51060511 \\
$10$ & H=21330020331 & K=51160511 & F=21030021300 & P=51060511 \\
$11$ & H=21330020331 & P=51060511 & H=21330020331 & N=60060511 \\
$12$ & I=21033121331210331 & K=51160511 & B=21030021031211300 & M=51151511 \\
$13$ & I=21033121331210331 & P=51060511 & I=21033121331210331 & K=51160511 \\
$14$ & E=21030020331210331 & P=51060511 & I=21033121331210331 & P=51060511 \\
$15$ & C=21330020331210331 & N=60060511 & I=21033121331210331 & P=51060511 \\
$16$ & I=21033121331210331 & P=51060511 & G=21030021300210300 & K=51160511 \\
$17$ & I=21033121331210331 & K=51160511 & A=21030021300211300 & K=51160511 \\
$18$ & E=21030020331210331 & P=51060511 & G=21030021300210300 & P=51060511 \\
\end{tabular}}

%% file: article2_markers.tex
\newcommand\markersO{\left\{0, 1\right\}}
\newcommand\markersI{\left\{8, 9, 10, 11, 12\right\}}
\newcommand\markersII{\left\{8, 9, 10, 11, 12, 13, 14, 15, 16, 17, 18, 19\right\}}
\newcommand\markersIII{\left\{0, 1, 2, 3\right\}}
\newcommand\markersVI{\left\{1, 6, 7, 8, 11, 12, 16, 17, 18, 19, 23, 26, 28\right\}}
\newcommand\markersVII{\left\{0, 1, 2, 3, 4, 5, 6\right\}}
\newcommand\markersVIII{\left\{0, 1, 2, 7, 8, 9, 10, 11\right\}}
\newcommand\markersIX{\left\{0, 1, 2, 3, 9, 10, 11, 12, 13\right\}}
\newcommand\markersX{\left\{0, 4, 5, 6, 7, 8\right\}}
\newcommand\markersXI{\left\{0, 1, 2, 9, 10, 11\right\}}
\newcommand\markersXII{\left\{0, 1, 2, 3, 4, 5, 6, 7\right\}}
\newcommand\markersXIII{\left\{0, 1, 2, 8, 9, 10, 11\right\}}